\documentclass[a4paper]{amsart}

\usepackage{amsmath, amssymb, amsthm}
\usepackage{cite}
\usepackage{graphicx}
\graphicspath{{/images/}}
\usepackage{pinlabel}
\usepackage{enumitem}
\usepackage{verbatim}
\usepackage{epigraph}
\usepackage{longtable}

\def\cydot{\leavevmode\raise.4ex\hbox{.}} 

\newcommand{\norm}[1]{\lVert#1\rVert}

\newcommand{\SLnR}{\mathrm{SL}_{n} (\mathbb{R})}

\newcommand{\SLnZ}{\mathrm{SL}_{n} (\mathbb{Z})}
\newcommand{\SLnQ}{\mathrm{SL}_{n} (\mathbb{Q})}

\newcommand{\R}{\mathbb{R}}
\newcommand{\Z}{\mathbb{Z}}
\newcommand{\N}{\mathbb{N}}

\newcommand{\Q}{\mathbb{Q}}

\newcommand{\CAT}[1]{\mathrm{CAT}(#1)}

\newcommand{\bdry}[1]{\partial_{\infty}#1}
\newcommand{\cone}[3]{\mathrm{Cone}_{\omega}(#1,#2,#3)}

\newcommand{\Lambdaplus}{{\Lambda}^{+}}

\newcommand{\ad}{\mathrm{ad}}
\newcommand{\Lie}[1]{\mathfrak{#1}}
\newcommand{\Dmod}{\Delta_{\mathrm{mod}}}

\newcommand{\hgt}{\mathrm{ht}}

\newcommand{\MIN}[1]{\mathrm{MIN}(#1)}

\newcommand{\IPp}{\varphi_p}
\newcommand{\Stab}{\mathrm{Stab}}

\newcommand{\modulus}[1]{\left| #1 \right|}

\newcommand{\CLF}{\mathrm{CLF}}
\newcommand{\rrank}{\mathrm{rank}_\R}
\newcommand{\qrank}{\mathrm{rank}_\Q}

\newcommand{\Out}{\mathrm{Out}}

\newtheorem{thm}{Theorem}[section]
\newtheorem{cor}[thm]{Corollary}
\newtheorem{lemma}[thm]{Lemma}
\newtheorem{prop}[thm]{Proposition}

\newtheorem{thmspecial}{Theorem}

\newtheoremstyle{citing}
   {3pt}
   {3pt}
   {\itshape}
   {}
   {\bfseries}
   {}
   {.5em}
   {\thmnote{#3}}

 \theoremstyle{citing}
 \newtheorem*{varthm}{}

\newcommand{\scheading}[1]{\vspace{2mm}

\noindent\textsc{#1:}
\vspace{2mm}

\noindent \hspace{-3.3pt}}

\newenvironment{remark}[1][Remark]{\begin{trivlist}
\item[\hskip \labelsep {\bf #1}:]}{\end{trivlist}}

\newenvironment{question}[1][Question]{\begin{trivlist}
\item[\hskip \labelsep {\bf #1}:]}{\end{trivlist}}

\newcounter{algorithm}
\newenvironment{algorithm}[1][Algorithm]{\refstepcounter{algorithm}\begin{trivlist}
\item[\hskip \labelsep {\bfseries #1 \Alph{algorithm}.}]} {\end{trivlist}}

\title[Bounded Conjugators in Semisimple Lie Groups]{Bounded Conjugators For Real Hyperbolic and Unipotent Elements in Semisimple Lie Groups}
\author{Andrew Sale} \email{andrew.sale@some.oxon.org}
\date{}
\thanks{The author was supported by the EPSRC}

\begin{document}

\begin{abstract}
Let $G$ be a real semisimple Lie group with trivial centre and no compact factors. Given a conjugate pair of either real hyperbolic elements or unipotent elements $a$ and $b$ in $G$ we find a conjugating element $g \in G$ such that $d_G(1,g) \leq L(d_G(1,u)+d_G(1,v))$, where $L$ is a positive constant which will depend on some property of $a$ and $b$. For the vast majority of such elements however, $L$ can be assumed to be a uniform constant.
\end{abstract}

\maketitle

\setcounter{tocdepth}{1}
\tableofcontents

The objective of this paper is to present results concerning an effective version of the conjugacy problem in the setting of semisimple Lie groups. We focus on finding short conjugators between real hyperbolic elements and unipotent elements.

The conjugacy problem is one of Max Dehn's three decision problems in group theory, which he set out in 1912, motivated by questions in low-dimensional manifolds. The other problems are the word problem and the isomorphism problem. These three problems are fundamental in realms of combinatorial and geometric group theory and have received much attention over the last century. Dehn originally described these problems in group theory because of the significance he discovered they had in the geometry of $3$--manifolds. He observed the interplay that occurs between the fundamental group of the manifold and its geometry. For example, the conjugacy problem in the fundamental group is equivalent to determining when two loops in the manifold are freely homotopic.

Let $\Gamma$ be a recursively presented group with finite symmetric generating set $A$. The \emph{word problem} on $\Gamma$ asks whether there is an algorithm which determines when any given word on the generating set $A$ represents the identity element of $\Gamma$. Associated to the word problem is the Dehn function, which measures its geometric complexity. It is a measure of the minimal area required to fill a loop in the Cayley $2$--complex of $G$. Because of this geometric interpretation, determining the Dehn function of groups has been a fundamental question in geometric group theory over the last couple of decades. The extra information the Dehn function provides means that we could describe calculating it as an \emph{effective} version of the word problem.

The \emph{conjugacy problem} is of a similar flavour to the word problem. We say the conjugacy problem in $\Gamma$ is solvable if there is an algorithm which, on input two words $u$ and $v$ on the generating set $A$, determines whether $u$ and $v$ represent conjugate elements in $\Gamma$.

\subsection*{An effective version of the conjugacy problem} Estimating the length of short conjugating elements in a group could be described as an \emph{effective} version of the conjugacy problem. Suppose a group $G$ admits a left-invariant metric $d_G$. For $g \in G$ let $\modulus{g}$ denote $d_G(1,g)$. The \emph{conjugacy length function} is the minimal function $\CLF_G : \R_{\geq 0} \rightarrow \R_{\geq 0}$ which satisfies the following: for $x \in \R_{\geq 0}$, if $u$ is conjugate to $v$ in $G$ and $\modulus{u}+\modulus{v}\leq x$ then there exists a conjugator $g \in G$ such that $\modulus{g} \leq \CLF_G(x)$. One can define it more concretely to be the function which sends $x \in \R_{\geq 0}$ to
				$$\sup \big\{ \inf \{ \modulus{g} : gu=vg \} : \modulus{u} + \modulus{v} \leq x \ \textrm{and $u$ is conjugate to $v$ in $G$} \big\}\textrm{.}$$ 
The question of determining conjugacy length functions has been addressed previously. See for example \cite{Sale12} and \cite{Sale12groupext} for results concerning groups including free solvable groups, wreath products, group extensions and abelian-by-cyclic groups. We know that the conjugacy length function is linear for groups including hyperbolic groups \cite{BH99}, Right-angled Artin groups \cite{CGW09} and Mapping class groups \cite{MM00}, \cite{BD11}, \cite{Tao11}. For $\CAT{0}$--groups and biautomatic groups all we know is that it is at most exponential (see \cite{BH99}), it is an open question as to whether this bound is sharp and indeed we do not even know if it is not necessarily linear. In their work on the stronger $\ell^1$--Bass conjecture, Ji, Ogle and Ramsey show $2$--step nilpotent groups have a quadratic conjugacy length function \cite{JOR10}, and also obtain a result for relatively hyperbolic groups. The fundamental group of a prime $3$--manifold also has a quadratic upper bound \cite{BD11}, \cite{Sale12groupext}.

The reader should note that, unlike the conjugacy problem itself, in order to define the conjugacy length function the only requirement on $\Gamma$ is that it should admit a left-invariant metric. In particular this means that we can define it for Lie groups.

In this paper we take the first steps towards understanding the conjugacy length function of higher-rank real semisimple Lie groups and their lattices by studying the conjugacy of real hyperbolic elements and unipotent elements.

Grunewald and Segal solved the conjugacy problem in arithmetic groups \cite{GS80} and hence, by Margulis arithmeticity, we know that every lattice in a higher-rank real semisimple Lie group has solvable conjugacy problem. However, as discussed in \cite{GI05}, their solution gives no insight into the length of the conjugating element. Determining a control on the lengths of short conjugators in lattices would not only be interesting in its own right, but would also provide us with a method of proving the solubility of the conjugacy problem via a method which does not rely on the power of Margulis arithmeticity.

\subsection*{Real hyperbolic elements}  Analogies have frequently been drawn between lattices in higher-rank semisimple Lie groups and mapping class groups. The pseudo-Anosov elements of a mapping class group are the elements which behave in a similar way to the real hyperbolic elements. Recently J. Tao \cite{Tao11} showed that mapping class groups have linear conjugacy length functions, however earlier work of Masur and Minsky \cite{MM00} showed that there is a linear bound on the length of short conjugators between a pair of pseudo-Anosov elements. Hence, following the analogy through to lattices, it is natural to begin approaching this problem by studying the real hyperbolic elements. The analogy also carries through to $\Out(F_n)$, the outer automorphism group of a free group $F_n$. Here we do not yet know if the conjugacy problem is solvable, however Lustig \cite{Lust07} has shown that it is solvable when we restrict to iwip elements, which are the elements analogous to the pseudo-Anosov and hyperbolic elements.

Theorem \ref{thmspecial:bounded conj in G} below describes the main result of this paper for real hyperbolic elements. Let $G$ be a real semisimple Lie group with trivial centre and no compact factors. An element $a \in G$ is said to be real hyperbolic if it translates some geodesic in the associated symmetric space $X$ and furthermore it also translates every other geodesic parallel to the first. The slope of $a$ describes how these geodesics sit inside the Weyl chambers of $X$. We formalise these definitions in Section \ref{sec:preliminaries}.

\begin{thmspecial}\label{thmspecial:bounded conj in G}
For each slope $\xi$ there exist positive constants $\ell_\xi$ and $d_\xi$ such that if $a$ and $b$ are conjugate real hyperbolic elements in $G$ with slope $\xi$ and such that $d_X(p,ap),d_X(p,bp)\geq d_\xi$ then there exists a conjugator $g \in G$ satisfying:
$$d_X(p,gp) \leq 2\ell_\xi \big(d_X(p,ap)+d_X(p,bp)\big)\textrm{.}$$
\end{thmspecial}

This leads naturally to the following result for lattices:

\begin{thmspecial}\label{thmspecial:bounded conj in G from lattice}
Let $\Gamma$ be a lattice in $G$. Then for each slope $\xi$, there exists a constant $L_\xi$ such that two elements $a,b \in \Gamma$ are conjugate in $G$ if and only if there exists a conjugator $g \in G$ such that
		$$d_X(p,gp) \leq L_\xi \big(d_X(p,ap)+d_X(p,bp)\big)\textrm{.}$$
\end{thmspecial}

The main tool in proving Theorem \ref{thmspecial:bounded conj in G} concerns an estimate of the distance from an arbitrary basepoint in $X$ to a flat (or union of flats):

\begin{varthm}[Lemma \ref{lemma:distance to MIN in symmetric space}] 
For each slope $\xi$ there exist positive constants $\ell_{\xi}$ and $d_{\xi}$ such that for each $a \in G$ which is real hyperbolic of slope $\xi$ and such that $d_{X}(p,ap)>d_{\xi}$ the following holds:
$$d_{X}(p, \MIN{a}) \leq 2\ell_{\xi}d_{X}(p,ap)\textrm{.}$$
\end{varthm}

\subsection*{Unipotent elements} In the second half of this paper we present a method for dealing with the conjugacy of certain unipotent elements. Crucial to our method is the root-space decomposition of the Lie algebra $\Lie{g}$ of $G$. This gives us a root system $\Lambda$ such that a maximal unipotent subgroup $N$ will have a Lie algbera $\Lie{n}$ of the form
		$$\Lie{n}=\sum_{\lambda \in \Lambdaplus} \Lie{g}_{\lambda}$$
where $\Lie{g}_\lambda$ is a root-space of $\Lie{g}$ and $\Lambdaplus$ is a subset of positive roots in $\Lambda$. In the following we assume that $\Lie{g}$ is a split Lie algebra, meaning that each root-space $\Lie{g}_\lambda$ has dimension $1$. The classical split Lie algebras are $\Lie{sl}_d(\R)$, $\Lie{sl}_{d,d+1}(\R)$, $\Lie{sp}_d(\R)$ and $\Lie{so}_{d,d}(\R)$. Because the exponential map restricted to $\Lie{n}$ is a diffeomorphism, each element $u$ of $N$ can be expressed uniquely as
		$$u = \exp \left( \sum_{\lambda \in \Lambdaplus} Y_\lambda\right) , \ \ Y_\lambda \in \Lie{g}_\lambda.$$
If $\Pi \subset \Lambdaplus$ is the set of simple roots, then we say that the \emph{simple entries} of $u$ are those $Y_\lambda$ for $\lambda \in \Pi$.

\begin{thmspecial}\label{thmspecial:unipotent clf}
Fix $\delta>0$. Let $u$ and $v$ be conjugate unipotent elements in $G$ such that each simple entry of $u$ and $v$ is of size at least $\delta$. Then there exists $g \in G$ such that $gug^{-1}=v$ and which satisfies:
	$$d_G(1,g) \leq L( d_G(1,u) + d_G(1,v))$$
where $L$ depends on $\delta$ and the root-system $\Lambda$ associated to $G$.
\end{thmspecial}

As with the real hyperbolic case we are able to deduce a corollary for lattices:

\begin{thmspecial}\label{thmspecial:unipotent in lattice}
Let $\Gamma$ be a lattice in $G$. Then there exists a constant $L>0$ such that two unipotent elements $u$ and $v$ in $\Gamma$ with non-zero simple entries are conjugate in $G$ if and only if there exists a conjugator $g \in G$ such that
		$$d_G(1,g) \leq L \big( d_G(1,u)+d_G(1,v)\big).$$
\end{thmspecial}

From here, there are several natural directions one could look. Firstly one could try to remove the restriction that the Lie algebra should be split. Secondly, it seems that the techniques to prove Theorem \ref{thmspecial:unipotent clf} could be extended to a larger family of unipotent elements in $G$, allowing the simple entries to be zero. In particular this would potentially allow the result for lattices to include all unipotent elements. However the computational aspect involved in doing this increased many times over from the ``simple case'' considered in this paper.

The next step should be to consider elliptic elements and then the idea would be to use the (complete) Jordan decomposition, which expresses each element of the group as a product of commuting elliptic, real hyperbolic and unipotent elements.

\begin{question}
Given the Jordan decompositions of two conjugate elements $u$ and $v$ of $G$, what can we infer about the length of a short conjugator between $u$ and $v$ from the lengths of short conjugators between each of the Jordan components?
\end{question}

The centralisers of each component will play an important role in answering this question.

The nature of Theorems \ref{thmspecial:bounded conj in G from lattice} and \ref{thmspecial:unipotent in lattice} raise further questions for lattices. In particular we would like the conjugating element to come from the lattice, rather than the Lie group as we have here.

\begin{question}
Consider two elements $u$ and $v$ in $\Gamma$ which are conjugate in the lattice. Suppose we have a conjugator $g$ for $u$ and $v$ such that $g$ lies in the ambient Lie group. How close to $g$ is a conjugating element from $\Gamma$?
\end{question}

We will now outline the structure of this paper. In Section \ref{sec:preliminaries} we discuss the relevant background information related in particular to the structure of symmetric spaces. Sections \ref{sec:semisimple:centraliser}-\ref{sec:semisimple:bounded conj in Gamma} deal with real hyperbolic elements, while Sections \ref{sec:unipotent}-\ref{sec:unipotent:constructing a conjugator} tackle the problem for unipotent elements. 

In Section \ref{sec:semisimple:centraliser} we describe the geometry of the centralisers of real hyperbolic elements of $G$. Theorems \ref{thmspecial:bounded conj in G} and \ref{thmspecial:bounded conj in G from lattice} are the main objectives of Section \ref{sec:semisimple:bounding in G}.
In Section \ref{sec:semisimple:no uniform bound} we consider the question of whether there is a uniform linear bound for the length of short conjugators between all real hyperbolic elements in $G$. We answer the question in the negative. Section \ref{sec:semisimple:bounded conj in Gamma} discusses the issue of translating the conjugator found in Theorem \ref{thmspecial:bounded conj in G from lattice}, which lies in the ambient Lie group $G$, into the lattice. In specific cases we show this can be done without losing the linear control on conjugator length. However the question of what happens in general, even for real hyperbolic elements, remains open.

Section \ref{sec:unipotent} introduces the problem for unipotent elements, and in particular includes a brief overview of what happens when we apply our method to upper-triangular unipotent matrices. In Section \ref{sec:unipotent:root system} we show how the root system of $G$ plays a vital role in the conjugacy of unipotent elements. Finally, construction of the short conjugator is done in Section \ref{sec:unipotent:constructing a conjugator}.

\vspace{3mm}
\noindent {\sc Acknowledgements:} The author would like to thank Cornelia Dru\c{t}u for many helpful conversations on the material of this paper. The input of Romain Tessera and Martin Bridson is also much appreciated.

\section{Preliminaries}\label{sec:preliminaries}

Let $G$ be a semisimple real Lie group and let $X$ be its associated symmetric space. For background on the structure of symmetric spaces we refer the reader to \cite{Helg01} and \cite{Eber96}. The author's thesis \cite{SaleThesis} contains the contents of this paper, and as such the preliminaries described there should also provide the required background material.

\subsection{Symmetric spaces and their isometries}\label{sec:sym spaces and isometries}

Given an isometry $g$ of a symmetric space $X$ (or any $\CAT{0}$--space) we can consider the following set
		$$\MIN{g}=\left\{x \in X \mid d_X(x,gx)=\inf_{q \in X}d_X(q,gq)\right\}.$$
If $\MIN{g}$ is non-empty then we say that $g$ is a \emph{semisimple} isometry of $X$, otherwise it is called parabolic. The semisimple isometries which fix some point in $X$ are called \emph{elliptic}. Various names have been given to those which don't fix a point, including hyperbolic, axial and loxodromic. To minimise confusion, we will simply call these elements the \emph{non-elliptic semisimple} elements of $G$.

It is not hard to see that the non-elliptic semisimple isometries of $X$ will translate some geodesic (this is shown, for example, in \cite{BGS85}). Given a bi-infinite geodesic $c:\R \to X$ consider $P(c)$, the subspace of $X$ consisting of all geodesics that are parallel to $c$. We call an isometry $g$ of $X$ \emph{real hyperbolic} if $g$ translates some geodesic $c$ in $X$ and furthermore $\MIN{g}=P(c)$. This says precisely that any geodesic parallel to $c$ will also be translated by $g$. The nomenclature used for these elements is justified in \cite[\S 3.1.3 \& Lemma 3.2.3]{SaleThesis}.

Let $\Lie{g=k \oplus p}$ be a Cartan decomposition of the Lie algebra $\Lie{g}$ of $G$.
There exists a point $p \in X$ such that the subalgebra $\Lie{k}$ is the Lie algebra of the maximal compact subgroup $K=G_p=\{g \in g \mid gp=p\}$. Meanwhile $\Lie{p}$ can be identified with $T_pX$, the tangent space at $p$ of $X$.
A geodesic $c : \R \to X$ such that $c(0)=p$ determines a vector in $T_pX$ and hence an element $H \in \Lie{p}$. Consider the maximal abelian subspace $\Lie{a}$ of $\Lie{p}$ which contains $H$. The submanifold $\exp(\Lie{a})p$ is a maximal flat in $X$. This follows from the following two facts:
\begin{itemize}
\item \cite[Ch. IV Thm 7.2]{Helg01} for $\Lie{s} \subset \Lie{p}$, the submanifold $\exp(\Lie{s})p$ in $X$ is totally geodesic in $X$ if and only if $\Lie{s}$ is a Lie triple system;
\item \cite[Ch. IV Thm 4.2]{Helg01} for $Y_1,Y_2,Y_3 \in \Lie{p}$, the curvature tensor at $p$ is given by $R_p(Y_1,Y_2)Y_3=-[[Y_1,Y_2],Y_3]$.
\end{itemize}

\begin{lemma}\label{lemma:flats and abelian subspaces}
Every maximal flat $F$ containing $p$ is of the form $F=\exp(\Lie{a})p$ for some maximal abelian subspace $\Lie{a}$ of $\Lie{p}$.
\end{lemma}

\begin{proof}
This follows from \cite[Ch. V Prop 6.1]{Helg01}.
\end{proof}

\subsubsection{Weyl chambers}

In order to define the Weyl chambers of a flat in $X$ we first need to describe the \emph{root-space decomposition} of the Lie algebra of $G$. Fix a Cartan decomposition $\Lie{g = k \oplus p}$. Let $\Lie{a}$ be a maximal abelian subspace of $\Lie{p}$. Then, since the operators $\ad(H)$, for $h \in \Lie{a}$, are simultaneously diagonalisable we can consider, for linear functionals $\lambda : \Lie{a} \to \R$, the eigenspaces
		$$\Lie{g}_\lambda = \{ Y \in \Lie{g} \mid \ad(H)Y = \lambda(H)Y \ \mathrm{ for \ all } \ H \in \Lie{a} \}.$$
Those $\lambda$ for which $\Lie{g}_\lambda$ is non-empty are called \emph{roots} of $\Lie{g}$ with respect to $\Lie{a}$ and the spaces $\Lie{g}_\lambda$ the corresponding \emph{root-spaces}. Let $\Lambda$ be the set of all non-zero roots of $\Lie{g}$ with respect to $\Lie{a}$. The \emph{root-space decomposition} of $\Lie{g}$ is the following:
		$$\Lie{g=g}_0+\sum\limits_{\lambda \in \Lambda} \Lie{g}_\lambda.$$
The roots $\Lambda$ form a root system in the dual space $\Lie{a}^*$.  A subset $\Pi$ of $\Lambda$ is called a \emph{base} if it is a basis for $\Lie{a}^*$ and if any root $\lambda$ can be written as
		$$\lambda = \sum\limits_{\alpha \in \Pi} c_\alpha \alpha$$
in such a way that either each $c_\alpha$ is non-negative or each $c_\alpha$ is non-positive. The elements of $\Pi$ are called \emph{simple roots} and those elements for which $c_\alpha \geq 0$ for each $\alpha \in \Pi$ are called \emph{positive roots} with respect to $\Pi$. We denote the set of positive roots by $\Lambdaplus$.

Consider a flat $F$ in $X$. By Lemma \ref{lemma:flats and abelian subspaces} there exists a maximal abelian subspace $\Lie{a}$ of $\Lie{p}$ such that $F=\exp(\Lie{a})p$. Let $\Lambda$ be the corresponding set of roots, $\Pi$ a set of simple roots and $\Lambdaplus$ the corresponding positive roots. The set of elements $H \in \Lie{a}$ for which $\lambda(H) > 0$ for each $\lambda \in \Pi$ forms an open \emph{Weyl chamber} in $\Lie{a}$, denoted $\Lie{a}^+$. The corresponding set $\mathcal{C}_\Pi=\exp(\Lie{a}^+)p$ is called an open Weyl chamber in $X$. The choice of $\Pi$ determines the Weyl chamber $\Lie{a}^+$.

For each root $\lambda \in \Lambda$ the kernel is a hyperplane in $\Lie{a}$. These are called the \emph{singular hyperplanes} of $\Lie{a}$. The \emph{walls} of $\mathcal{C}_\Pi$ are contained in the singular hyperplanes and are defined, for a subset $\Theta \subset \Pi$, as
		$$\mathcal{C}_\Theta = \{ \exp (H) \in \overline{\mathcal{C}_\Pi} \mid \lambda(H)=0\ \mathrm{ for } \ H \in \Pi \setminus \Theta \}$$
where $\overline{\mathcal{C}_\Pi}$ is the closure of $\mathcal{C}_\Pi$ in $F$. The flat $F$ is partitioned into Weyl chambers and walls. In fact, after removing all the singular hyperplanes from $F$, the connected components of what remains are all the Weyl chambers corresponding to the different choices for $\Pi$.

Let $c:\R \to X$ be a geodesic in $F$ with $c(0)=p$. Then there exists $H \in \Lie{a}$ such that $c(t)=\exp(tH)$ for every $t \in \R$. If $H$ is contained in a singular hyperplane of $\Lie{a}$ then we say the geodesic $c$ is \emph{singular}. Otherwise $H$ is contained in some Weyl chamber $\mathcal{C}_\Pi$ and we call $c$ a \emph{regular geodesic} in $X$. Since every geodesic in $X$ is contained in some maximal flat this definition extends to all geodesics. The following is an equivalent definition of regular and singular geodesics:

\begin{prop}
A geodesic is regular if and only if it is contained in a unique maximal flat.
\end{prop}

\begin{proof}
See \cite[\S 2.11]{Eber96}.
\end{proof}

\subsubsection{Slopes of geodesics}\label{sec:lattice background:Ideal Boundary}

Given two geodesic rays $\rho_1,\rho_2$ in $X$, we say they are \emph{asymptotic} if they are at finite Hausdorff distance from one-another. This defines an equivalence relation on geodesic rays in $X$, the equivalence classes of which form the \emph{ideal boundary} $\bdry{X}$ of $X$. The action of an isometry $g \in G$ on $X$ can be extended to an action on $\bdry{X}$ since $\rho_1$ and $\rho_2$ are asymptotic if and only if $g\rho_1$ and $g\rho_2$ are asymptotic. Hence we may consider the quotient of the action of $G$ on $\bdry{X}$. We denote this quotient by $\Dmod$. The $\Dmod$--\emph{direction}, or \emph{slope}, of a ray $\rho$ is the image of $\rho$ under the quotient maps.

Consider a bi-infinite geodesic $c:\R \to X$. This determines two boundary points, one for each end of the geodesic. Although, by the definition of a symmetric space, there is an isometry $\varphi=s_{c(0)}$ of $X$, the geodesic involution at $c(0)$, such that $\varphi c(t)=c(-t)$ for all $t\geq 0$, this isometry will not be in the connected component of the group of isometries of $X$, and thus not in $G$. Thus the two ends of $c$ determine two ideal points corresponding to $c(\infty)$ and $c(-\infty)$, and these will usually give rise to distinct slopes. We call the \emph{slope of the bi-infinite geodesic} $c:\R \to X$ the projection of $\sigma (\infty )$ onto $\Dmod$.

If the geodesic $c$ is regular, then we say the corresponding $\Dmod$--directions are regular, while if $c$ is singular its slopes are said to be singular too. Equivalently, the regular $\Dmod$--slopes are the ones contained in the interior of $\Dmod$, while the singular slopes are those in the boundary of $\Dmod$.

\subsubsection{Families of parallel geodesics}

We defined in Section \ref{sec:lattice background:Ideal Boundary} the slope of a geodesic. Given a geodesic $c:\R \to X$, its slope $\xi$ belongs to the set $\Dmod$, which is the closure of a model chamber in the boundary of $X$. If the slope of $c$ is contained in the interior of $\Dmod$ then it is regular and therefore contained in a unique maximal flat $F$. In fact, by the Flat Strip Theorem \cite[Pg.\ 182]{BH99}, any geodesic parallel to $c$ will also be contained in $F$. Hence $F$ is equal to the subspace of $X$ containing all geodesics parallel to $c$.

If $c$ is a singular geodesic in $X$ then it will be contained in a whole family of maximal flats. Again, using the Flat Strip Theorem, any geodesic parallel to $c$ must be contained in one of these flats.

More formally, let $P(c)$ denote the subspace of $X$ consisting of all geodesics that are parallel to $c$. So when $c$ is regular, $P(c)$ is equal to the unique maximal flat $F$ described above. While if $c$ is singular $P(c)$ will be the union of (infinitely many) maximal flats. The structure of these sets is discussed in more detail in \cite[\S 2.20]{Eber96}.

\begin{lemma}\label{lem:transitive actions on P sigma}
$G$ acts transitively on the set of subspaces of the form $P(\sigma)$, where $\sigma$ varies over geodesics with the same slope.
\end{lemma}

\begin{proof}
Let $\sigma,\tau:\R \to X$ be non-parallel geodesics of the same slope and let $F_1$, $F_2$ be any pair of maximal flats containing $\sigma$, $\tau$ respectively. The transitivity of the action of $G$ on the set of maximal flats in $X$ is well known and follows from \cite[Ch. V Thm 6.4]{Helg01}. Further more we know there exists $g \in G$ such that $gF_1=F_2$ and $g\sigma(0)=\tau(0)$. We now have two geodesics, $g\sigma$ and $\tau$, which are contained in the same maximal flat and have the same slope. If the positive rays, that is $g\sigma[0,\infty)$ and $\tau[0,\infty)$, are in the same Weyl chamber, then having the same slope implies the geodesics must coincide, hence $g\sigma=\tau$. If they are not in the same Weyl chamber then we apply an element of the Weyl group of $F_2$, which acts transitively on the Weyl chambers, so that they end up in the same Weyl chamber. An element of the Weyl group is a coset of the point-wise stabiliser of $F_2$. So by choosing a representative of this coset we have $k \in K=G_{\tau(0)}$ such that $kg\sigma=\tau$. Finally, if $\sigma^\prime$ is any geodesic parallel to $\sigma$, then $kg\sigma^\prime$ will be parallel to $kg\sigma=\tau$. Hence $kgP(\sigma)\subseteq P(\tau)$ and equality follows by symmetry.
\end{proof}

\subsection{Asymptotic cones}

\label{sec:background:asymptotic} 

We briefly discuss here asymptotic cones and a result of Kleiner and Leeb about the asymptotic cones of a symmetric space. Before we define an asymptotic cone we should discuss ultralimits and ultrafilters.

A \emph{non-principal ultrafilter} $\omega$ on $\mathbb{N}$ is a finitely additive probability measure on $\mathbb{N}$ which takes values of either $0$ or $1$ and all finite sets have zero measure. Given a sequence $(a_{n})_{n \in \mathbb{N}}$ of real numbers the \emph{ultralimit} of this sequence is $a = \lim_{\omega}(a_{n})\in\R$ which has the property that $\omega \lbrace n \in \mathbb{N} \mid \left| a_{n}-a\right| < \varepsilon \rbrace=1$ for every $\varepsilon >0$.

Let $X$ be a metric space with metric $d$ and let $p=(p_{n})_{n \in \mathbb{N}}$ be a sequence of points in $X$. Let $(d_{n})_{n \in \mathbb{N}}$ be a sequence in $\left( 0,\infty\right)$ which diverges to infinity. Given a non-principal ultrafilter $\omega$ on $\mathbb{N}$ we can define the \emph{asymptotic cone} $\cone{X}{p}{d_{n}}$ to be the quotient space of sequences $(x_{n})_{n \in \mathbb{N}}$ such that $\lim _{\omega} \frac{d(p_{n},x_{n})}{d_{n}} < \infty$ under the equivalence relation saying that two sequences $(x_{n})$ and $(y_{n})$ are equivalent if and only if $\lim _{\omega} \frac{d(x_{n},y_{n})}{d_{n}} = 0$. We can define a metric $d_{\omega}$ on the cone by setting $d_{\omega}((x_{n}),(y_{n})) = \lim _{\omega} \frac{d(x_{n},y_{n})}{d_{n}}$. A point $x \in \cone{X}{p}{d_{n}}$ is said to be the ultralimit of a sequence of points $(x_{n})$ in $X$ if $(x_{n})$ is a member of the equivalence class determining $x$.

As explained in, for example, \cite[Appendix 5]{BGS85}, the ideal boundary of a symmetric space can be given a spherical building structure. If we take an asymptotic cone of a symmetric space $X$ then, as Kleiner and Leeb put it in \cite{KL97}, intuitively speaking we are ``pulling the spherical building structure from infinity to the space of directions.'' Spherical buildings have a useful property, which can be described as \emph{rigidity of angles}. This means that given two points in a spherical building, their location inside their respective chambers determines a finite set of possible (angular) distances between them. When this property is pulled to the tangent space, by taking the asymptotic cone of our symmetric space, it transfers to a similar statement regarding the angle between intersecting geodesics.

We refer the reader to \cite{KL97} for the definition of a Euclidean building used by Kleiner and Leeb. However, note that in her PhD thesis \cite{Parr00} Parreau showed that Kleiner and Leeb's axioms are equivalent to the definition of a building given by Tits. 

\begin{thm}[Kleiner--Leeb \cite{KL97}] \label{thm:AC of X}
Let $X$ be a symmetric space of noncompact type. For any sequence of positive numbers $(d_{n})$ diverging to infinity and for any sequence of points $p=(p_{n})$ in $X$ the asymptotic cone $\cone{X}{p}{d_{n}}$ is a Euclidean building modelled on the Euclidean Coxeter complex $(E,M)$, where $E$ is $\mathrm{rank}(X)$--dimensional Euclidean space and $M$ is the quotient of the set-wise stabiliser $\mathrm{Stab}_{G}(E)$ by the point-wise stabiliser $\mathrm{Fix}_{G}(E)$.
\end{thm}

\subsection{A note on lattices}

A \emph{lattice} in a semisimple real Lie group $G$ is a discrete subgroup $\Gamma$ such that $\Gamma \backslash G$ has finite volume with respect to the Haar measure on $G$. If the quotient $\Gamma \backslash G$ is compact, then we say $\Gamma$ is a \emph{cocompact} or \emph{uniform} lattice in $G$. Otherwise we say it is \emph{non-uniform}.
A lattice is said to be \emph{irreducible} if, whenever $G$ is given as a product of Lie groups $G_1 \times G_2$, then the projections of $\Gamma$ into each factor are dense.

Suppose that $\Gamma$ is irreducible. Let $d_\Gamma$ denote a word metric on $\Gamma$ with respect to some finite generating set. We can also consider the size of an element of $\Gamma$ using a Riemannian metric $d_G$ on $G$. By a theorem of Lubotzky, Mozes and Raghunathan \cite{LMR00}, provided the real rank of $G$ is at least $2$, it does not matter which we use. Furthermore, this means that if we fix a basepoint $p$ in the symmetric space $X$ then we could also use $d_\Gamma(p,\gamma p)$ to estimate $d_\Gamma(1,\gamma)$.

\section{Centralisers of real hyperbolic elements}\label{sec:semisimple:centraliser}

Let $g$ be a non-elliptic semisimple isometry of $X$ and suppose $\sigma:\R \to X$ is a geodesic translated by $g$, oriented so that $g\sigma(0) = \sigma(t)$ for some $t>0$. We define the \emph{slope} of $g$ to the be $\Dmod$--direction of $\sigma$ corresponding to the positive direction, $\sigma(\infty)$. Of course parallel geodesics are asymptotic and so we get the same $\Dmod$--direction regardless of which geodesic we consider. We say an isometry $g$ is \textit{regular semisimple} if it is non-elliptic semisimple and its slope is regular in $\Dmod$. When $g$ is non-elliptic semisimple but its slope is singular we say $g$ is a \textit{singular semisimple}  isometry.

\begin{lemma}\label{lemma:regular centraliser and flat stabiliser}
Let $a$ be a regular semisimple element in $G$, contained in a maximal torus $A$. Then:
\begin{enumerate}[label={({\arabic*})}]
\item\label{item:1:lemma:regular centraliser and flat stabiliser} there exists a unique maximal flat $F_a$ in $X$ which is stabilised by $a$;
\item\label{item:3:lemma:regular centraliser and flat stabiliser} for any $p \in F_a$, let $K$ be the stabiliser of $p$, then  $\Stab(F_a)=Z_G(a)N_K(A)$, where $Z_G(a)$ is the centraliser of $a$ in $G$ and $N_K(A)$ is the normaliser of $A$ in $K$.
\end{enumerate}
\end{lemma}

\begin{proof}
Suppose $a$ is regular semisimple, translating a geodesic $c$. We can decompose $a$ as $a=hk=kh$ where $h$ is real hyperbolic and $k$ is elliptic. The real hyperbolic component will translate $c$, hence $\MIN{h}$ is equal to the maximal flat $F_a$. Consider the action of $k=h^{-1}a$ on $F_a$. It will fix $c$ pointwise and hence must act trivially on $F_a$ since any non-trivial elliptic isometry on $F_a$ will permute the Weyl chambers and therefore cannot fix a regular geodesic. But this tells us that the action of $a$ on $F_a$ is precisely the same as that of $h$, thus $\MIN{a}=F_a$. 

Suppose that another flat $F'$ is stabilised by $a$. Since $a$ is not elliptic , it must act on $F'$ hyperbolically, by translating some geodesic in $F'$. This geodesic must therefore be contained in $\MIN{a}=F$. So $F$ and $F'$ intersect in a regular geodesic, hence must be equal. This proves \ref{item:1:lemma:regular centraliser and flat stabiliser}.

Let $p,K$ be as in \ref{item:3:lemma:regular centraliser and flat stabiliser} and note that $F_a=Ap$. Let $z \in Z_G(a)$ and $w \in N_K(A)$. Then $zwF_a=zF_a$ and $azF_a=zaF_a=zF_a$, so by \ref{item:1:lemma:regular centraliser and flat stabiliser} $zF_a=F_a$. Hence $Z_G(a)N_K(A) \subseteq \Stab(F_a)$.
Now let $g \in \Stab(F_a)$. Then there exists $b \in A$ such that $gp=bp$. Thus $b^{-1}g=k \in K$ and stabilises $F_a$. Since $kAk^{-1}p=F_a$ we see that $kAk^{-1}$ is a maximal torus stabilising the flat $F_a$. Such a torus is unique, so $k \in N_K(A)$. Hence $g=bk \in Z_G(a)N_K(A)$ and \ref{item:3:lemma:regular centraliser and flat stabiliser} holds.
\end{proof}

We can see that the orbit of the centraliser of a regular semisimple element will be a maximal flat. Consider instead a singular semisimple element that is also real hyperbolic element. We may call such elements \emph{singular real hyperbolic}. The following Lemma tells us that the orbit of the centraliser of a singular real hyperbolic element will contain many maximal flats.

\begin{lemma}\label{lemma:singular centraliser and MIN}
Let $a$ be a singular real hyperbolic element in $G$. Then:
\begin{enumerate}[label={({\arabic*})}]
\item\label{item:1:lemma:singular centraliser and MIN} the subspace $\MIN{a}$ is precisely the set of all geodesics translated by $a$, which is the Riemannian product of a Euclidean space and a symmetric space of noncompact type; and
\item\label{item:2:lemma:singular centraliser and MIN} for any $p \in \MIN{a}$, let $K$ be the stabiliser of $p$, then $$Z_G(a) \subseteq \Stab(\MIN{a}) \subseteq Z_G(a)K$$ where $Z_G(a)$ is the centraliser of $a$ in $G$. 
\end{enumerate}
\end{lemma}

\begin{proof}
The first part of assertion \ref{item:1:lemma:singular centraliser and MIN} follows from the definition we gave for a real hyperbolic element, while for a proof of the latter part of \ref{item:1:lemma:singular centraliser and MIN} we refer the reader to \cite[2.11.4]{Eber96}.

For \ref{item:2:lemma:singular centraliser and MIN}, take $b \in Z_G(a)$ and let $c'$ be any geodesic translated by $a$. Then $bc'=bac'=abc'$ implies that $bc'$ is translated by $a$, hence is contained in $\MIN{a}$. Now let $g$ stabilise $\MIN{a}$. Let $c_1,c_2$ be the geodesics translated by $a$ such that $c_1(0)=p$ and $c_2(0)=gp$ respectively. Since $c_1$ and $c_2$ are parallel they are contained in a common flat $F=Ap$, where $A$ is a maximal abelian Lie subgroup of $G$ which contains $a$. There exists $b \in A$ such that $bc_2=c_1$ and $bgp=p$. Hence $bg \in K$ and in particular $g \in Z_G(a)K$.
\end{proof}

\section{Finding a short conjugator for real hyperbolic elements}\label{sec:semisimple:bounding in G}

The aim of this section is to obtain a control on the length of a conjugator between two real hyperbolic elements $a,b$ in $G$. The control will be linear, but the constant in the upper bound will depend on the slope of $a$ and $b$, and hence their conjugacy class. Our method of demonstrating this is to first show that a conjugator corresponds to an isometry that maps $\MIN{a}$ to $\MIN{b}$. Then, by obtaining a control on the distance from an arbitrary basepoint $p$ to $\MIN{a}$ in terms of $d_X(p,ap)$, we can obtain a control on the length of a conjugator from $G$.

\subsection{Relating conjugators to maps between flats}\label{sec:semisimple:bounding in G:maps between flats}

Here we show why we can obtain a short conjugator by understanding the distance to $\MIN{a}$ and $\MIN{b}$. In the following let $\pi_a$ be the orthogonal projection of $X$ onto $\MIN{a}$ and for $x \in X$ let $G_q=\{k \in G \mid gq=q\}$.

\begin{prop}\label{lem:conjugator and MIN sets}
Let $a,b$ be conjugate non-elliptic semisimple elements in $G$. Then:
\begin{enumerate}[label={({\arabic*})}]
		\item\label{item:1:prop:conjugator and MIN sets} for $g \in G$, if $gag^{-1}=b$ then $g\MIN{a}=\MIN{b}$;
		\item\label{item:2:prop:conjugator and MIN sets} for $h \in G$, if $h\MIN{a}=\MIN{b}$ then there exists $x \in G_{\pi_a(p)}$ such that $(hx)a(hx)^{-1}=b$.
\end{enumerate}
\end{prop}

\begin{proof}
For \ref{item:1:prop:conjugator and MIN sets}, let $c$ be any geodesic stabilised by $a$. Since $g=bga^{-1}$ we see that the geodesic $gc$ is translated by $b$ and so is contained in $\MIN{b}$. Hence $g\MIN{a} \subseteq \MIN{b}$. Similarly we get $g^{-1}\MIN{b} \subseteq \MIN{a}$ and \ref{item:1:prop:conjugator and MIN sets} is proved.

Next suppose that $g,h \in G$ are such that $gag^{-1}=b$ and $h\MIN{a}=\MIN{b}$. By the first part we observe $g^{-1}h\MIN{a}=\MIN{a}$, so $g^{-1}h \in \Stab(\MIN{a}) \subseteq Z_{G}(a)K$, with the latter relationship coming from Lemma \ref{lemma:singular centraliser and MIN}, where we take $K=G_{\pi_a(p)}$. Then there exists $x \in K$ such that $g^{-1}hx \in Z_{G}(a)$. This implies $(g^{-1}hx)a(g^{-1}hx)^{-1}=a$ and so $(hx)a(hx)^{-1}=gag^{-1}=b$, proving \ref{item:2:prop:conjugator and MIN sets}.
\end{proof}

If $a$ and $b$ are regular semisimple elements in $G$ then they each stabilise a unique maximal flat in $X$. Suppose $a$, $b$ stabilise maximal flats $F_a$ and $F_b$ respectively. By taking our basepoint $p$ to be in $F_a$ we can build a quadrilateral which has two vertices in $F_a$ and two vertices in $F_b$, as in Figure \ref{fig:quad}. In light of \ref{lem:conjugator and MIN sets}, the aim is to find an element $g \in G$ of a controlled size which maps the flat $F_a$ to $F_b$. We will then obtain a conjugator of controlled size.

\begin{figure}[h!]
   \labellist
    \tiny\hair 5pt
    \pinlabel $p$ [r] at 45 96
    \pinlabel $ap$ [r] at 45 148
    \pinlabel $gp$ [l] at 186 51
    \pinlabel $bgp=gap$ [b] at 184 102
 \small   \pinlabel $F_a$ at 12 175
    \pinlabel $F_b$ at 205 141
    \endlabellist

    \centering
    \includegraphics[width=5cm]{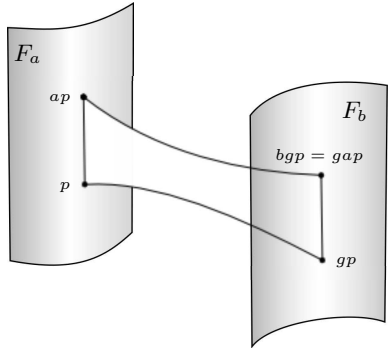}
     \caption[A quadrilateral demonstrating the conjugacy of hyperbolic elements]{A quadrilateral in $X$ demonstrating the conjugacy of $a$ and $b$ in $G$.}
     \label{fig:quad} 
\end{figure}

\begin{prop}\label{prop:distance to flat bounded implies conjugator bounded}
Let $p$ be our basepoint in $X$. Suppose for all semisimple $a$ in $G$ we can find a constant $\ell(a)$ such that:
$$d_X(p,\MIN{a})\leq \ell(a)d_X(p,ap)\textrm{.}$$
Then for $a,b$ conjugate hyperbolic elements in $G$ there exists a conjugator $g \in G$ such that:
$$d_X(p,gp) \leq \ell(a)d_X(p,ap)+\ell(b)(p,bp)\textrm{.}$$
\end{prop}

\begin{proof}
Choose points $p_{a} \in \MIN{a}$ and $p_{b} \in \MIN{b}$ which satisfy 
		$$d_{X}(p,p_{a}) \leq \ell(a)d_{X}(p,ap)\ \ \textrm{and} \ \ d_{X}(p,p_{b}) \leq \ell(b)d_{X}(p,bp).$$
Let $g_{1}\in G$ be such that $g_{1}\MIN{a}=\MIN{b}$ and $g_{1}p_{a}=p_{b}$. By Lemma \ref{lem:conjugator and MIN sets} there exists $x$ in $G_{p_b}$ such that $(xg_{1})a(xg_{1})^{-1}=b$. Let $g=xg_{1}$.

\begin{figure}[h!]
   \labellist
      \small\hair 5pt
       \pinlabel $p$ [r] at 40 177
       \pinlabel $p_{a}$ [t] at 74 41
       \tiny \pinlabel $p_{b}=gp_{a}$ [t] at 235 128
       \small \pinlabel $gp$ [b] at 185 219
       \tiny
       \pinlabel $\MIN{a}$ at 27 18
       \pinlabel $\MIN{b}$ at 238 82
    \endlabellist
    \centering
    \includegraphics[width=7cm]{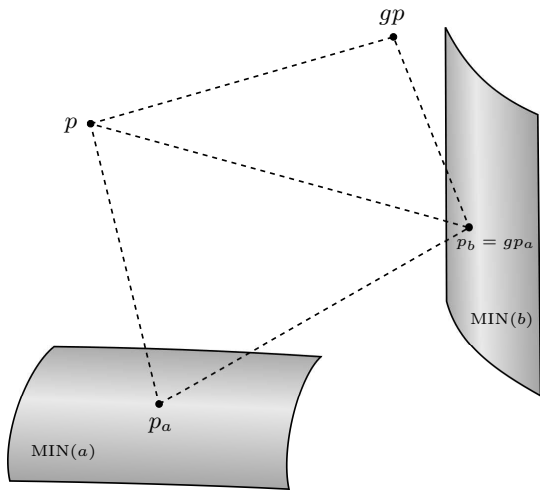}
  \caption[Using the distance to flats to bound the length of a conjugator]{Obtaining an upper bound on $d_{X}(p,gp)$.}
\end{figure}

To finish the proof we need to check that we have the required upper bound on $d_{X}(p,gp)$. By the triangle inequality:
\begin{eqnarray*}
d_{X}(p,gp) & \leq & d_{X}(p,p_{b})+d_{X}(p_{b},gp) \\
            & =    & d_{X}(p,p_{b}) + d_{X}(gp_{a},gp) \\
            & =    & d_{X}(p,p_{b})+d_{X}(p,p_{a})) \\
            & \leq & \ell(a)d_{X}(p,ap)+\ell(b)d_{X}(p,bp)\textrm{.}
\end{eqnarray*}
\end{proof}

\subsection{Bounding the distance to a flat}\label{sec:semisimple:bounding in G:distance to flat}

In Proposition \ref{prop:distance to flat bounded implies conjugator bounded} we saw how finding some constant $\ell(a)$ such that $d_X(p,\MIN{a})\leq \ell(a)d_X(p,ap)$ helps us to control the size of a conjugator in $G$ between $a$ and another element $b$. The aim of this section is to find such constants for the case when $a$ is real hyperbolic. In fact, we will determine a value for $\ell(a)$ which depends only on the slope of $a$. Note that we have the following:

\begin{lemma}
If $a$ is conjugate to $b$ then $a$ and $b$ have the same slope.
\end{lemma}

\begin{proof}
Let $a$ have slope $\xi \in \Dmod$. This means that the geodesic segment $\left[ p , ap \right]$, where $p$ is a point in $\textrm{MIN}(a)$, has $\Dmod$-direction $\xi$. Suppose $b=gag^{-1}$ for some $g \in G$. Then $g$ maps the bi-infinite geodesic through $p$ and $ap$ to a bi-infinite geodesic through $gp$ and $gap=bgp$. This geodesic is translated by $b$, so the slope of $b$ is the $\Dmod$-direction of the geodesic segment $\left[ gp,bgp \right] = g\left[ p,ap \right]$. Since the $\Dmod$--direction is defined to be $G$--invariant we have that the slope of $b$ is $\xi$.
\end{proof}

In order to determine the value of $\ell(a)$ we will use an asymptotic cone of $X$, which, by a result of Kleiner and Leeb \cite{KL97}, is a Euclidean building. It is helpful therefore to first determine the corresponding value in a Euclidean building. This will then be useful to find the value for symmetric spaces.
Recall that a real hyperbolic element $a$ satisfies $\MIN{a}=P(\sigma)$, where $\sigma$ is any geodesic translated by $a$ and $P(\sigma)$ is the subspace of $X$ consisting of all geodesics parallel to $\sigma$. We are therefore interested in the distance to similarly defined subspaces of a Euclidean building. Also recall that, for a Euclidean building $Y$, the quotient of $\bdry{Y}$ by the group of isometries of $Y$ is denoted by $\Dmod$, and $\theta : \bdry{Y} \to \Dmod$ is the natural map.

\begin{lemma}[{see \cite[Lemma 5.1]{HKM10}}]\label{lem:ray enters family of flats}
Let $Y$ be a Euclidean building, $\delta \in \bdry{Y}$, $c$ be a geodesic in $Y$ with one end asymptotic to $\delta$ and $E$ be the subset of $Y$ consisting of all geodesics parallel to $c$. Then for any point $p \in Y$ the geodesic ray emanating from $p$ which is asymptotic to $\delta$ enters the set $E$ in finite distance.
\end{lemma}

\begin{remark}
The proof of this Lemma given in \cite[Lemma 5.1]{HKM10} only covers algebraic Euclidean buildings. It is worth noting that by \cite{KT04} the asymptotic cone of a symmetric space is an algebraic Euclidean building, so their proof applies to the buildings which we are concerned with.
\end{remark}

\begin{prop}\label{prop:distance to MIN in Euclidean building when regular} 
Let $a$ be an isometry of a Euclidean building $Y$ such that $a$ translates all geodesics parallel to a geodesic $c$ and let $E$ be the subset of  $Y$ containing all these geodesics. Suppose the ray $c[0,\infty)$ satisfies $ac[0,\infty)\subset c[0,\infty)$ and is asymptotic to $\delta \in \bdry{E}$, where $\theta ( \delta ) = \xi \in \Dmod$. Then there exists a constant $\ell_{\xi}$ such that for any basepoint $p$ in $Y$ the following holds:
$$d(p,E) \leq \ell_{\xi} d(p,ap)\textrm{.}$$
\end{prop}

\begin{proof}
\begin{figure}[h!]
   \labellist
      \small \hair 5pt
         \pinlabel $p$ [b] at 37 195
         \pinlabel $ap$ [b] at 177 199
         \pinlabel $e$ [t] at 175 59
         \pinlabel $ae$ [t] at 251 59
         \pinlabel $E$ at 75 35
         \pinlabel $\delta$ [l] at 429 59
      \endlabellist
   \centering
   \includegraphics[width=10.2cm]{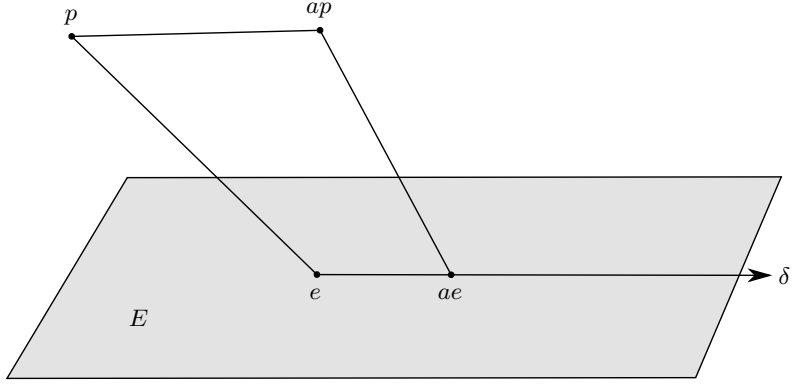}
  \caption[Bounding the distance to an apartment in a Euclidean building]{The angle $\measuredangle_{ae}(e,ap)$ is rigid, that is the angle is contained in a finite set which is determined by $\xi$. This leads to a bound on the distance from $p$ to $E$.}
 \label{fig:EuclideanBuildingRays} 
\end{figure}
Consider the ray emanating from $p$ which represents $\delta$. By Lemma \ref{lem:ray enters family of flats} this ray enters $E$. Let $e$ be the first point along this ray such that $e \in E$. By design $ae$ also lies on this ray. Now translate the ray by $a$. What we get is a geodesic triangle in $Y$, as seen in Figure \ref{fig:EuclideanBuildingRays}, with vertices $p,ae,ap$. By the rigidity of angles in Y, $\measuredangle_{ae}(p,ap)$ belongs to the finite set $D(\xi)$. Let $\phi$ be minimal in this set. Then since $Y$ is a $\textrm{CAT}(0)$ space we have that $d(p,ap)\geq d(ap,ae) \sin \phi$. Hence we put $\ell_{\xi}=\frac{1}{\sin \phi}$ and the proposition holds.
\end{proof}

\begin{lemma}\label{lemma:distance to MIN in symmetric space} 
Let the Tits building structure on $\bdry{X}$ have anisotropy polyhedron $\Delta_{mod}$. Fix a basepoint $p$ in $X$. Then for each element $\xi \in \Delta_{mod}$ there exists positive constants $\ell_{\xi}$ and $d_{\xi}$ such that for each $a \in G$ which is real hyperbolic of slope $\xi$ and such that $d_{X}(p,ap)>d_{\xi}$ the following holds:
$$d_{X}(p, \MIN{a}) \leq 2\ell_{\xi}d_{X}(p,ap)\textrm{.}$$
\end{lemma}

\begin{proof}[Proof {[regular slopes]}]
The proof which follows applies to the case when $\xi$ is regular. The proof for singular slopes is analogous, with slight modifications which are described at the end of this proof.

First note that the constant $\ell_{\xi}$ will be the same constant that we obtained in Proposition \ref{prop:distance to MIN in Euclidean building when regular}.

We proceed by contradiction, supposing the statement is false. We then obtain a sequence of regular semisimple elements $a_n$ in $G$, each of slope $\xi$, such that $d_{X}(p,a_{n}p)$ diverges to infinity and which satisfies: 
\begin{equation}\label{eq:lemma:distance to MIN in symmetric space when regular:contra}
d_X(p,F_{n})>2\ell_{\xi}d_{X}(p,a_{n}p)
\end{equation} where $F_{n}$ is the unique maximal flat stabilised by $a_{n}$. Write $d_{n}:=d_{X}(p,a_{n}p)$ and $D_{n}:=d_{X}(p,F_{n})$. Let $\pi_{n} : X \rightarrow F_{n}$ be the orthogonal projection onto $F_{n}$ and let $t_{n}$ be the translation length of $a_n$, that is $t_{n}:=d_{X}(\pi_{n}(p),a_{n}\pi_{n}(p))$. We split the proof into three parts depending on the limits of the ratios $t_{n}/D_{n}$ and $d_{n}/D_{n}$.

\vspace{2mm}\noindent\textsc{\underline{Case 1:}} $\lim_\omega(t_n/D_n)\neq 0 \neq \lim_\omega(d_n/D_n)	$.\vspace{2mm} 

\noindent In the first part of the proof we build a Euclidean building and use Proposition \ref{prop:distance to MIN in Euclidean building when regular} to obtain a contradiction under the assumption that $t_{n}/D_{n}$ does not converge to zero.

Since $d_{n}$ diverges to infinity, it follows from (\ref{eq:lemma:distance to MIN in symmetric space when regular:contra}) that $D_n$ does too. Hence we pick a non-principal ultrafilter $\omega$ and consider the asymptotic cone $Y=\cone{X}{p}{D_n}$. By choice of scalars it follows that the ultralimit $E$ of the sequence of flats $F_n$ is contained in $Y$ and lies a distance $1$ away from the point $p$ (when we view $p$ as an element of the cone $Y$).

Define the map $g:Y\rightarrow Y$ by sending $(x_{n}) \in Y$ to $(a_{n}x_{n})$. To check it is well-defined on $Y$ we need only observe that it moves $p$ a bounded distance:

\begin{eqnarray*}
d_{\omega}(p,gp) & = & {\lim}_{\omega}\left(\frac{d_{X}(p,a_{n}p)}{D_{n}}\right) \\
                 & = & {\lim}_{\omega}\left(\frac{d_{n}}{D_{n}}\right) \\
                 & \leq & {\lim}_{\omega}\left(\frac{1}{2\ell_{\xi}}\right) \\
                 & = & \frac{1}{2\ell_{\xi}}\textrm{.}
\end{eqnarray*}

Furthermore, since $a_n$ acts on $X$ by isometries for each $n$ it follows that $g$ acts on $Y$ by isometries.

By assumption $t_{n}/D_{n}$ does not converge to zero, hence $d_{\omega}(\pi(p),g\pi(p))>0$. It implies, since $t_{n}\leq d_{n}$, that $d_{X}(p,gp)>0$. Under these conditions we may apply Proposition \ref{prop:distance to MIN in Euclidean building when regular}, since $g$ acts on $E$ by translating along geodesics towards a boundary point $\xi$. This gives us the following contradiction:
\begin{eqnarray*}
1 = d_{\omega}(p,E)
  \leq \ell_{\xi}d_{\omega}(p,gp)
  \leq\frac{\ell_{\xi}}{2\ell_{\xi}} = \frac{1}{2}\textrm{.}
\end{eqnarray*}

\vspace{2mm}\noindent\textsc{\underline{Case 2:}} $\lim_\omega (t_n/D_n)=0=\lim_\omega (d_n/D_n)$.\vspace{2mm}

\noindent We must therefore assume the $\omega$-limit of $t_{n}/D_{n}$ is zero. We assume this for the second part of the proof and we also assume that the $\omega$-limit of $d_{n}/D_{n}$ is zero. We will build a sequence of quadrilaterals and take their Hausdorff limit. The limiting quadrilateral will be flat and intersecting a flat $F$ only in one edge, along a regular geodesic. This will give the contradiction.

Fix a flat $F$ in $X$ and a point $q \in F$. For each $n$ consider an isometry $g_{n} \in G$ which sends $\pi_{n}(p)$ to $q$ and $F_n$ to $F$. The first thing to note is that each geodesic segment $\left[ \pi_{n}(p),a_{n}\pi_{n}(p)\right]$ is mapped to a geodesic segment $T:=\left[ q , g_{n}a_{n}g_{n}^{-1}q\right]$ in $F$ of $\Dmod$-direction $\xi$. Consider the projection $\tau : \left[ g_{n}p,g_{n}a_{n}p \right] \rightarrow T$. For a fixed constant $h$ and for large enough $n$ we may pick a subsegment $S$ of $T$ of length $h$ such that the pre-image $\tau^{-1}(S)$ in $ \left[ g_{n}p,g_{n}a_{n}p \right]$ has length at most $d_{n}h/t_{n}$. 

\begin{figure}[b!]
   \labellist
      \small\hair 5pt
         \pinlabel $g_{n}p$ [b] at 14 262
         \pinlabel $g_{n}a_{n}p$ [b] at 327 264
         \pinlabel $q$ [r] at 109 51
         \pinlabel $g_{n}a_{n}g_{n}^{-1}q$ [l] at 251 50
         \pinlabel $F=g_{n}F_{n}$ at 53 20
         \pinlabel $S$ [t] at 169 46
         \pinlabel $D_{n}$ [r] at 79 152
         \pinlabel $L_{n}^{(1)}$ [r] at 136 152
         \pinlabel $L_{n}^{(2)}$ [l] at 182 152
         \pinlabel $b_{n}^{(1)}$ [b] at 110 238
         \pinlabel $b_{n}^{(2)}$ [b] at 176 233
         \pinlabel $h$ [l] at 185 66
         \pinlabel $h$ [r] at 145 65
         \pinlabel $\psi_{n}(h)$ [b] at 167 82
         \endlabellist

         \centering
    \includegraphics[width=12cm]{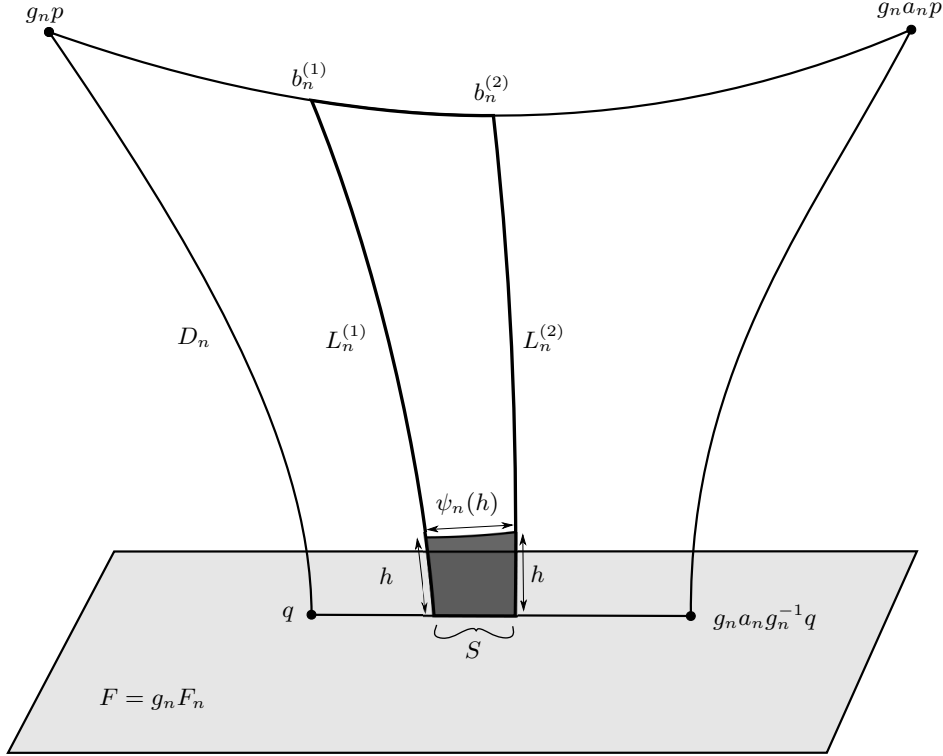}

     \caption{The quadrilateral $Q_n$ with side $S$.}
      \label{fig:RegSemisimpleQuadrilateral} 
  \label{quad1}
\end{figure}

Label points $b_{n}^{(1)},b_{n}^{(2)}$ on the geodesic $[g_{n}p,a_{n}g_{n}p]$ which are mapped under $\tau$ to each end of the segment $S$, see Figure \ref{fig:RegSemisimpleQuadrilateral}. Let $L_{n}^{(i)}:=d_{X}(b_{n}^{(i)},T)$ and without loss of generality assume $L_n^{(1)}\leq L_{n}^{(2)}$. Observe: $$D_{n}=d_{X}(g_{n}p,T)\leq d_{X}(g_{n}p,b_{n}^{(1)})+d_{X}(b_{n}^{(1)},T)$$ and replacing $g_{n}p$ by $g_{n}a_{n}p$ if necessary we get: $$L_{n}^{(1)}\geq D_{n}-\frac{d_{n}}{2}\textrm{.}$$ Using our hypothesis we therefore get $L_{n}^{(1)}>\left(2\ell_{\xi}-\frac{1}{2}\right)d_{n}$. So, referring back to the value of $\ell_{\xi}$ obtained in Proposition \ref{prop:distance to MIN in Euclidean building when regular}, since $2\ell_{\xi}-\frac{1}{2}>0$ for any choice of $\xi$, we get that $L_{n}^{(1)}$ diverges to infinity.

Now define quadrilaterals $Q_{n}$ as follows. We take one edge to be the segment $S$ and the two adjacent edges are those subsegments of $[b_{n}^{(i)},\tau (b_{n}^{(i)})]$ of length $h$ which include the points $\tau (b_{n}^{(i)})$, for $i=1,2$. The quadrilateral $Q_{n}$ has three sides of length $h$. Let $\psi_{n}(h)$ be the length of the fourth side. Define a map $\hat{\psi}_{n}:[0,L_{n}^{(1)}]\rightarrow \mathbb{R}$ measuring the distance across the flat rhombus (see Figure \ref{fig:psi}).
\begin{figure}[b!]
   \labellist
      \small\hair 5pt
         \pinlabel $L_{n}^{(1)}$ [r] at 11 71
         \pinlabel $L_{n}^{(1)}$ [l] at 146 71
         \pinlabel $t$ [r] at 34 28
         \pinlabel $t$ [l] at 123 28
         \pinlabel $\hat{\psi}_{n}(t)$ [b] at 79 55
         \pinlabel $h$ [b] at 79 0
         \pinlabel $\frac{h}{t_{n}}d_{n}$ [t] at 79 140
         \endlabellist
                  \centering
    \includegraphics[width=6cm]{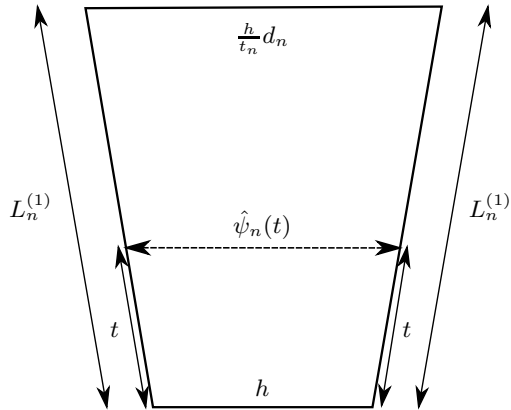}

     \caption{Defining $\hat{\psi}_{n}(t)$.}
      \label{fig:psi}
\end{figure}

Using the $\CAT{0}$ property of $X$ we see that
\begin{eqnarray*}
h \leq \psi_{n}(h) & \leq & \hat{\psi}_{n}(h) \\
                   & = & \left(\frac{L_{n}^{(1)}-h}{L_{n}^{(1)}}\right)h+\frac{h^{2}d_{n}}{L_{n}^{(1)}t_{n}} \\
                   & = & h - \frac{h^{2}}{L_{n}^{(1)}}+h^{2}\left(\frac{d_{n}}{D_{n}t_{n}}\right)\left(\frac{D_{n}}{L_{n}^{(1)}}\right)
\end{eqnarray*}
Above we showed that $L_{n}^{(1)}\geq D_{n}-\frac{d_{n}}{2}$. We can use this to show that $D_{n}/L_{n}^{(1)}$ is bounded above by $\left(1-\frac{1}{4\ell_{\xi}}\right)^{-1}$. This is therefore enough, since we have the assumption that $d_{n}/D_{n}$ converges to zero, to conclude that $\psi_{n}(h)$ converges to $h$ as $n$ tends to infinity. After translating the quadrilaterals $Q_{n}$ along the geodesics to $q$ we can find the Hausdorff limit $Q$ of a convergent subsequence of the quadrilaterals $Q_{n}$. The quadrilateral $Q$ will have four sides with length $h$, two right-angles and hence must be a flat square. But $Q$ intersects $F$ only through the regular geodesic segment $T$. This gives a contradiction.

\vspace{2mm}\noindent\textsc{\underline{Case 3:}} $\lim_\omega(t_n/D_n) = 0 \neq \lim_\omega(d_n/D_n)$.\vspace{2mm}

\noindent We conclude by combining both of the above arguments into one in order to obtain a contradiction when $t_{n}/D_{n}$ converges to zero but $d_{n}/D_{n}$ does not. We start by looking at the situation inside the Euclidean building $Y$ that we built in case 1. It is constructed so that $d_\omega(p,E)=1$, but in what follows we will show that in this case we would have the contradiction $d_\omega(p,E)<1$.

In $Y$, take the two geodesic rays asymptotic to $\xi$ which begin at $p$ and at $gp$ respectively. Note that the second ray is the image of the first under $g$. Also recall that both rays will enter the apartment $E$. Since $g$ fixes $E$ pointwise we see that the two rays must come together at some point $y$. In particular either $y$ is the point where the rays enter $E$ or it is not in $E$. We will show that $d_{\omega}(y,E)\leq (4\ell_{\xi})^{-1}$.

\begin{figure}[h!]
   \labellist
      \small\hair 5pt
         \pinlabel $p$ [r] at 41 140
         \pinlabel $gp$ [b] at 116 148
         \pinlabel $y$ [l] at 96 100
         \pinlabel $e$ [t] at 136 36
         \pinlabel $\xi$ [t] at 319 36
         \pinlabel $E$ at 28 11
         \endlabellist
                  \centering
    \includegraphics[width=10cm]{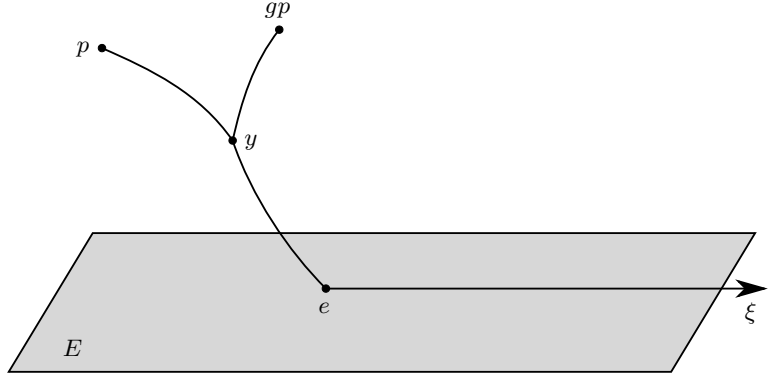}

     \caption[Two asymptotic rays entering an apartment]{Two geodesic rays asymptotic to $\xi$, entering the apartment $E$ at a point $e$ and merging at the point $y$.}
\end{figure}

Suppose that $d_{\omega}(y,E) > (4\ell_{\xi})^{-1}$ and let $(y_{n})$ be a sequence of points in $X$ which represent $y$ in $Y$. Then there exists $\varepsilon$ such that  $(4\ell_{\xi})^{-1}<\varepsilon < d_{\omega}(y,E)$ and $\omega\lbrace n \in \mathbb{N} \mid d_{X}(y_{n},F_{n})\geq \varepsilon D_{n} \rbrace=1$. We now proceed as in case 2, for each $n$ applying the isometry $g_{n}$ and constructing quadrilaterals $Q_{n}$ with one edge in the fixed flat $F$. As before we pick a segment $S$ of length $h$ from $\left[ q,g_{n}a_{n}g_{n}^{-1}q\right]$ whose pre-image under the projection onto $\left[ q,g_{n}a_{n}g_{n}^{-1}q\right]$ intersects $\left[ g_{n}y_{n},g_{n}a_{n}y_{n}\right]$ in a segment of length at least $\frac{h}{t_{n}}d_{X}(y_{n}.a_{n}y_{n})$. Let $L_{n}^{(i)}$ be the distances between the corresponding end-points of these subsegments and suppose $L_{n}^{(1)} \leq L_{n}^{(2)}$. Then in order to proceed as before we need that the function:
$$\hat{\psi}_{n}(t)=\left(\frac{L_{n}^{(1)}-t}{L_{n}^{(1)}}\right)h+\frac{t}{L_{n}^{(1)}}\frac{h}{t_{n}}d_{X}(y_{n},a_{n}y_{n})$$
converges to $h$ when we put $t=h$. We need to check two things: firstly that $L_{n}^{(1)}$ diverges to infinity and secondly that $\frac{d_{X}(y_{n},a_{n}y_{n})}{L_{n}^{(1)}}$ converges to zero. For the former we note that for all but finitely many $n \in \N$ we have the following:
\begin{eqnarray*}
L_{n}^{(1)} & \geq & d_{X}(y_{n},F_{n}) - \frac{1}{2}d_{X}(y_{n},a_{n}y_{n})\\
            & > & \left(2\ell_{\xi}\varepsilon - \frac{1}{2}\right)d_{X}(y_{n},a_{n}y_{n})
\end{eqnarray*}
If $d_{X}(y_{n},a_{n}y_{n})$ is bounded we use the first line to show $L_{n}^{(1)}$ is unbounded. Otherwise we use the second line, recalling that $\varepsilon > \frac{1}{4\ell_{\xi}}$. To prove that $\frac{d_{X}(y_{n},a_{n}y_{n})}{L_{n}^{(1)}}$ converges to zero we first check that $\frac{D_{n}}{L_{n}^{(1)}}$ is bounded. This is so because for all but finitely many $n$ we have the following:
\begin{eqnarray*}
 L_{n}^{(1)} & \geq & d_{X}(y_{n},F_{n}) - \frac{1}{2}d_{X}(y_{n},a_{n}y_{n}) \\
             & > & \varepsilon D_{n} - \frac{1}{2}d_{n} \\
\frac{L_{n}^{(1)}}{D_{n}} & > & \varepsilon - \frac{d_{n}}{2D_{n}}\\
\frac{D_{n}}{L_{n}^{(1)}} & < & \left( \varepsilon - \frac{1}{4\ell_{\xi}}\right)^{-1}
\end{eqnarray*}
Hence we see that $\frac{d_{X}(y_{n},a_{n}y_{n})}{L_{n}^{(1)}} = \frac{d_{X}(y_{n},a_{n}y_{n})}{D_{n}}\frac{D_{n}}{L_{n}^{(1)}}$ converges to zero by our choice of $y_{n}$. Then as before, after translating the quadrilaterals $Q_n$ so they each have a vertex at $q$, we take the Hausdorff limit of a convergent subsequence of these quadrilaterals and obtain a flat quadrilateral which intersects $F$ only through a regular geodesic segment. Here we have our contradiction and conclude that $d_{\omega}(y,E)\leq (4\ell_{\xi})^{-1}$.

To finish the argument we look at the triangle in $Y$ with vertices $p,gp,y$. In a similar manner to the proof of Proposition \ref{prop:distance to MIN in Euclidean building when regular} we use the fact that $Y$ is a $\textrm{CAT}(0)$ space to get $d_{\omega}(p,gp)\geq d_{\omega}(p,y) \ell_{\xi}^{-1}$, recalling that $\ell_{\xi}^{-1}$ is the sine of the minimal angle in the finite set $D(\xi)$ of possible angles between geodesics of $\Dmod$-direction $\xi$.

Therefore we have the following:
\begin{eqnarray*}
 d_{\omega}(p,E) & \leq & d_{\omega}(p,y) + d_{\omega}(y,E) \\
                & \leq & \ell_{\xi}d_{\omega}(p,gp) + \frac{1}{4\ell_{\xi}} \\
                & \leq & \ell_{\xi}\frac{1}{2\ell_{\xi}} + \frac{1}{4\ell_{\xi}} \\
                & = & \frac{1}{2} + \frac{1}{4\ell_{\xi}} < 1
\end{eqnarray*}
However, we also have $d_{\omega}(p,E) = \lim_{\omega}\left(\frac{D_{n}}{D_{n}}\right)=1$, thus giving the contradiction and proving the Lemma.
\end{proof}

\begin{proof}[Proof {[singular slopes]}]
In order to modify the above proof to work for singular directions we need to make the following adjustments. First we replace the flats $F_n$ by $\MIN{a_n}$. Case 1 continues as above with no change. For cases 2 and 3, the contradiction we obtain will be similar. Instead of using a fixed flat $F$, we use a fixed subspace $P(c)$ which consists of a family of geodesics parallel to some geodesic $c$ of slope $\xi$, for example we may take $M=\MIN{a_1}$ and $c$ any geodesic translated by $a_1$. By Lemma \ref{lem:transitive actions on P sigma} we know there exists $g_n \in G$ which sends $\MIN{a_n}$ to $M$. From the proof of Lemma \ref{lem:transitive actions on P sigma} it is also clear that $g_n$ can be chosen so it sends a geodesic translated by $a_n$ to a geodesic translated by $a_1$. Furthermore, if we fix a point $q$ in $M$, as we did in the above proof, then we can choose $g_n$ so it sends $\pi_n(p)$ to $q$. Once we have this, we can find a flat quadrilateral ${Q}$ in the same way as above, but it will intersect $M$ only in one side, which is a geodesic segment of slope $\xi$. The opposite edge of ${Q}$ will be a segment of a geodesic parallel to $c$, hence should be contained in $M$, but it is not.
\end{proof}

In light of Lemma \ref{lemma:distance to MIN in symmetric space} we can put $\ell(a)=\ell(b)=2\ell_\xi$, where $\xi$ is the slope of $a$ and $b$, into Proposition \ref{prop:distance to flat bounded implies conjugator bounded} to get the following:

\begin{thm}\label{thm:bounded conj in G}
Let $\ell_\xi$ and $d_\xi$ be the constants from Lemma \ref{lemma:distance to MIN in symmetric space}. Suppose $a$ and $b$ are conjugate real hyperbolic elements in $G$ with slope $\xi \in \Dmod$ and such that $d_X(p,ap),d_X(p,bp)\geq d_\xi$. Then there exists a conjugator $g \in G$ such that:
$$d_X(p,gp) \leq 2\ell_\xi \big(d_X(p,ap)+d_X(p,bp)\big)\textrm{.}$$
\end{thm}

The constant $2\ell_\xi$ that we have obtained will depend on the slope of $a$ and $b$, and hence on the conjugacy class. However it is important to note that it is independent of the basepoint $p$ that was chosen.

When we restrict our attention to a lattice $\Gamma$ in $G$, Theorem \ref{thm:bounded conj in G} will apply to all but finitely many real hyperbolic elements of $\Gamma$. This leads to the following result:

\begin{cor}\label{cor:bounded conj in G from lattice}
Let $\Gamma$ be a lattice in $G$. Then for each $\xi \in \Dmod$, there exists a constant $L_\xi$ such that two elements $a,b \in \Gamma$ are conjugate in $G$ if and only if there exists a conjugator $g \in G$ such that
		$$d_X(p,gp) \leq L_\xi \big(d_X(p,ap)+d_X(p,bp)\big)\textrm{.}$$
\end{cor}

We now offer a couple of corollaries to Theorem \ref{thm:bounded conj in G} which shed a little more light on the nature of short conjugators between real hyperbolic elements in a semisimple real Lie group. For a semisimple element $a$ of $G$, the translation length of $a$ is 
		$$\tau(a)=\inf \{ d_X(x,ax) \mid x \in X\}.$$ 
We can reformulate Theorem \ref{thm:bounded conj in G} so that it applies to all real hyperbolic elements in $G$, provided their translation length isn't too small.

\begin{cor}
Let $0 <\varepsilon\leq d_\xi$ and suppose that $a$ and $b$ are real hyperbolic elements of $G$ of slope $\xi$ and with translation lengths $\tau(a),\tau(b)\geq\varepsilon$. Then $a$ and $b$ are conjugate in $G$ if and only if there exists a conjugator $g \in G$ such that
		$$d_X(p,gp) \leq {2\ell_\xi}\left(\frac{d_\xi}{\varepsilon}+1\right)\big(d_X(p,ap)+d_X(q,aq)\big).$$
\end{cor}

\begin{proof}
Since $a$ is real hyperbolic, $\tau(a)>0$ and $\tau(a^k)=k\tau(a)$ for all $k \in \N$. Let $k$ to be the maximal positive integer such that $k\varepsilon < d_\xi$. In particular, maximality of $k$ implies that $d_X(p,a^{k+1}p),d_X(p,b^{k+1}p)\geq d_\xi$ and we are able to apply Lemma \ref{lemma:distance to MIN in symmetric space}, concluding that 
		$$d_X(p,\MIN{a})\leq 2\ell_\xi (k+1)d_X(p,ap), \ \ d_X(p,\MIN{b}) \leq 2 \ell_\xi (k+1) d_X(p,bp).$$
We can then apply Lemma \ref{prop:distance to flat bounded implies conjugator bounded}, taking $\ell(a)=\ell(b)=2\ell_\xi(\frac{d_\xi}{\varepsilon}+1)$. This gives the upper bound on the length of a conjugator $g$ as required.
\end{proof}

We complete this section with the following consequence of the above work. It says that if we restrict ourselves to looking at the majority of regular semisimple elements --- that is, those with not too small translation length and of a slope which is not too close to being singular --- then we can obtain a linear bound on the length of short conjugators.

\begin{cor}
For every $\varepsilon_1,\varepsilon_2 >0$ there exists $\kappa=\kappa(\varepsilon_1,\varepsilon_2)$ with the following property: assume that $a$ and $b$ are conjugate hyperbolic elements with translation lengths $\tau(a),\tau(b)\geq \varepsilon_1$ and slope $\xi \in \Dmod$ such that the spherical distance from $\xi$ to $\partial \Dmod$ is at least $\varepsilon_2$. Then there exists a conjugator  $g \in G$ such that:
		$$d_X(p,gp) \leq \kappa \big(d_X(p,ap)+d_X(p,bp)\big)\textrm{.}$$
\end{cor}

It is worth noting that while we have a precise expression for the constant $\ell_\xi$ in terms of the slope $\xi$, we have no grasp on the value taken by $d_\xi$.

\section{Dependence on the slope}\label{sec:semisimple:no uniform bound}

The aim here is to show that any constant satisfying the linear relationship in Theorem \ref{thm:bounded conj in G} must depend on the common slope of $a$ and $b$. To do this, we first show there is not a uniform constant $\ell > 0$ such that for all regular semisimple elements $a \in G$ the following holds:
\begin{equation}\label{eq:contradiction}
d_{X}(p,\MIN{a})\leq \ell d_{X}(p,ap)
\end{equation}
where $p$ is an arbitrary basepoint in $X$. This will imply that the constants $\ell(a)$ required for Proposition \ref{prop:distance to flat bounded implies conjugator bounded} will have to depend somehow on $a$.

To do this, we will construct a sequence of regular semisimple elements which contradict the existence of such an $\ell$. The sequence will converge to a singular semisimple element, agreeing with the intuition that the constant $\ell_{\xi}$ of Lemma \ref{lemma:distance to MIN in symmetric space} diverges to infinity as the slope $\xi$ converges to a singular direction.

We first note that if (\ref{eq:contradiction}) is true for a point $p \in X$ then it is true for every point $q \in X$. Indeed let $g \in G$ be any isometry such that $gp=q$. Then firstly:
\begin{eqnarray*}
d_{X}(q,\MIN{a}) & = & d_{X}(gp,\MIN{a})  	\\
				 & = & d_{X}(p,g^{-1}\MIN{a})	\\
				 & = & d_{X}(p,\MIN{g^{-1}ag})
\end{eqnarray*}
and secondly:
\begin{eqnarray*}
d_{X}(q,aq) & = & d_{X}(gp,agp) \\
			& = & d_{X}(p,g^{-1}agp)\textrm{.}
\end{eqnarray*}
But since we assume (\ref{eq:contradiction}) to be true for all hyperbolic elements it follows that it is true for $g^{-1}ag$ and thus:
$$
d_{X}(q,\MIN{a})\leq \ell d_{X}(q,aq)\textrm{.}
$$

Fix a pair of distinct flats $F,F'$ whose intersection is non-trivial and of dimension at least one. Let $p_0$ be a point in their intersection and let $\Lie{g=k\oplus p}$ be the corresponding Cartan decomposition. Let $\Lie{a,a'}$ be the maximal abelian subspaces of $\Lie{p}$ such that:
$$F=\exp(\Lie{a})p_0$$
$$F'=\exp(\Lie{a'})p_0$$
Furthermore suppose $H \in \Lie{a \cap a'}$ has length $1$ and $Y \in \Lie{a' \setminus a}$ is orthogonal to $H$ and is also of length $1$.

Let $(H_n)$ be a sequence of regular unit vectors in $\Lie{a \setminus a'}$ which converge to $H$. Define $a_0 :=\exp(H) \in G$ and $a_n := \exp(H_n) \in G$ for $n \in \N$ (see Figure \ref{fig:no uniform bound}).

\begin{figure}[h!]
\labellist
	\small\hair 5pt
		\pinlabel $p_0$ [t] at 84 74
		\pinlabel $q$ [r] at 73 177
		\pinlabel $a_0q$ [l] at 195 177
		\pinlabel $H$ [l] at 206 86
		\pinlabel $H_n$ [l] at 198 61
		\pinlabel $F$ at 302 109
		\pinlabel $F'$ at 260 210
\endlabellist

\centering
\includegraphics[width=11cm]{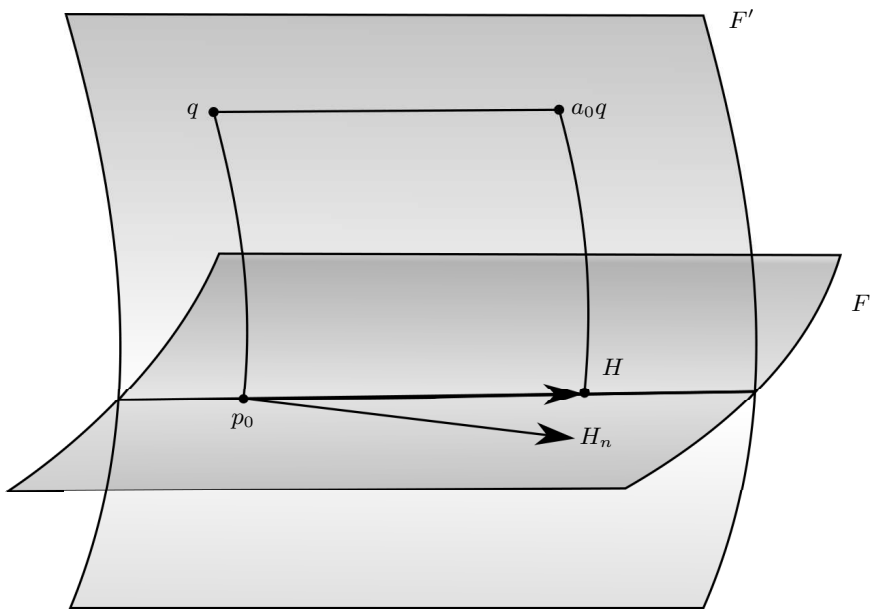}
\caption[Two intersecting flats]{The two flats $F$ and $F'$.}
\label{fig:no uniform bound}
\end{figure}

We suppose (\ref{eq:contradiction}) holds for some $\ell > 0$. Let $\varphi$ be the angle between $Y$ and the flat $F$ and set $q=\exp(\frac{2\ell}{\sin \varphi} Y)p_0$. Then 
$$
d_{X}(q,F)\geq \sin \varphi d_{X}(q,p_0)=2\ell\textrm{.}
$$
But for each $n \in \N$, by construction, $F=\MIN{a_n}$. Thus $d_{X}(q,\MIN{a_n})\geq 2\ell$ for each $n$. Meanwhile the sequence of points $a_n q$ converges to the point $a_0 q$ as $n$ tends to infinity. Hence $d_{X}(q,a_n q)$ converges to $1$. This gives the following contradiction:
\begin{eqnarray*}
2\ell & \leq & d_{X}(q,\MIN{a_n}) \\
	  & \leq & \ell d_{X}(q,a_{n}q) \\
	  & \rightarrow & \ell
\end{eqnarray*}
Hence there cannot exist such a constant $\ell$ which satisfies (\ref{eq:contradiction}) for every hyperbolic element in $G$.

The following Lemma explains how to use the above to demonstrate the non-existence of a uniform constant for the linear control on conjugacy length among all regular semisimple elements in $G$.

\begin{lemma}
Suppose for $p \in X$ there exists $\ell^\prime>0$ such that for every pair of conjugate regular semisimple elements $a,b$ in $G$ there exists a conjugator $g \in G$ such that
		$$d_X(p,gp) \leq \ell^\prime(d_X(p,ap)+d_X(p,bp))\textrm{.}$$
Then (\ref{eq:contradiction}) holds for $\ell = 2\ell^\prime$.
\end{lemma}

\begin{proof}
Let $a$ be regular semisimple in $G$ such that $\MIN{a}=F_a$ does not contain $p$. Let $\pi_a$ be the orthogonal projection of $X$ onto $F_a$ and let $m$ be the midpoint of the geodesic segment $[p,\pi_a(p)]$. Consider the geodesic symmetry $s_m$ of $X$ about $m$, that is the map $s_m:X\to X$ such that for any geodesic $c:\R \to X$ with $c(0)=m$, $s_m(c(t))=c(-t)$ for all $t \in \R$. Since $X$ is a symmetric space $s_m$ is an isometry. Let $F$ denote the flat $s_m(F_a)$. Then $p \in F$ and $[p,\pi_a(p)]$ meets both flats at right-angles, so this geodesic segment realises the distance between the flats.

Take $g \in G$ such that $gF=F_a$ and $gp=\pi_a(p)$. Then by the hypothesis of the Lemma, $d_X(p,\pi_a(p))\leq \ell^\prime(d_X(p,ap)+d_X(p,g^{-1}ag p))$. But $d_X(p,g^{-1}agp)$ is the translation length of $a$, so is less than $d_X(p,ap)$. Hence $d_X(p,F_a) \leq 2\ell^\prime d_X(p,ap)$.
\end{proof}

Since (\ref{eq:contradiction}) cannot hold, we have the following:

\begin{cor}
For $p \in X$ there does not exists $\ell^\prime>0$ such that for every pair of conjugate regular semisimple elements $a,b$ in $G$ there exists a conjugator $g \in G$ such that
		$$d_X(p,gp) \leq \ell^\prime(d_X(p,ap)+d_X(p,bp))\textrm{.}$$
\end{cor}

\section{Looking for a short conjugator in $\Gamma$}\label{sec:semisimple:bounded conj in Gamma}

In Corollary \ref{cor:bounded conj in G from lattice} we found a short conjugator between two real hyperbolic elements in $\Gamma$. However this conjugator lies in the ambient Lie group $G$. In order to improve our understanding of conjugacy length in $\Gamma$ we need to work out how to move our conjugator from $G$ so that it becomes a conjugator in $\Gamma$. The main obstacle here is in understanding how the lattice will intersect flats in the symmetric space. 

Given a conjugator $g\in G$ for $a,b \in \Gamma$, the set of all conjugators is the coset $Z_G(a)g$ of the centraliser of $a$. We are therefore interested in the contents of the set $Z_G(a)g \cap \Gamma$, or equivalently $gZ_G(b)\cap \Gamma$. If we begin with the assumption that a conjugator for $a,b$ from the lattice exists, then at least we know these sets are non-empty.

When $a$ and $b$ are real hyperbolic we have a good understanding of the geometry of their centralisers (see Lemmas \ref{lemma:regular centraliser and flat stabiliser} and \ref{lemma:singular centraliser and MIN}). For example, when $b$ is regular we can find a point $q \in X$ such that the orbit $Z_G(b)q$ is a maximal flat. Then $gZ_G(b)q$ is also a maximal flat and is in fact the unique maximal flat stabilised by $a$. Hence it is equal to $Z_G(a)gq$. Clearly $(Z_G(a)g \cap \Gamma)q$ is contained in this flat, and the question is how far is $gq$ from this subset? Once we know this distance, we can shift our conjugator $g$, whose length we have an estimate for courtesy of Section \ref{sec:semisimple:bounding in G}, to a conjugator in $\Gamma$ and keep track of the size of the new lattice conjugator.

Suppose $\gamma \in gZ_G(b)\cap \Gamma$. Then $gZ_G(b)\cap \Gamma = \gamma(Z_G(b)\cap \Gamma)$. It is therefore enough to look at how the set $(Z_G(b)\cap \Gamma)q$ sits inside $Z_G(b)q$. In the singular case $Z_G(b)q$ will be made up of a family of maximal flats. Each flat will be the orbit of a maximal torus $T$ contained in $Z_G(b)$. If there exists some such torus $T$ which satisfies the properties that $b \in T \cap \Gamma$ and $T \cap \Gamma$ is isomorphic to $\Z$, then we understand what the fundamental domain for the action of $T\cap \Gamma$ on $Tq$ will look like: it will be an $R$--tubular neighbourhood of a hyperplane orthogonal to the geodesics translated by $b$, where $2R \leq d_X(q,bq)$.

This situation cannot arise if the $\Q$--rank of the lattice is too small: it must satisfy $\qrank(\Gamma) \geq \rrank(G)-1$. If $\qrank(\Gamma)=\rrank(G)-1$ then a maximal $\Q$--split torus $S$ is a hyperplane inside a maximal $\R$--split torus $T$. Because $\Gamma$ must intersect $S$ in a finite set, $\Gamma$ will intersect $T$ in nothing more than a finite extension of $\Z$. In particular, if $q \in X$ is chosen so that $Tq$ is flat, then $(\Gamma \cap T)q$ will look like a copy of $\Z$ inside the flat. Furthermore, the fundamental domain for the action of $T\cap \Gamma$ on this flat will be a tubular neighbourhood of $Sq$.

In general, if $\qrank(\Gamma)=\rrank(G)-d$, then the flats in $X$ which have a non-trivial intersection with an orbit of $\Gamma$ can do so only with copies of $\Z^k$ for $\rrank(G)\geq k \geq d$.

\begin{lemma}
Suppose $\qrank(\Gamma) \geq \rrank(G)-1$ and let $T$ be a maximal $\R$--split torus in $G$ such that $T \cap \Gamma$ is a finite extension of $\Z$. Take a real hyperbolic element $b \in T\cap\Gamma$ and suppose it is conjugate in $\Gamma$ to $a$. If $a$ and $b$ have slope $\xi$, then there exists a conjugator $\gamma \in \Gamma$ for $a,b$ such that
		$$d_X(p,\gamma p) \leq (6L_\xi+1) (d_X(p,ap)+d_X(p,bp))$$
where $L_\xi$ is as in Corollary \ref{cor:bounded conj in G from lattice}.
\end{lemma}

\begin{proof}

\begin{figure}[h!]
   \labellist
      \small\hair 5pt
         \pinlabel $\gamma Sq$ [l] at 485 288
         \pinlabel $\gamma q$ [r] at 251 211
         \pinlabel $gq$ [b] at 401 270
         \pinlabel $gTq$ at 50 330
         \pinlabel $\gamma(T\cap\Gamma)q$ [l] at 380 33
         \endlabellist
                  \centering
    \includegraphics[width=8cm]{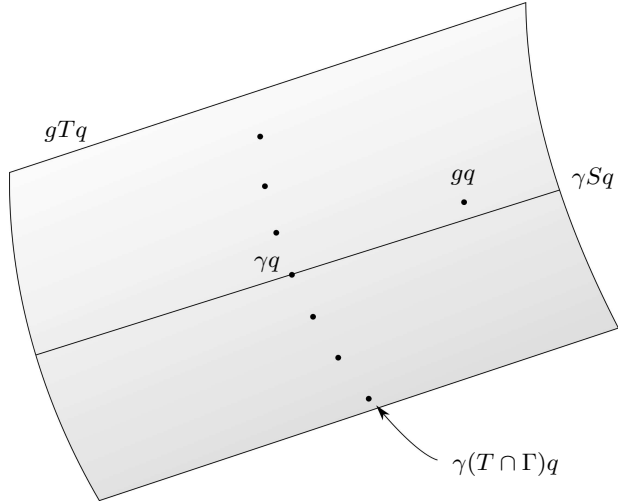}

     \caption[The intersection of a lattice with a flat when $\rrank(G)=2$]{When $\rrank(G)=2$, the quotient of the maximal flat $gTq$ by the action of $\Gamma \cap gTg^{-1}$ is a cylinder whose diameter is bounded above by $d_X(\gamma q, b\gamma q)=d_X(q,aq)$.}
\end{figure}

Suppose first that $\qrank(\Gamma)=\rrank(G)-1$ and let $S$ be a maximal $\Q$--split torus contained in $T$. To adapt the following to the case when $\qrank(\Gamma)=\rrank(G)$, we merely take $S^+$ to instead be the face of a Weyl chamber in $T$ (which will be isometric to a $\Q$--Weyl chamber, so $S^+$ will still isometrically embed into $\Gamma \backslash X$).

Let $q$ be any point in $X$ such that $Tq$ is a (maximal) flat. Note that $Tq$ will be contained in $\MIN{b}$. Choose $g \in G$ as in Theorem \ref{thm:bounded conj in G}, taking $p=q$. Then $g$ maps $Tq$ to a maximal flat contained in $\MIN{a}$ and in particular there exists $\gamma \in \Gamma$ such that $gq$ lies in an $R$--tubular neighbourhood of $\gamma S q$, where $2R=d_X(q,aq)$. In particular there is a $\Q$--Weyl chamber $S^+$ in $S$ such that $gq$ lies in an $R$--tubular neighbourhood of $\gamma S^+ q$.

The $\Q$--Weyl chamber maps isometrically into $\Gamma \backslash X$ (see \cite{Leuz04}), so in particular $d_X(\gamma q, g q) \leq d_{X}(\Gamma q , g q) + R$. Hence $d_X(q,\gamma q) \leq 2d_X(q,gq) + R$.

To finish, we need to translate it so it works for an arbitrary basepoint $p$. To do this we use Lemma \ref{lemma:distance to MIN in symmetric space} and apply the triangle inequality:
\begin{eqnarray*}
d_X(p,\gamma p) & \leq & 2d_X(p,q) + d_X(q,\gamma q) \\
				& \leq & 4L\xi d_X(p,ap) + 4L_\xi (d_X(q,aq)+d_X(q,bq)) + r \\
				& \leq & (6 L_\xi +1) (d_X(p,ap)+d_X(p,bp))
\end{eqnarray*}
Note that in the above we have assumed that $d_X(p,ap) \leq d_X(p,bp)$; if this is not the case, we can just reverse the roles of $a$ and $b$.
\end{proof}

The preceding Lemma works because we know enough about the shape and size of a fundamental domain for the action of $\Gamma\cap T$ on the flat $Tq$. In this case, $\Gamma q$ intersects $Tq$ in a copy of $\Z$, and we essentially have a tubular neighbourhood of a co-dimension $1$ flat, the radius of the neighbourhood bounded above by the translation length of $a$. 

\begin{question}
When $\Gamma q$ intersects $Tq$ in a copy of $\Z^k$, for some $k\geq 2$, what can we say about the dimensions of the fundamental domain for the action of $\Gamma \cap T$ on the $k$--dimensional flat inside $Tq$ stabilised by $\Gamma \cap T$ in terms of $d_X(p,bp)$ and $d_X(p,\MIN{b})$, where $b \in T\cap \Gamma$?
\end{question}

\section{Unipotent elements}\label{sec:unipotent}

We move on now to look at the conjugacy of unipotent elements in $G$. The result obtained here gives a linear bound on the length of a short conjugator from $G$ between two unipotent elements which satisfy a certain algebraic condition. Provided the two elements are, in some sense, pushed far enough away from certain root-spaces in $\Lie{g}$, the linear bound obtained will be uniform. The method used to prove it relies heavily on the Lie algebra and in particular on the root system corresponding to $\Lie{g}$.

Consider two conjugate elements $u,v$ in $N$. When looking at the case when $N$ is the subgroup of $\SLnR$ consisting of the unipotent upper triangular matrices, the condition we impose on $u$ and $v$ is equivalent to demanding that the super-diagonal entries in the matrix, that is the $(i,i+1)$--entries, are all not too close to zero. The short conjugator we obtain is built up by gradually knocking off entries in the matrix until you are left with a matrix with zeros above the super-diagonal. Doing this for both $u$ and $v$ gives two matrices that are then related via conjugation by a diagonal matrix. In Section \ref{sec:unipotent:outline} we give a few more details about how the process works for $\SLnR$.

In the process used to knock off the extra terms in the matrix, the super-diagonal entries play an important role and it is crucial that they are not too close to zero. The underlying root system for $\SLnR$ is of type $A_{n-1}$, and each positive root corresponds to a particular entry of the matrix. The simple roots correspond to the super-diagonal entries of the matrix. Hence we use the term ``simple case'' to describe the situation in which we insist that the super-diagonal, or simple, entries of $u$ and $v$ all avoid a neighbourhood of $0$.

The process works especially well in $\SLnR$. As mentioned in the introduction we consider only those semisimple real Lie groups whose Lie algebras are split. This is because we require that the root spaces $\Lie{g}_\lambda$ all have dimension $1$.

The ideas laid out below for the simple case could potentially be extended to a more general situation where the simple entries are allowed to be zero. First one should find the largest parabolic subgroup of $G$ which contains $u$ in its unipotent radical. When looking at the root system, this corresponds to taking bites out of it --- i.e.\ removing the linear spans of certain simple roots. We would then need to find a new subset of the set of positive roots which can play the role of the simple roots. Then conjugate by an element of the parabolic subgroup in order to ensure the corresponding entries are non-zero.

Let $N$ be a maximal unipotent subgroup of $G$. As a consequence of the root-space decomposition of the Lie algebra $\Lie{g}$, associated to $G$ is a reduced root system $\Lambda$, containing a subset $\Lambdaplus$ of positive roots, such that any element $u \in N=\exp (\Lie{n})$ can be uniquely expressed as
\begin{equation}\label{eq:u exp 2}
u=\exp \left( \sum_{\lambda \in \Lambdaplus}Y_\lambda \right)
\end{equation}
where each $Y_\lambda$ lies in a root-space $\Lie{g}_\lambda$ in $\Lie{g}$. Based on this, we introduce some terminology. For $u$ as above, the element $Y_\lambda$ will be called the $\lambda$--\emph{entry} of $u$. If $\lambda$ is a simple root in $\Lambdaplus$, that is $\lambda \in \Pi$, then we will say that $Y_\lambda$ is a \emph{simple entry}.

\subsection{Outline of the method}\label{sec:unipotent:outline}

The idea is to conjugate $u\in N$ by a sequence elements of the form $\exp(Z_\mu)$, where $Z_\mu \in \Lie{g}_\mu$, or by a commutator of two such elements (from two distinct root-spaces), each step in the sequence removing a $\lambda$--entry of $u$.

For example, when dealing with unipotent upper triangular matrices in $\SLnR$ each entry in the triangle above the diagonal corresponds to a root. The simple roots correspond to the super-diagonal entries, that is those which lie adjacent to the diagonal. Take
$$u:= \left( \begin{array}{cccc}
1  & x_{1}  & y_{1}   & z_{1} \\
0  &  1     & x_{2}   & y_{2} \\
0  &  0     &   1     & x_{3} \\
0  &  0     &   0     & 1
\end{array} \right).$$
Here the $x_i$ entries are the simple entries of $u$, the $y_i$ terms correspond to roots of height $2$ and $z_1$ to the unique root of height $3$. Suppose all these entries are non-zero. We will conjugate $u$ by an elementary matrix to make the $y_1$ term vanish:
$$\left( \begin{array}{cccc}
1  & \alpha  & 0   & 0 \\
0  &  1     & 0   & 0 \\
0  &  0     &   1     & 0 \\
0  &  0     &   0     & 1
\end{array} \right)
u
\left( \begin{array}{cccc}
1  & -\alpha  & 0   & 0 \\
0  &  1     & 0   & 0 \\
0  &  0     &   1     & 0 \\
0  &  0     &   0     & 1
\end{array} \right) = 
\left( \begin{array}{cccc}
1  & x_{1}  & y_{1} +\alpha x_2   & z_{1} + \alpha y_2 \\
0  &  1     & x_{2}   & y_{2} \\
0  &  0     &   1     & x_{3} \\
0  &  0     &   0     & 1
\end{array} \right).$$
So if we set $\alpha=-\frac{y_1}{x_2}$ then the entry where $y_1$ was has now been made to be zero. Notice that all simple entries and the other entry of height $2$ are unchanged by this conjugation --- the only collateral damage is to entries corresponding to roots of strictly greater height that the entry we removed.

The idea is to repeat this process, next removing the other height $2$ entry. This will again cause collateral damage, but it will similarly only effect the height $3$ entry. This then is the last entry to be removed and is done so by one last conjugation, but in this case there is no root of greater height than $3$, so there will be no collateral damage.

We have conjugated $u$ to 
$$u':= \left( \begin{array}{cccc}
1  & x_{1}  & 0   & 0 \\
0  &  1     & x_{2}   & 0 \\
0  &  0     &   1     & x_{3} \\
0  &  0     &   0     & 1
\end{array} \right)$$
via upper triangular matrices with rational entries. We do the same for another upper triangular matrix $v$, reducing to $v'$ in a similar manner. If $u$ and $v$ are conjugate in $\SLnZ$ then $u'$ and $v'$ must be conjugate in $\SLnQ$. In fact, when the simple entries in $u$ and $v$ are positive we can find a diagonal matrix over $\R$ to do the job:
$$\left( \begin{array}{cccc}
\alpha_1  & 0  & 0   & 0 \\
0  &  \alpha_2     & 0   & 0 \\
0  &  0     &   \alpha_3     & 0 \\
0  &  0     &   0     & \alpha_4
\end{array} \right)
u'
\left( \begin{array}{cccc}
\alpha_1  & 0  & 0   & 0 \\
0  &  \alpha_2     & 0   & 0 \\
0  &  0     &   \alpha_3   & 0 \\
0  &  0     &   0     & \alpha_4
\end{array} \right)^{-1} = 
\left( \begin{array}{cccc}
1  & w_{1}  & 0   & 0 \\
0  &  1     & w_{2}   & 0 \\
0  &  0     &   1     & w_{3} \\
0  &  0     &   0     & 1
\end{array} \right)$$
where $\alpha_1^4=\frac{w_1^3w_2^2w_3}{x_1^3x_2^2x_3}$, $\alpha_2^4=\frac{x_1 w_2^2 w_3}{w_1 x_2^3 x_3}$, $\alpha_3^4=\frac{x_1 x_2^2 w_3}{w_1 w_2^2 x_3}$, and $\alpha^4 = \frac{x_1 x_2^2 x_3^3}{w_1 w_2^2 w_3^3}$.

From our point of view, the crucial aspect of this process is that we can keep track of the size of the conjugator in each step and also control the extent of the collateral damage occurring to entries of greater height.

\subsection{The metric on $G$}

The Killing form on $\Lie{g}$ is the symmetric bilinear form $B:\Lie{g \times g} \to \R$ given by $B(V,W)=\mathrm{Trace}(\ad(V)\ad(W))$. Given $\Lie{g=k\oplus p}$, the Cartan decomposition at $p \in X$, the Killing form is negative definite on $\Lie{k}$ and positive definite on $\Lie{p}$. Furthermore $\Lie{p}$ is the orthogonal complement of $\Lie{k}$ with respect to the Killing form. Let $\theta_p$ be the Cartan involution on $\Lie{g}$ defined at $p$, that is $\theta_p$ acts on $\Lie{k}$ as the identity and $\theta_p(Y)=-Y$ for any $Y \in \Lie{p}$. We can then define an inner product on $\Lie{g}$ as follows (see \cite[\S 2.7]{Eber96} or \cite[Ch. III Prop 7.4]{Helg01}):
		$$\IPp(Y,Z)=-B(\theta_p Y, Z), \textrm{ for all }Y,Z \in \Lie{g}.$$
This then determines a left-invariant Riemannian metric $d_G$ on $G$. 
Denote the norm on $\Lie{g}$ corresponding to $\IPp$ by $\norm{.}$.

We will be interested in the effect of the Lie bracket on the size of elements from the root-spaces of $\Lie{g}$.

\begin{prop}\label{prop:Lie product size}
Suppose $\Lie{g}$ is a split real Lie algebra. Let $Y_\lambda \in \Lie{g}_\lambda$ and $Y_\mu \in \Lie{g}_\mu$, where $\lambda \in \Pi$ and $\mu \in \Lambdaplus$ and $\lambda+\mu$ is a root. Then there exist constants $c_1\geq c_0 > 0$, independent of the choice of $Y_\lambda,Y_\mu,\lambda,\mu$, such that
$$c_1 \norm{Y_\lambda} \norm{Y_\mu} \geq \norm{[Y_\mu,Y_\lambda]} \geq c_0 \norm{Y_\lambda} \norm{Y_\mu}\textrm{.}$$
\end{prop}

\begin{proof}
This follows from the fact that the root-spaces have dimension one and also from the bilinearity of the Lie bracket and of the inner product $\IPp$. In particular, if we let $Z_\lambda$ denote one of the two elements of $\Lie{g}_\lambda$ such that $\norm{Z_\lambda}=1$, and similarly for $Z_\mu$, then for  $\alpha,\beta \in \R$ such that $Y_\lambda =\alpha Z_\lambda$ and $Y_\mu = \beta Z_\mu$:
		$$\norm{[Y_\lambda,Y_\mu]}=\modulus{\alpha}\modulus{\beta}c_{\lambda,\mu}$$
where $c_{\lambda,\mu}=\norm{[Z_\lambda,Z_\mu]}$. By taking
\begin{eqnarray*}
&c_0 = \min \{c_{\lambda,\mu} \mid \mu \in \Lambdaplus, \lambda \in \Pi  \textrm{ such that } \lambda+\mu \in \Lambdaplus \}& \\
&c_1 = \max \{c_{\lambda,\mu} \mid \mu \in \Lambdaplus, \lambda \in \Pi  \textrm{ such that } \lambda+\mu \in \Lambdaplus \}&
\end{eqnarray*}
we obtain the result, since $\modulus{\alpha}=\norm{Y_\lambda}$ and $\modulus{\beta}=\norm{Y_\mu}$.
\end{proof}

The following tells us that the size of any $\lambda$--entry of $u$ will give us a lower bound for the size of $d_G(1,u)$.

\begin{lemma}\label{lemma:unipotent entry bounded by u}
Let $u \in N$ be as in (\ref{eq:u exp 2}). For each $\lambda \in \Lambdaplus$ we have $\norm{Y_\lambda} \leq d_G (1,u)$.
\end{lemma}

\begin{proof}
Since the root-spaces $\Lie{g}_\lambda$ for $\lambda \in \Lambdaplus$ are pairwise orthogonal with respect to the Killing form, and hence also the inner product $\IPp$, we observe that:
	$$\norm{Y_\lambda} \leq \left\|\sum_{\lambda \in \Lambda}Y_\lambda \right\| = d_G(1,u)
	\textrm{.}$$
\end{proof}

\section{The role of the root system}\label{sec:unipotent:root system}

\subsection{Relating the root system to conjugation}

If $N$ is a maximal unipotent subgroup with corresponding root system $\Lambda$ then any element in $N$ can be written uniquely as 
\begin{equation}\label{eq:u exp}
u=\exp \left( \sum_{\lambda \in \Lambdaplus} Y_{\lambda} \right)
\end{equation}
where $Y_\lambda \in \Lie{g}_\lambda$. We begin by studying the behaviour of the $\lambda$--entries under the action of conjugation by elements in $N$ of the form $\exp(Z_\mu)$, where $Z_\mu \in \Lie{g}_\mu$, or a commutator of two such elements.

\begin{lemma}\label{lem:conjugating by mu}
Let $u \in N$ be as in (\ref{eq:u exp}) and let $Z_{\mu} \in \Lie{g}_{\mu}$ for some $\mu \in \Lambdaplus$. When we conjugate $u$ by $\exp (Z_{\mu})$ all entries of $u$ are unchanged except (possibly) for the $\lambda$--entries where $\lambda = r\mu + \lambda'$ for some $r \in \mathbb{N}$ ($r \neq 0$) and $\lambda' \in \Lambdaplus$ such that $Y_{\lambda'} \neq 0$.
Furthermore, the $\lambda$--entry of $\exp(Z_\mu)u\exp(-Z_\mu)$ is
$$
\sum_{r\mu+\lambda'=\lambda} \frac{(\ad Z_\mu)^r Y_{\lambda'}}{r!}
$$
where the sum takes values of $r$ from $\N \cup \{ 0 \}$ and $\lambda'$ from $\Lambdaplus$.
\end{lemma}

\begin{proof}
Let $Z=Z_\mu$. We observe that:
\begin{eqnarray*}
\exp (Z) u \exp(-Z) & = & \exp\left( e^{\ad (Z)} \sum_{\lambda' \in \Lambdaplus} Y_{\lambda'} \right) \\
                    & = & \exp \left( \sum_{r=0}^{\infty} \sum_{\lambda' \in \Lambdaplus} \frac{( \ad Z)^{r} Y_{\lambda'}}{r!} \right)
\end{eqnarray*}
Recall that if $r\mu + \lambda'$ is not a root then $(\ad Z)^{r}Y_{\lambda'}=0$. Otherwise $(\ad Z)^{r}Y_{\lambda'} \in \Lie{g}_{r\mu+\lambda'}$. It follows that if $\lambda$ cannot be written as $r\mu + \lambda'$ for any $r \neq 0$ or any $\lambda'$ then the $\lambda$--entry of $u$ in unchanged by this conjugation.
\end{proof}

The preceding Proposition is important in recognising the link between conjugation of unipotent elements and the root system of $G$. In particular we can see that if the $\lambda$--entry of $u$ is affected by conjugating by $\exp (Z_{\mu})$ then the height of $\lambda$, denoted $\hgt{\lambda}$, must be greater than $\hgt{\mu}$. Furthermore, the affected entries whose height is precisely $\hgt{\mu}+1$ will be in the set $\lbrace \mu \rbrace + \Pi$, where $\Pi$ is the set of simple roots in $\Lambdaplus$. This is crucial for motivating Lemma \ref{lemma:simple unipotent}.

To complete the picture which lies behind the scenes of Lemma \ref{lemma:simple unipotent} we must also consider conjugating by a commutator of two elements. Building up to this, which is Lemma \ref{lem:conjugating by commutator}, we give the following:

\begin{lemma}\label{lem:conjugating by product}
Let $u \in N$ be as in (\ref{eq:u exp}) and let $Z_{1} \in \Lie{g}_{\mu_{1}}$, $Z_{2} \in \Lie{g}_{\mu_{2}}$ for some $\mu_{1} , \mu_{2} \in \Lambdaplus$. When we conjugate $u$ by $\exp(Z_{1})\exp(Z_{2})$ all entries of $u$ are unchanged except (possibly) for the $\lambda$--entries where $\lambda = r\mu_{1} + t\mu_{2} + \lambda'$ for some $\lambda' \in \Lambdaplus$ and non-negative integers $r,t$ where at least one of $r,t$ is non-zero.
\end{lemma}

\begin{proof}
As in the proof of Lemma \ref{lem:conjugating by mu} we get:
$$
\exp(Z_{2})\exp(Z_{1})u\exp(-Z_{1})\exp(-Z_{2}) = \sum_{t=0}^{\infty}\sum_{r=0}^{\infty}\sum_{\lambda'\in\Lambdaplus} \frac{(\ad Z_{2})^{t}(\ad Z_{1})^{r}}{r!t!}Y_{\lambda'}
$$
Since $(\ad Z_{2})^{t}(\ad Z_{1})^{r}Y_{\lambda'} \in \Lie{g}_{r\mu_{1}+t\mu_{2}+\lambda'}$ if $r\mu_{1} + t\mu_{2} + \lambda'$ is a root, or is zero otherwise,  it follows as in Lemma \ref{lem:conjugating by mu} that if $\lambda$ cannot be expressed as $r\mu_{1} + t\mu_{2} + \lambda'$ for some $\lambda'$ and non-negative integers $r,t$ where one of $r,t$ is non-zero, then the $\lambda$--entry of $u$ in not affected by this conjugation process.
\end{proof}

Rather than conjugating by $\exp(Z_{1})\exp(Z_{2})$ we will conjugate by their commutator $\left[ \exp(Z_{1}) \exp(Z_{2}) \right]$. The extra terms in the product act to clean up any effect conjugating by $\exp(Z_{1})\exp(Z_{2})$ had on the entries of height less than or equal to $\hgt{\mu_{1}}+\hgt{\mu_{2}}$. Observe that the $\lambda$--entry of the conjugate of $u$ by $\exp(Z_{1})\exp(Z_{2})$ is:
$$\sum_{r\mu_{1}+t\mu_{2} + \lambda' = \lambda}  \frac{(\ad Z_{2})^{t}(\ad Z_{1})^{r}}{r!t!}Y_{\lambda'}\textrm{.}
$$
Suppose $\hgt{\lambda} \leq \hgt{\mu_{1}}+\hgt{\mu_{2}}$. Then in each term in the sum either $r=0$ or $t=0$. We can therefore rewrite it as:
$$\sum_{r\mu_{1}+\lambda'=\lambda}\frac{(\ad Z_{1})^{r}}{r!}Y_{\lambda'} + \sum_{t\mu_{2}+\lambda'=\lambda}\frac{(\ad Z_{2})^{t}}{t!}Y_{\lambda'}-Y_{\lambda}\textrm{.}
$$
The extra $-Y_{\lambda}$ term is needed because when $r=t=0$ we count $Y_{\lambda}$ twice when it should only be counted once.
Next we conjugate by $\exp(-Z_{1})\exp(-Z_{2})$ and we get the following for the $\lambda$--entry:
\begin{eqnarray*}
\sum_{R\mu_{1}+\lambda'=\lambda}\frac{(\ad (-Z_{1}))^{R}}{R!}\left(\sum_{r\mu_{1}+\lambda''=\lambda'}\frac{(\ad Z_{1})^{r}}{r!}Y_{\lambda''} + \sum_{t\mu_{2}+\lambda''=\lambda'}\frac{(\ad Z_{2})^{t}}{t!}Y_{\lambda''} -Y_{\lambda'}\right) \\
 + \sum_{T\mu_{2}+\lambda'=\lambda}\frac{(\ad (-Z_{2}))^{T}}{T!}\left(\sum_{r\mu_{1}+\lambda''=\lambda'}\frac{(\ad Z_{1})^{r}}{r!}Y_{\lambda''} + \sum_{t\mu_{2}+\lambda''=\lambda'}\frac{(\ad Z_{2})^{t}}{t!}Y_{\lambda''}-Y_{\lambda'}\right) - Y_{\lambda}
\end{eqnarray*}
Since $\hgt \lambda \leq \hgt{\mu_1}+\hgt{\mu_2}$, we cannot write $\lambda = R\mu_{1}+t\mu_{2} + \lambda'$ when both $R,t$ are non-zero (and similarly for $r$ and $T$). Hence this expression can be reduced to:
\begin{eqnarray*}
\sum_{R\mu_{1}+\lambda'=\lambda}\frac{(\ad (-Z_{1}))^{R}}{R!}\left(\sum_{r\mu_{1}+\lambda''=\lambda'}\frac{(\ad Z_{1})^{r}}{r!}Y_{\lambda''}\right) \\
 + \sum_{T\mu_{2}+\lambda'=\lambda}\frac{(\ad (-Z_{2}))^{T}}{T!}\left(\sum_{t\mu_{2}+\lambda''=\lambda'}\frac{(\ad Z_{2})^{t}}{t!}Y_{\lambda''}\right)-Y_{\lambda}
\end{eqnarray*}
This can be rewritten as:
\begin{eqnarray*}
\sum_{R\mu_{1}+\lambda'=\lambda}\left(\sum_{r\mu_{1}+\lambda''=\lambda'}\frac{(-1)^{R}}{R!r!}(\ad Z_{1})^{R+r}Y_{\lambda''}\right) \\
 + \sum_{T\mu_{2}+\lambda'=\lambda}\left(\sum_{t\mu_{2}+\lambda''=\lambda'}\frac{(-1)^{T}}{T!t!}(\ad Z_{2})^{T+t}{t!}Y_{\lambda''}\right)-Y_{\lambda}
\end{eqnarray*}
Notice that whenever $R+r \neq 0$ all the terms cancel, since if $R+r=k\neq 0$ then the coefficient of $(\ad Z_1)^k Y_{\lambda''}$ is:
$$\sum_{R+r=k}\frac{(-1)^{R}}{R!r!}=0
\textrm{.}$$

A similar statement holds for $T+t\neq 0$. Hence, whenever  $\hgt \lambda \leq \hgt{\mu_1}+\hgt{\mu_2}$, the $\lambda$--entry is $Y_\lambda$. We use this in the following:

\begin{lemma}\label{lem:conjugating by commutator}
Let $u \in N$ be as in (\ref{eq:u exp}) and let $Z_{1} \in \Lie{g}_{\mu_{1}}$, $Z_{2} \in \Lie{g}_{\mu_{2}}$ for some $\mu_{1} , \mu_{2} \in \Lambdaplus$. When we conjugate $u$ by $[\exp(Z_{1}),\exp(Z_{2})]$ all entries of $u$ are unchanged except (possibly) for the $\lambda$--entries where $\lambda = r\mu_{1} + t\mu_{2} + \lambda'$ for some $\lambda' \in \Lambdaplus$ and non-negative integers $r,t$.	

Furthermore, for such $\lambda$, the $\lambda$--entry of the conjugate is
$$\sum_{(r+t)\mu_1 + (s+u)\mu_2 + \lambda' = \lambda} \frac{\ad(-Z_2)^u \ad(-Z_1)^t \ad (Z_2)^s \ad (Z_1)^r Y_\lambda'}{r!s!t!u!}
$$
where the summation takes non-negative integers $r,s,t,u$ and positive roots $\lambda'$.
\end{lemma}

\begin{proof}
By repeating Lemma \ref{lem:conjugating by mu} we get that the $\lambda$--entry of the conjugate of $u$ by $[\exp(Z_{1}),\exp(Z_{2})]$ is given by
$$\sum_{(r+t)\mu_1 + (s+u)\mu_2 + \lambda' = \lambda} \frac{\ad(-Z_2)^u \ad(-Z_1)^t \ad (Z_2)^s \ad (Z_1)^r Y_\lambda'}{r!s!t!u!}
$$
as required.

This, together with the argument preceding the statement of the Proposition, gives the result.
\end{proof}

The important difference between Lemma \ref{lem:conjugating by commutator} and Lemma \ref{lem:conjugating by product} is that when we conjugate by the commutator all entries of height no more than $\hgt{\mu_{1}}+\hgt{\mu_{2}}$ are left unchanged. Furthermore the only (possibly) affected entries of height $\hgt{\mu_{1}}+\hgt{\mu_{2}} +1$ are precisely those entries corresponding to roots in the set $\lbrace \mu_{1} \rbrace + \lbrace \mu_{2} \rbrace + \Pi$ where $\Pi$ is the set of simple roots in $\Lambdaplus$.

\subsection{An ordering on the root system}\label{sec:simple case}

We consider in this paper the ``simple case,'' by which we mean the case when $u$ is given by
$$ u = \exp \left( \sum_{\lambda \in \Lambdaplus} Y_\lambda\right)
$$
and $Y_\lambda \neq 0$ for each simple root $\lambda$. The aim is to find a sequence of elements like those considered in Lemmas \ref{lem:conjugating by mu} and \ref{lem:conjugating by commutator} which reduce $u$ to a form where the only non-zero $\lambda$--entries are those where $\lambda$ is simple. The following Lemma is necessary to ensure that such a sequence of elements can be found in the simple case.

\begin{lemma}\label{lemma:simple unipotent}
Let $\Lambda^{+}$ be a set of positive roots and $\Pi$ the corresponding simple roots associated to a reduced root system $\Lambda$. We can assign to $\Lambdaplus$ an ordering, which we will denote by $<$, such that for every $\lambda \in \Lambdaplus \setminus \Pi$ either:
\begin{enumerate}[label=(\alph*)]
\item \label{lemma:simple unipotent mu} there exists some root $\mu$ such the set $\lbrace \mu \rbrace + \Pi$ contains $\lambda$ and $\lambda \neq \lambda' \in \{ \mu \} + \Pi$ implies $\lambda<\lambda'$; or
\item \label{lemma:simple unipotent mu1 mu2} there exist roots $\mu_{1} , \mu_{2} \in \Lambdaplus$ such that  $\lbrace \mu_{1} \rbrace + \lbrace \mu_{2} \rbrace + \Pi = \lbrace \lambda \rbrace$ and $\mu_1 + \mu_2$ is not a root.
\end{enumerate}
\end{lemma}

\begin{remark}
Case \ref{lemma:simple unipotent mu} corresponds to conjugation by something in $\exp (\Lie{g}_{\mu})$, see Lemma \ref{lem:conjugating by mu}. Case \ref{lemma:simple unipotent mu1 mu2} corresponds to conjugating by a commutator as in Lemma \ref{lem:conjugating by commutator}. This Lemma, combined with Lemmas \ref{lem:conjugating by mu} and \ref{lem:conjugating by commutator}, tells us that we can always conjugate $u \in N$ by an element of $N$ in such a way that we can choose the smallest entry of $u$ which is affected by the conjugation.
\end{remark}

\begin{proof}[Proof of Lemma \ref{lemma:simple unipotent}] Before we proceed, note that if we find $\mu_1,\mu_2$ satisfying \ref{lemma:simple unipotent mu1 mu2} but $\mu_1+\mu_2$ is a root, then case \ref{lemma:simple unipotent mu} also applies.

Suppose that $\Lambda$ is the sum of irreducible root systems $\Lambda_1 , \ldots , \Lambda_r$ and that $\Lambdaplus = \Lambdaplus_1 \cup \ldots \cup \Lambdaplus_r$. Suppose also that on each $\Lambdaplus_i$ we have an ordering $<_i$ which satisfies the Lemma. Then we can define an ordering $<$ on $\Lambdaplus$ given by $\lambda < \mu$ if and only if
\begin{enumerate}
\item $\lambda \in \Lambdaplus_i$ and $\mu \in \Lambdaplus_j$ such that $i<j$; or
\item if $\lambda,\mu \in \Lambdaplus_i$ for some $i$ then $\lambda <_i \mu$.
\end{enumerate}
If $\lambda$ and $\mu$ are in different irreducible root systems inside $\Lambda$, then $\lambda+\mu$ cannot be a root. Hence it follows that if $<_i$ satisfies the Lemma for each $i$, then so does $<$. Thus it suffices to check the conditions of the Lemma for each irreducible root system.

In the classical root systems $A_n,B_n,C_n,D_n$, we make the base assumption that the simple roots are $\Pi=\{\lambda_1,\ldots ,\lambda_n\}$ and are ordered by $\lambda_1>\lambda_2>\ldots>\lambda_n$. Note that because we will assume that the Lie algebra is split we do not need to consider $BC_n$ root systems.

\scheading{Root systems of type $A_{n}$}
This is the root system associated to $\SLnZ$ so we expect this to be straightforward. The non-simple positive roots will be sums of consecutive simple roots:
		$$\lambda_{i} + \lambda_{i+1} + \ldots + \lambda_{j}$$
for $1\leq i<j \leq n$. The ordering we assign is in two steps: primarily we order by height, then within each height we order the elements lexicographically. So if $\lambda = \lambda_{i} + \ldots + \lambda_{j}$ then we take $\mu = \lambda_{i} + \ldots + \lambda_{j-1}$. It follows that:
$$
\lbrace \mu \rbrace + \Pi = \left\{ \begin{array}{ll}
  \lbrace\lambda,\lambda_{i-1}+\ldots+\lambda_{j-1}\rbrace &\mbox{ if $i\neq 1$} \\
  \lbrace\lambda\rbrace &\mbox{ if $i=1$}
       \end{array} \right.
$$
and hence our chosen ordering satisfies the requirements of the lemma.

\scheading{Root systems of type $B_{n}$}
The non-simple positive roots are of the following forms:
$$
\begin{array}{ll}
\lambda_{i} + \ldots + \lambda_{j} & \textrm{for }1\leq i<j\leq n \\
\lambda_{i} + \ldots + \lambda_{j-1} + 2\lambda_{j} + \ldots + 2\lambda_{n} & \textrm{for }1 \leq i<j \leq n
\end{array}
$$
We order the roots as we did for type $A_{n}$: first by height, then order the elements of each height lexicographically. If we first take $\lambda$ of the first form listed above, i.e.\ $\lambda = \lambda_{i} + \ldots + \lambda_{j}$. Then we take $\mu = \lambda_{i} + \ldots + \lambda_{j-1}$ and observe that:
$$
\lbrace \mu \rbrace + \Pi = \left\{ \begin{array}{ll}
  \lbrace\lambda,\lambda_{i-1}+\ldots+\lambda_{j-1}\rbrace &\mbox{ if $i\neq 1$} \\
  \lbrace\lambda\rbrace &\mbox{ if $i=1$}
       \end{array} \right.
$$
satisfies the required conditions. If on the other hand we consider 
$$\lambda = \lambda_{i} + \ldots + \lambda_{j-1} + 2\lambda_{j} + \ldots + 2\lambda_{n}$$
then we take
$$
\mu = \left\{ \begin{array}{ll}
   \lambda_{i} + \ldots + \lambda_{j} + 2\lambda_{j+1} + \ldots + 2\lambda_{n} & \textrm{if } j \neq n \\
   \lambda_{i} + \ldots + \lambda_{n} & \textrm{if }j=n
   \end{array} \right.
$$
and observe that:
$$
\lbrace \mu \rbrace + \Pi = \left\{ \begin{array}{ll}
  \lbrace\lambda, \lambda_{i-1} + \ldots + \lambda_{j} + 2\lambda_{j+1} + \ldots + 2\lambda_{n} \rbrace &\mbox{ if $i\neq 1,j\neq n$} \\
  \lbrace \lambda , \lambda_{i-1} + \ldots + \lambda_{n} \rbrace &\mbox{ if $i\neq 1,j= n$} \\
  \lbrace\lambda\rbrace &\mbox{ if $i=1$}
       \end{array} \right.
$$
satisfies the requirements for every choice of $i,j$.

\scheading{Root systems of type $C_{n}$}
The positive non-simple roots have one of the following forms:
		$$
		\begin{array}{ll}
		\lambda_{i} + \ldots + \lambda_{j} & \textrm{for }1\leq i<j\leq n \\
		2\lambda_{i} + \ldots + 2\lambda_{n-1} + \lambda_{n} & \textrm{for }1 \leq i \leq n-1 \\
		\lambda_{i} + \ldots + \lambda_{j-1} + 2\lambda_{j} + \ldots + 2\lambda_{n-1} + \lambda_{n} & \textrm{for }1 \leq i<j \leq n-1
		\end{array}
		$$
We order these first by height, then order the elements of the same height by lexicographic ordering. We now give the choice for $\mu$ in each case.

First, for $1\leq i < j \leq n$, let $\lambda=\lambda_{i} + \ldots + \lambda_{j}$. Then we take $\mu = \lambda_{i} + \ldots + \lambda_{j-1}$ and we have:
$$
\lbrace \mu \rbrace + \Pi = \left\{ \begin{array}{ll}
   \lbrace \lambda , \lambda_{i-1} + \ldots + \lambda_{j-1} \rbrace & \mbox{ if $i \neq 1$} \\
   \lbrace \lambda \rbrace & \mbox{ if $i = 1$}
      \end{array} \right.
$$
Under our chosen ordering these satisfy the requirements of the Lemma.

Second, for $1\leq i \leq n-1$, let $\lambda=2\lambda_{i} + \ldots + 2\lambda_{n-1} + \lambda_{n}$. Then we take $\mu=\lambda_{i} + 2\lambda_{i+1} + \ldots + 2\lambda_{n-1} + \lambda_{n}$ if $i \neq n-1$ or $\mu = \lambda_{n-1} + \lambda_{n}$ if $i=n-1$ and we have:
$$
\lbrace \mu \rbrace + \Pi = \left\{ \begin{array}{ll}
   \lbrace \lambda , \lambda_{i-1} + \lambda_{i} + 2\lambda_{i+1} \ldots + 2\lambda_{n-1} + \lambda_{n} \rbrace & \mbox{ if $i \neq n-1$ and $i \neq 1$} \\
   \lbrace \lambda, \lambda_{n-2}+\lambda_{n-1}+\lambda_{n} \rbrace & \mbox{ if $i = n-1$}\\
   \lbrace \lambda \rbrace & \mbox{ if $i = 1$}
      \end{array} \right.
$$
In each case the elements of $\lbrace \mu \rbrace + \Pi$ are at least as big as $\lambda$ in our chosen ordering, so the Lemma is satisfied in this case.

Finally, for $1 \leq i < j \leq n-1$, let $\lambda = \lambda_{i} + \ldots + \lambda_{j-1} + 2\lambda_{j} + \ldots + 2\lambda_{n-1} + \lambda_{n}$. We take $\mu = \lambda_{i} + \ldots + \lambda_{j} + 2\lambda_{j+1} + \ldots + 2\lambda_{n-1} + \lambda_{n}$ if $j\neq n-1$ or $\mu = \lambda_{i} + \ldots + \lambda_{n}$ when $j=n-1$. Then:
$$
\lbrace \mu \rbrace + \Pi = \left\{ \begin{array}{ll}
   \lbrace \lambda , \lambda_{i-1} + \ldots + \lambda_{j} + 2\lambda_{j+1} \ldots + 2\lambda_{n-1} + \lambda_{n} \rbrace & \mbox{ if $j \neq n-1$ and $i \neq 1$} \\
   \lbrace \lambda, \lambda_{i-1}+ \ldots+\lambda_{n} \rbrace & \mbox{ if $j = n-1$ and $i \neq 1$}\\
   \lbrace \lambda \rbrace & \mbox{ if $i=1$}
      \end{array} \right.
$$
The requirements of the Lemma are satisfied in each case, and hence it follows that the Lemma holds for root systems of type $C_{n}$.

\scheading{Root systems of type $D_{n}$}
The non-simple positive roots in the root system $D_{n}$ are of one of the following two types:
$$
\begin{array}{ll}
   \lambda_{i} + \ldots + \lambda_{j-1} & \mbox{ if $1\leq i < j \leq n$} \\
   \lambda_{i} + \ldots + \lambda_{n-2} + \lambda_{j} + \ldots + \lambda_{n} & \mbox{ if $ 1 \leq i < j \leq n$}
\end{array}
$$
Apply the same ordering to $D_{n}$ as we applied to each of the preceding root systems: first order by height, then order the elements of the same height lexicographically. In most instances we are able to satisfy the conditions of the Lemma by choosing a single $\mu \in \Lambdaplus$. However there are some for which we must use the second allowable case, namely find two positive roots $\mu_{1} , \mu_{2}$ to satisfy the Lemma.

We first suppose $\lambda = \lambda_{i}+\ldots +\lambda_{j-1}$ where $1 \leq i < j < n$. Then we take $\mu = \lambda_{i} + \ldots \lambda_{j-2}$ and observe:
$$
\lbrace \mu \rbrace + \Pi = \left\{ \begin{array}{ll}
   \lbrace \lambda , \lambda_{i-1} + \ldots + \lambda_{j-2} \rbrace & \mbox{ if $i \neq 1$}\\
   \lbrace \lambda \rbrace & \mbox{ if $i=1$} 

\end{array} \right.
$$
Hence the conditions of the Lemma are satisfied in each case.

For $1\leq i<j \leq n$ let $\lambda = \lambda_{i}+\ldots+\lambda_{n-2}+\lambda_{j}+\ldots + \lambda_{n}$. First assume $i \neq n-2$ and $j \neq n-1,n$. If we take $\mu = \lambda_{i} + \ldots + \lambda_{n-2} + \lambda_{j+1} + \ldots + \lambda_{n}$ then:
$$
\lbrace \mu \rbrace + \Pi = \left\{ \begin{array}{ll}
   \lbrace \lambda , \lambda_{i-1} + \ldots + \lambda_{n-2} + \lambda_{j+1} + \ldots + \lambda_{n} & \mbox{ if $i \neq 1$}\\
   \lbrace \lambda \rbrace & \mbox{ if $i=1$}
\end{array} \right.
$$
and the Lemma is satisfied.

Now suppose $j=n$, then $\lambda = \lambda_{i} + \ldots + \lambda_{n-2} + \lambda_{n}$. Take $\mu = \lambda_{i} + \ldots + \lambda_{n-2}$ then:
$$
\lbrace \mu \rbrace + \Pi = \left\{ \begin{array}{ll}
   \lbrace \lambda , \lambda_{i-1} + \ldots + \lambda_{n-2}, \lambda_{i} + \ldots +\lambda_{n-1} \rbrace & \mbox{ if $i \neq 1$}\\
   \lbrace \lambda , \lambda_{1} + \ldots + \lambda_{n-1} \rbrace & \mbox{ if $i=1$}
\end{array} \right.
$$
and our choice of $\mu$ here satisfies the requirements of the Lemma.

We are left with the cases when $\lambda = \lambda_{i} + \ldots + \lambda_{n-1}$ and when $\lambda = \lambda_{n-2} + \lambda_{n-1}+ \lambda_{n}$. In the former case we take $\mu_{1} = \lambda_{i} + \ldots + \lambda_{n-3}$ and $\mu_{2} = \lambda_{n-1}$ and observe the only way to make a root by adding $\mu_{1}, \mu_{2}$ and a simple root together is if the simple root is $\lambda_{n-2}$, thus giving $\lambda$. Hence $\lbrace \mu_{1} \rbrace + \lbrace \mu_{2} \rbrace + \Pi = \lbrace \lambda \rbrace$. In the latter case we take $\mu_{1} = \lambda_{n-1}$ and $\mu_{2} = \lambda_{n}$. Similarly, since the only simple root which we can add to $\mu_{1} + \mu_{2}$ and still have a root is $\lambda_{n-2}$, we have $\lbrace \mu_{1} \rbrace + \lbrace \mu_{2} \rbrace + \Pi = \lbrace \lambda \rbrace$.

This completes the verification of the Lemma in the case when the root system is of type $D_{n}$.

\scheading{Root systems of type $E_{6}, E_{7}, E_{8},F_4,G_2$}
These are dealt with in the appendix. For root systems $E_8$ and $F_4$ a table is produced with an example of an ordering satisfying the Lemma. They also give suitable choices of $\mu$ or of $\mu_1$ and $\mu_2$ for each non-simple positive root. Table \ref{table:E_8} gives the ordering for $E_8$, and hence for $E_7$ and $E_6$ by using the induced ordering. Table \ref{table:F_4} gives the ordering for $F_4$. Figures \ref{fig:E_8}, \ref{fig:F_4} and \ref{fig:G_2} provide a visual method of checking in each case that the given root $\mu$ satisfies the requirements: given $\mu \in \Lambdaplus$ one can quickly see what $\lbrace \mu \rbrace + \Pi$ will be by following all edges heading down the page from $\mu$ to the row below.

When dealing with $G_2$, there is only one root of each height strictly greater than $1$, hence we can order the roots by height alone.
\end{proof}

\section{Construction of a short conjugator}\label{sec:unipotent:constructing a conjugator}

\subsection{Reduction of the simple case}

From here on in we will assume that $\Lie{g}$ is a split real Lie algebra, meaning that the root spaces $\Lie{g}_\lambda$ are $1$--dimensional. We first give an algorithm to reduce $u \in N$, all of whose simple entries are non-zero, to  $u' \in N$, all of whose non-simple entries are zero and the simple entries of $u'$ are equal to those of $u$. Write $u$ in terms of the elements from the root-spaces of $\Lie{g}$:
		$$u = \exp \left( \sum_{\lambda \in \Lambdaplus} Y_\lambda \right)$$
where $Y_\lambda \in \Lie{g}_\lambda$ for each $\lambda \in \Lambdaplus$. Assign to $\Lambdaplus$ the ordering from Lemma \ref{lemma:simple unipotent}. The algorithm is based on an iteration of the following result:

\begin{lemma}\label{lemma:simple one step}
Let $\lambda_0$ be the smallest non-simple root such that $Y_{\lambda_0}$ is non-zero. Then there exists $g \in N$ and a positive constant $c_0 >0$ such that:
\begin{enumerate}[label=(\roman*)]
		\item \label{lemma:simple one step exists} the $\lambda_0$--entry of $gug^{-1}$ is zero and all entries corresponding to smaller roots are unchanged; and
		\item \label{lemma:simple one step size} $d_G (1,g) \leq \frac{\norm{Y_{\lambda_0}}}{c_0\delta}$, where $\delta = \min \lbrace \norm{Y_{\lambda_i}} \mid \lambda_i \in \Pi \rbrace$. 
\end{enumerate} 
\end{lemma}

\begin{proof}
We begin by applying Lemma \ref{lemma:simple unipotent} to $\lambda_0$. This gives us either:
\begin{enumerate}[label=(\alph*)]
		\item\label{pf:simple one step mu} $\mu \in \Lambdaplus$ such that $\lambda_0$ is minimal in $\{\mu \} + \Pi$; or
		\item\label{pf:simple one step mu1 mu2} $\mu_1 , \mu_2 \in \Lambdaplus$ such that $\{ \mu_1\} + \{\mu_2 \} + \Pi = \{\lambda_0\}$ and $\mu_1+\mu_2$ is not a root.
\end{enumerate}
First suppose \ref{pf:simple one step mu} holds. Take $g = \exp \left( Z_\mu \right)$ where $Z_\mu \in \Lie{g}_\mu$ is chosen so that
$$\left[ Z_\mu , Y_{\lambda_i} \right] = -Y_{\lambda_0}
$$
where $\lambda_i$ is the simple root such that $\mu + \lambda_i = \lambda_0$. By Lemma \ref{lem:conjugating by mu}, the $\lambda_0$--entry of $gug^{-1}$ is, by construction,
$$Y_{\lambda_0} + \ad(Z_\mu)Y_{\lambda_i}=0
$$
and the other affected entries are of the form $r\mu+\lambda$ for some $\lambda \in \Lambdaplus$. All of these are larger than $\lambda_0$ in the ordering from Lemma \ref{lemma:simple unipotent}, hence the first part of the lemma is proved when case \ref{pf:simple one step mu} holds.

Now suppose that instead case \ref{pf:simple one step mu1 mu2} holds. Then we take $g = \left[ \exp(Z_1) , \exp(Z_2) \right]$ where $Z_i \in \Lie{g}_{\mu_i}$ for $i = 1,2$. By Lemma \ref{lem:conjugating by commutator}, the $\lambda_0$--entry of $gug^{-1}$ is
$$\sum_{(r+t)\mu_1 + (s+u)\mu_2 + \lambda' = \lambda_0} \frac{\ad(-Z_2)^u \ad(-Z_1)^t \ad (Z_2)^s \ad (Z_1)^r Y_\lambda'}{r!s!t!u!}
$$
where the summation takes non-negative integers $r,s,t,u$ and positive roots $\lambda'$. Since $Y_\lambda=0$ for non-simple roots $\lambda < \lambda_0$, there is no other way to obtain a non-zero term in the sum except by either taking $r=s=t=u=0$ and $\lambda'=\lambda_0$ or with $\lambda' = \lambda_i \in \Pi$ such that $\mu_1 + \mu_2 + \lambda_i = \lambda_0$. In the latter case we know $r+t = 1 = s+u$. Hence there are only finitely many combinations to consider and the $\lambda_0$ entry becomes:
\begin{eqnarray*}
& \ad (Z_2) \ad (Z_1) Y_{\lambda_i} + \ad (-Z_2) \ad (Z_1) Y_{\lambda_i} +\\
& \ad (-Z_1) \ad (Z_2) Y_{\lambda_i} + \ad (-Z_2) \ad (-Z_1) Y_{\lambda_i} + Y_{\lambda_{0}}
\end{eqnarray*}
which simplifies to
$$\ad (Z_2) \ad (Z_1) Y_{\lambda_i} - \ad (Z_1) \ad (Z_2) Y_{\lambda_i} + Y_{\lambda_0}
\textrm{.}$$
Finally, by application of the Jacobi identity, we see this is equal to
$$[[Z_2,Z_1],Y_{\lambda_i}]+Y_{\lambda_0}
\textrm{.}$$
Hence, by choosing $Z_1$ and $Z_2$ so that $\left[[Z_2 , Z_1 ], Y_{\lambda_i}\right]=-Y_{\lambda_0}$, the $\lambda_0$--entry of $gug^{-1}$ is zero.

Finally, Lemma \ref{lem:conjugating by commutator} tells us that entries corresponding to roots of height less than or equal to $\hgt{\mu_1}+\hgt{\mu_2}$ are unchanged. Since also $\{\mu_1\}+\{\mu_2\}+\Pi=\{\lambda_0\}$, all entries corresponding to roots smaller than $\lambda_0$ are unaffected. Thus we have proved \ref{lemma:simple one step exists}.

Note that we have the flexibility to choose $Z_\mu , Z_1$ and $Z_2$ as above because each root-space has dimension one so we only need to choose the appropriate scalar multiple of a basis element to get what we want.

Now we look at the size of $g$. If $g=\exp (Z_\mu )$ arises from a situation like \ref{pf:simple one step mu} then, since we chose $Z_\mu$ to satisfy $\left[ Z_\mu , Y_{\lambda_i} \right] = -Y_{\lambda_0}$, we can use Proposition \ref{prop:Lie product size} to show:
		$$d_G (1,g)  =  \norm{Z_\mu} \leq  \frac{\norm{Y_{\lambda_0}}}{c_0\norm{Y_{\lambda_i}}}  \leq  \frac{\norm{Y_{\lambda_0}}}{c_0\delta}.$$
Suppose instead that $g=[\exp (Z_1),\exp(Z_2)]$, as is necessary for case \ref{pf:simple one step mu1 mu2}. Using the Baker--Campbell--Hausdorff formula, $g=\exp([Z_1,Z_2])$ since $\mu_1+\mu_2$ is a not a root. Then, again using Proposition \ref{prop:Lie product size} and our choice of $Z_1,Z_2$ such that $\left[[Z_2 , Z_1 ], Y_{\lambda_i}\right]=-Y_{\lambda_0}$, we see that:
		$$ d_G(1,g) =  \norm{[Z_1,Z_2]} \leq \frac{\norm{Y_{\lambda_0}}}{c_0\norm{ Y_{\lambda_i}}} \leq \frac{\norm{Y_{\lambda_0}}}{c_0\delta}.$$
This completes \ref{lemma:simple one step size}.
\end{proof}

The following algorithm describes a process by which, in the simple case, we can reduce $u \in N$ to $u' \in N$, where $u'$ has no non-simple entries.

\begin{algorithm}\label{alg:reduction to superdiagonal simple case}
Let $u \in N$ be given by
		$$u = \exp \left( \sum_{\lambda \in \Lambdaplus} Y_\lambda \right), \ \ Y_\lambda \neq 0 \textrm{ for } \lambda \in \Pi.$$
We define a sequence of elements $u(i) \in N$ where $u(0)=u$ and $u(i+1)$ has one fewer non-zero non-simple entry than $u(i)$ and is obtained by
		$$u(i):=g(i)u(i-1)g(i)^{-1}, \textrm{ for $i \geq 1$}$$
where $g(i)$ is determined by Lemma \ref{lemma:simple one step}. This process clearly terminates as $\Lambdaplus$ is a finite set. Let $g(1),\ldots,g(r)$ be the complete set of conjugators obtained. Define $g := g(r) \ldots g(1)$. Then $u':=u(r)=gug^{-1}$, which has no non-zero non-simple entries.
\end{algorithm}

\subsection{The collateral damage of Algorithm \ref{alg:reduction to superdiagonal simple case}}

Suppose that the simple entries of $u$ are bounded away from zero. In particular, define a function $\Delta : N \to [0,\infty)$ by
		$$\Delta(u)=\min \{\norm{Y_{\lambda_i}}\mid \lambda_i \in \Pi \}$$
and suppose there exists some $\delta >0$ such that $\Delta(u) \geq \delta$. Note that $\Delta$ can be extended to all unipotent elements of $G$. This function measures, in some vague sense, the distance of $u$ from the simple root-spaces of $\Lie{g}$.

Before determining the size of a short conjugator in $G$ we need to determine the effect each step of Algorithm \ref{alg:reduction to superdiagonal simple case} has on the entries of $u$. This is a notion we described in Section \ref{sec:unipotent:outline} as collateral damage. We showed in Lemmas \ref{lem:conjugating by mu} and \ref{lem:conjugating by commutator} that while removing the $\lambda_0$ entry of $u$ it was possible that some of the entries of greater height could be altered in the process. We will call those entries affected by one of the steps of Algorithm \ref{alg:reduction to superdiagonal simple case}, other than the intended target entry, the \emph{collateral damage} of this step.

In general we expect collateral damage. We can, nonetheless, use an iterative method, bounding the size of each $u(i)$ in the sequence. By applying Lemmas \ref{lemma:simple one step} and \ref{lemma:unipotent entry bounded by u} we see that the first conjugator $g(1)$ will satisfy
		\begin{equation}
		\label{eq:collateral damage first step} d_G(1,g(1))\leq \frac{1}{c_0 \delta}d_G(1,u).
		\end{equation}
The collateral damage of conjugating $u$ by $g(1)$ includes elements of height greater than that of the smallest non-simple non-zero entry of $u$. Suppose $g(1),\ldots ,g(t_1)$ correspond to the steps to remove all entries of height $2$. Since conjugating by any of these will not effect any height $2$ entry of $u$, each $g(i)$, for $1 \leq i \leq t_1$, will satisfy inequality (\ref{eq:collateral damage first step}) in place of $g(1)$. Let $g_{\mathrm{ht}}(2)=g(t_1)\ldots g(1)$. Then
		$$d_G(1,g_{\mathrm{ht}}(2))\leq \frac{R_2}{c_0 \delta}d_G(1,u)$$
where $R_2$ is equal to the number of roots of height $2$. After the first $t_1$ steps of Algorithm \ref{alg:reduction to superdiagonal simple case} we obtain an element $u_{\mathrm{ht}}(2)=g_{\mathrm{ht}}(2)ug_{\mathrm{ht}}(2)^{-1}$ whose entries of height $2$ are all zero. Furthermore, by the triangle inequality
		$$d_G(1,u_{\mathrm{ht}}(2)) \leq \left( \frac{2R_2}{c_0 \delta} + 1\right) d_G(1,u).$$
Suppose the $\lambda$--entry of $u_{\mathrm{ht}}(2)$ is $Y_\lambda^{(2)}$. Then by Lemma \ref{lemma:unipotent entry bounded by u}
		$$\norm{Y_\lambda^{(2)}} \leq \left( \frac{2R_2}{c_0 \delta} + 1\right) d_G(1,u).$$
By Lemma \ref{lemma:simple one step}, the size of the next conjugator will be bounded above:
		$$d_G(1,g(t_1 +1))\leq \frac{1}{c_0 \delta}\left( \frac{2R_2}{c_0 \delta} + 1\right) d_G(1,u)$$
noting that we can still use $\delta$ as defined above since the simple entries of $u_{\mathrm{ht}}(2)$ are exactly those of $u$.  Let $g_{\mathrm{ht}}(3)=g(t_2)\ldots g(t_1 +1)$, where $g(t_1+1),\ldots ,g(t_2)$ are those conjugators from Algorithm \ref{alg:reduction to superdiagonal simple case} corresponding to the removal of height $3$ entries of $u$. Then, as in the height $2$ case, we get
		$$d_G(1,g_\mathrm{ht}(3) )\leq \frac{R_3}{c_0 \delta}\left( \frac{2R_2}{c_0 \delta} + 1\right) d_G(1,u)$$
where $R_3$ is the number of roots of height $2$. Then $u_\mathrm{ht}(3)=g_\mathrm{ht}(3)u_\mathrm{ht}(2)g_\mathrm{ht}(3)^{-1}$ has no entries of height $2$ or $3$, and it satisfies
		$$d_G(1,u_{\mathrm{ht}}(3)) \leq \left(\frac{2R_3}{c_0 \delta} +1 \right)\left( \frac{2R_2}{c_0 \delta} + 1\right) d_G(1,u).$$
Continuing in this way, if $r$ is the greatest height of a root in $\Lambdaplus$, then for each $2\leq i \leq r$ we have
		$$d_G(1,g_\mathrm{ht}(i)) \leq \frac{R_i}{c_0\delta}\prod\limits_{j=2}^{i-1} \left(\frac{2R_j}{c_0\delta}+1\right) d_G(1,u).$$
Let $g=g_\mathrm{ht}(r) \ldots g_\mathrm{ht}(2)$. Then $g$ is the element obtained from Algorithm \ref{alg:reduction to superdiagonal simple case} and conjugates $u$ to an element $u'$ whose non-simple entries are all zero, while its simple entries are the same as for $u$. Finally, we see that the size of $g$ is bounded linearly by the size of $u$:
\begin{prop}\label{prop:simple to super-diagonal conjugator is bounded}
Let $g$ be the conjugator obtained by Algorithm \ref{alg:reduction to superdiagonal simple case} such that the non-simple entries of $gug^{-1}$ are all zero. Then
		$$d_G(1,g) \leq K(\delta) d_G(1,u)$$
where
		$$K(\delta)= \sum\limits_{i=2}^r \frac{R_i}{c_0\delta}\prod\limits_{j=2}^{i-1} \left(\frac{2R_j}{c_0\delta}+1\right).$$
\end{prop}
\subsection{The last step towards finding a short conjugator}

Take $u$ as above and let $v$ be an element in $N$ conjugate to $u$. Suppose we can express $v$ as
		$$\exp \left(\sum_{\lambda \in \Lambdaplus} W_\lambda\right), \ \ W_\lambda \in \Lie{g}_\lambda.$$
By applying Algorithm \ref{alg:reduction to superdiagonal simple case} we may assume that $Y_\lambda = 0 = W_\lambda$ for all non-simple roots $\lambda \in \Lambdaplus$. Then, by choosing $H \in \Lie{a}$ appropriately, we can conjugate $u$ to $v$ using $\exp (H)$. To be precise:
\begin{eqnarray*}
gug^{-1} & = & \exp(H) \exp \left( \sum_{\lambda \in \Pi} Y_\lambda \right)\exp(-H)\\
		 & = & \exp \left( \sum_{\lambda \in \Pi} e^{\lambda(H)}Y_\lambda \right).
\end{eqnarray*}
Hence our choice of $H$ needs to be such that $e^{\lambda(H)}Y_\lambda = W_\lambda$. We might ask, what if we need negative scalars? The following Proposition answers this question:

\begin{prop}\label{prop:existence of diagonal conjugator}
Let $u$ and $v$ be unipotent elements contained in the same maximal unipotent subgroup $N$ of $G$. Suppose that $u$ is conjugate to $v$ in $G$ and furthermore suppose that the non-simple entries of $u$ and $v$ are all trivial while the simple entries are all non-zero. Then there exists $H_0 \in \Lie{a}$ such that 
$$\exp (H_0)u\exp(-H_0) = v
\textrm{.}$$
\end{prop}

\begin{proof}
Let $g \in G$ be such that $gug^{-1}=v$. First observe that, since we are dealing with the simple case, both $u$ and $v$ fix the same unique chamber $\bdry{\mathcal{C}}$ in the ideal boundary of $X$ and belong to the same minimal parabolic subgroup $G_\xi=Z_\xi N_\xi$, where $Z_\xi=Z_G(A)$, $A=\exp (\Lie{a})$ and $N_\xi = N$. Any conjugator from $u$ to $v$ must map $\bdry{C}$ to itself, hence $g \in G_\xi$ as well. We may therefore write $g$ as $g=a'n$ where $a' \in Z_G(A)$ and $n \in N$.

Since $n \in N$ it follows that we may write $n$ as
$$n = \exp \left( \sum_{\lambda \in \Lambdaplus}Z_\lambda\right)
$$
where $Z_\lambda \in \Lie{g}_\lambda$. When we conjugate $u$ by $n$ we get the following:
	\begin{eqnarray*}
	nun^{-1}	& = &	\exp \left( \sum_{\lambda \in \Lambdaplus} Z_\lambda \right) \exp\left(\sum_{\lambda \in \Pi} Y_\lambda \right) \exp \left( -\sum_{\lambda \in \Lambdaplus} Z_\lambda \right)\\
				& = &	\exp \left( \sum_{\lambda \in \Pi} Y_\lambda + \tilde{Y} \right)
	\end{eqnarray*}
where $\tilde{Y}$ is the sum of elements $\tilde{Y}_\lambda$ from the non-simple positive root-spaces. Let $\Lie{a}$ be the maximal abelian subspace of $\Lie{p}$ such that $A = \exp (\Lie{a})$. The exponential map, when restricted to $Z_\Lie{g}(\Lie{a})$, is surjective onto $Z_G(A)$. So there exists $H' \in Z_\Lie{g}(\Lie{a})$ such that $a'=\exp(H')$.
We can decompose $Z_\Lie{g}(\Lie{a})$ into the direct sum (see, for example, \cite[2.17.10]{Eber96})
	$$Z_\Lie{g}(\Lie{a})=\Lie{k}\cap Z_\Lie{g}(\Lie{a}) \oplus \Lie{a}
	\textrm{.}$$
Hence there exists unique $U \in \Lie{k} \cap Z_\Lie{g}(\Lie{a})$ and $H \in \Lie{a}$ such that $H'=U+H$. Since $U$ and $H$ commute, $a'=\exp(U)\exp(H)=\exp(H)\exp(U)$. Conjugating $nun^{-1}$ by $\exp(H)$ gives us
	\begin{eqnarray*}
	\exp(H)nun^{-1}\exp(-H)	&=&	\exp (H) \exp \left( \sum_{\lambda \in \Pi} Y_\lambda + \sum_{\lambda \in \Lambdaplus \setminus \Pi} \tilde{Y}_\lambda \right) \exp(-H)\\
							&=& \exp \left( \sum_{\lambda \in \Pi} e^{\lambda(H)}Y_\lambda + \sum_{\lambda \in \Lambdaplus \setminus \Pi} e^{\lambda(H)}\tilde{Y}_\lambda \right).
	\end{eqnarray*}
Conjugating this by $\exp(U)$ gives us $v$ as
	$$v=\exp \left( \sum_{\lambda \in \Pi} e^{\ad (U)}e^{\lambda(H)}Y_\lambda + \sum_{\lambda \in \Lambdaplus \setminus \Pi} e^{\ad (U)}e^{\lambda(H)}\tilde{Y}_\lambda \right)
	\textrm{.}$$
Notice that, since $U \in Z_\Lie{g}(\Lie{a})$, for each $\lambda \in \Lambdaplus \setminus \Pi$ the term $e^{\ad(U)}e^{\lambda(H)}\tilde{Y}_\lambda$ is in the root-space $\Lie{g}_\lambda$. But the exponentional map gives a bijection between $\Lie{n}$ and $N$. Hence
	$$\sum_{\lambda \in \Pi} W_\lambda = \sum_{\lambda \in \Pi}e^{\ad(U)}e^{\lambda(H)}Y_\lambda + \sum_{\lambda \in \Lambdaplus \setminus \Pi}e^{\ad(U)}e^{\lambda(H)}\tilde{Y}_\lambda
\textrm{.}$$
It follows that $W_\lambda = e^{\ad(U)}e^{\lambda(H)}Y_\lambda$ for each simple root $\lambda$ and $0=e^{\ad(U)}e^{\lambda(H)}\tilde{Y}_\lambda$ when $\lambda$ is non-simple. Thus $\tilde{Y}=0$ and in particular
	$$nun^{-1}=u
	\textrm{.}$$
It follows that
	$$v=gug^{-1}=a'nun^{-1}a'^{-1}=a'ua'^{-1}
	\textrm{.}$$
In order to finish the proof we find an element $H_0 \in \Lie{a}$ to do the required job. Let $C_\lambda(U) \in \R$ be such that $[U,Y_\lambda]=C_\lambda(U) Y_\lambda$. Then $e^{\ad(U)}Y_\lambda = e^{C_\lambda(U)}Y_\lambda$ and in particular we see that there exists a positive constant $C_\lambda=e^{C_\lambda(U)+\lambda(H)}$ for each simple root $\lambda$ such that
	$$W_\lambda = C_\lambda Y_\lambda
	\textrm{.}$$
Now we notice that in $\Lie{a}$ we have sufficient degrees of freedom to choose $H_0 \in \Lie{a}$ such that $\lambda(H_0)=C_\lambda$ for each $\lambda \in \Pi$. Then $H_0$ is the required element to complete the proof.
 \end{proof}

\begin{remark}
Note that to the existence of the constants $C_\lambda(U)$ required the dimension of each simple root-space in $\Lie{g}$ to be equal to $1$. So Proposition \ref{prop:existence of diagonal conjugator} requires $\Lie{g}$ to be split.
\end{remark}

\subsection{The short conjugators}

Let $u,v$ be unipotent elements contained in the same maximal unipotent subgroup $N$ of $G$, both of which have all simple entries non-zero. By Algorithm \ref{alg:reduction to superdiagonal simple case} we can construct $g_1$ and $g_2$ in $N$ such that all non-simple entries in $u'=g_1 u g_1^{-1}$ and $v'=g_2 v g_2^{-1}$ are zero. By Proposition \ref{prop:existence of diagonal conjugator} there exists $g_3 \in A$ such that $g_3 u' g_3^{-1} = v'$. Put $g = g_2^{-1} g_3 g_1$. Then
	$$gug^{-1}=v
	\textrm{.}$$
With this process we can find a short conjugator for $u$ and $v$.

\begin{thm}\label{thm:unipotent clf}
Fix $\delta>0$. Let $N$ be a maximal unipotent subgroup of $G$. Consider two conjugate unipotent elements 
		$$u=\exp \left( \sum_{\lambda \in \Lambdaplus} Y_\lambda \right) , \ \ v=\exp \left( \sum_{\lambda \in \Lambdaplus} W_\lambda \right) \ \in \ N$$
such that $\min \{\Delta(u),\Delta(v)\}\geq \delta$. Then there exists $g \in G$ such that $gug^{-1}=v$ and which satisfies:
	$$d_G(1,g) \leq L(\delta)( d_G(1,u) + d_G(1,v))
	$$
where $L(\delta)$ will depend on $\delta$ and on the root-system $\Lambda$ associated to $G$.
\end{thm}

\begin{proof}
Recall that $g=g_2^{-1}g_3g_1$ with $g_2$ and $g_1$ as in Algorithm \ref{alg:reduction to superdiagonal simple case}. By Proposition \ref{prop:simple to super-diagonal conjugator is bounded}
		$$d_G(1,g_1)+d_G(1,g_2) \leq K(d_G(1,u)+d_G(1,v))$$
where $K$ depends on $\Lambda$, $c_0$ and $\delta$.
All we need to do now is obtain a linear upper bound for the size of $g_3$. By Proposition \ref{prop:existence of diagonal conjugator} this is member of $A$, equal to $\exp (H)$ for some $H \in \Lie{a}$, which satisfies the following for each simple root $\lambda$:
	\begin{equation}\label{eq:lambda(H)}
	e^{\lambda(H)}=\frac{\norm{W_\lambda}}{\norm{Y_\lambda}}
	\end{equation}
where $Y_\lambda$ is the $\lambda$--entry of $u$ and $W_\lambda$ is the $\lambda$--entry of $v$. The size $d_G(1,g_3)$ is given by the norm of $H$, which is equal to the Killing form 
	$$B(H,H)=\textrm{Trace}(\ad (H)^2)=\sum_{\lambda \in \Lambda}\lambda(H)^2
	\textrm{.}$$
Since every root in $\Lambda$ can be expressed as an integer linear combination of simple roots, it follows that there exists a constant $S_\Lambda$ such that when we take the sum over only the simple roots, rather than all positive roots, we get:
	\begin{equation}\label{eq:norm(H)}
	\sum_{\lambda \in \Pi}\lambda(H)^2 \leq \norm{H}=B(H,H)\leq S_\Lambda \sum_{\lambda \in \Pi} \lambda(H)^2
	\textrm{.}
	\end{equation}
By combining (\ref{eq:lambda(H)}) and (\ref{eq:norm(H)}) we get
	\begin{eqnarray*}
	d_G(1,g_3)	&=&		\norm{H}\\
				&\leq&	S_\Lambda \sum_{\lambda \in \Pi} \lambda(H)^2\\
				&=&		S_\Lambda \sum_{\lambda \in \Pi} \left( \ln \norm{W_\lambda} - \ln \norm{Y_\lambda}\right)^2\\
				&=&		S_\Lambda \sum_{\lambda \in \Pi} \left( \ln \norm{W_\lambda} \right)^2 + \left( \ln \norm{Y_\lambda} \right)^2 - \ln\norm{W_\lambda}\ln\norm{Y_\lambda}\\
				&\leq&	S_\Lambda \sum_{\lambda \in \Pi} \left( \ln d_G(1,v) \right)^2 + \left( \ln d_G(1,u) \right)^2 - 2 \ln(\delta)^2
	\end{eqnarray*}
This is therefore sufficient to conclude that the size of $g$, for sufficiently large $u,v$, is bounded above by a linear function of $d_G(1,u)+d_G(1,v)$, the coefficient of which will depend on $\delta$, $K(\delta)$ and $S_\Lambda$. This completes the proof.
\end{proof}

It is well known that the maximal unipotent subgroups in $G$ form one conjugacy class. Furthermore, if we fix a maximal compact subgroup $K$, then given any pair of maximal unipotent subgroups $N_1$ and $N_2$ there exists $k \in K$ such that $kN_1k^{-1}=N_2$. This gives us the following consequence of Theorem \ref{thm:unipotent clf}:

\begin{thm}\label{thm:unipotent clf 2}
For every $\delta>0$ there exists a constant $\hat{L}(\delta)$ such that, if $u$ and $v$ are unipotent elements in $G$ satisfying $\min\{\Delta(u),\Delta(v)\} \geq \delta$, then $u$ is conjugate to $v$ if and only if there exists some $g \in G$ such that $gug^{-1}=v$ and
		$$d_G(1,g)\leq \hat{L}(\delta)\big( d_G(1,u)+d_G(1,v)\big).$$
\end{thm}

\subsection{Application to lattices}

The condition that $\Delta(u)$ and $\Delta(v)$ must be sufficiently far away from zero is a stronger property than saying they must avoid a neighbourhood of the identity. Nonetheless, with the following Lemma we can use Theorem \ref{thm:unipotent clf 2} to deduce a result for lattices.

\begin{lemma}\label{lemma:unipotent simple entries lower bound}
Let $u \in N$ be as in (\ref{eq:u exp 2}). Then there exists $\delta > 0$ such that if $u \in \Gamma$ then for each simple root $\lambda_i \in \Pi$ either $\norm{Y_{\lambda_i}} \geq \delta$ or $Y_{\lambda_i}=0$.
\end{lemma}

\begin{proof}
Since $\Gamma \cap N$ is a discrete subgroup of $N$ we know it is finitely generated (see, for example, Corollary 2 of Theorem 2.10 in \cite{Ragh72}). Let $\{\gamma_1 , \ldots , \gamma_r\}$ be a set of generators for $\Gamma \cap N$ and let $\gamma = \gamma_{i_1}^{\varepsilon_1} \ldots \gamma_{i_s}^{\varepsilon_s}\in \Gamma \cap N$ where $i_j \in \{1,\ldots , r\}$ and $\varepsilon_j \in \Z \setminus \{0\}$. We can write each generator as
$$\gamma_i = \exp \left(\sum_{\lambda \in \Lambdaplus}Y_{\lambda}^{(i)}\right)
$$
where $Y_{\lambda}^{(i)} \in \Lie{g}_\lambda$ for each $i$ and each $\lambda$.
Then, by using the Campbell--Baker--Hausdorff formula,
		$$\gamma = \exp \left( \sum_{\lambda \in \Lambdaplus} \sum_{j=1}^s \varepsilon_j Y_\lambda^{(i_j)} + \tilde{Y} \right)$$
where $\tilde{Y}$ is a sum of terms from non-simple root-spaces. This tells us that each simple entry $Y_{\lambda_i}$ of $u$ belongs to the integer linear span of the set $\{Y_{\lambda_i}^{(1)},\ldots , Y_{\lambda_i}^{(r)}\}$, hence there is an element of minimal length for each simple root which can appear as an entry of an element in $\Gamma \cap N$. By taking the shortest of these lengths we obtain a positive value for $\delta$.
\end{proof}

\begin{cor}
Let $\Gamma$ be a lattice in $G$. Then there exists a constant $L>0$ such that two unipotent elements $u$ and $v$ in $\Gamma$ with non-zero simple entries are conjugate in $G$ if and only if there exists a conjugator $g \in G$ such that
		$$d_G(1,g) \leq L \big( d_G(1,u)+d_G(1,v)\big).$$
\end{cor}

\newpage
\appendix
\section{Tables and Figures for Lemma \ref{lemma:simple unipotent}}

\begin{figure}[b!]
   \labellist
    \small\hair 5pt

   \endlabellist

   \centering
   \includegraphics[width=5cm]{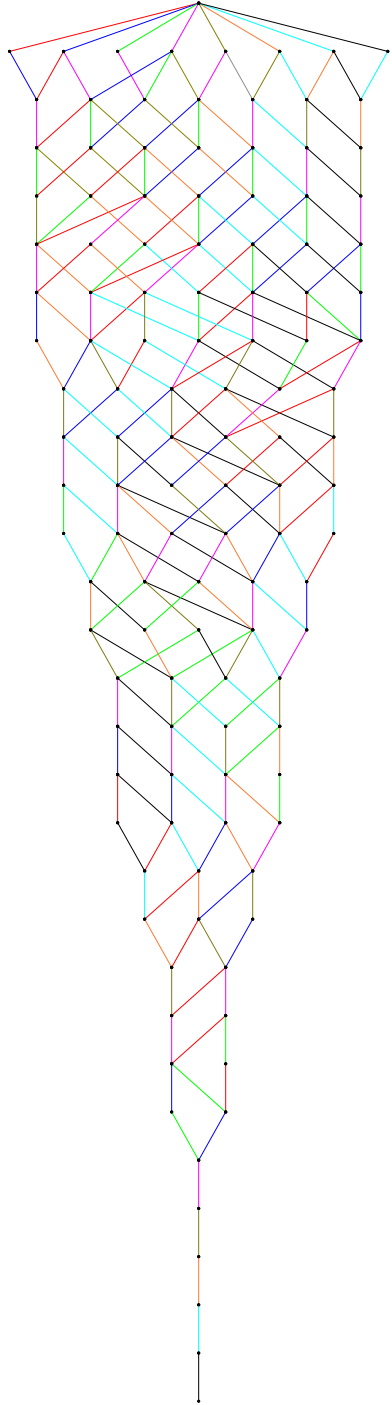}
   \caption[The root system $E_8$]{A graphical depiction of the positive roots in $E_{8}$. The vertices correspond to positive roots (the top vertex is $0$), while the edges correspond to addition of a simple root, when reading downwards. Each root has its own colour.}
   \label{fig:E_8}
\end{figure}

\begin{longtable}{|l|l|c|c|c|}
\caption{The simple case for root systems of type $E_{8}$}\label{table:E_8}\\
\hline \multicolumn{1}{|l|}{Height}& \multicolumn{1}{l|}{{Order}} &\multicolumn{1}{c|}{$\lambda$} &\multicolumn{1}{c|}{$\mu${ or }$\mu_{1}$} &\multicolumn{1}{c|}{$\mu_{2}${ (if needed)}}\\  \hline \hline \endhead
2	& 1 &$	\lambda_{1} + \lambda_{2}							$&$ \lambda_{2}									$	& \\ \hline
	& 2 &$	\lambda_{2} + \lambda_{4} 							$&$ \lambda_{4} 								$		& \\ \hline
	& 3 &$	\lambda_{3} + \lambda_{4}							$&$ \lambda_{3} 								$		& \\ \hline
	& 4 &$	\lambda_{4} + \lambda_{5} 							$&$ \lambda_{5} 								$		& \\ \hline
	& 5 &$	\lambda_{5} + \lambda_{6} 							$&$ \lambda_{6} 								$		& \\ \hline
	& 6 &$	\lambda_{6} + \lambda_{7} 							$&$ \lambda_{7} 								$		& \\ \hline
	& 7 &$	\lambda_{7} + \lambda_{8} 							$&$ \lambda_{8} 								$		& \\ \hline \hline
3	& 4 &$	\lambda_{1}+\lambda_{2}+\lambda_{4} 				$&$ \lambda_{1} + \lambda_{2}					$		& \\ \hline
	& 1 &$	\lambda_{2}+\lambda_{3}+\lambda_4					$&$ \lambda_{3} + \lambda_{4}					$		& \\ \hline
	& 3 &$	\lambda_{2}+\lambda_{4}+\lambda_{5}					$&$ \lambda_{2}								$		& $\lambda_5$ \\ \hline
	& 2 &$	\lambda_{3}+\lambda_{4}+\lambda_{5} 				$&$ \lambda_{1} + \lambda_{2}					$		& \\ \hline
	& 5 &$	\lambda_{4}+\lambda_{5}+\lambda_{6} 				$&$ \lambda_{5} + \lambda_{6}					$		& \\ \hline
	& 6 &$	\lambda_{5}+\lambda_{6}+\lambda_{7} 				$&$ \lambda_{6} + \lambda_{7}					$		& \\ \hline
	& 7 &$	\lambda_{6}+\lambda_{7}+\lambda_{8} 				$&$ \lambda_{7} + \lambda_{8}					$		& \\ \hline \hline
4	& 1 &$	\lambda_{1}+\lambda_{2}+\lambda_{3}+\lambda_4		$&$ \lambda_{1}+\lambda_{2}+\lambda_{4} 		$		& \\ \hline
	& 3 &$	\lambda_{1}+\lambda_{2}+\lambda_{4}+\lambda_{5}		$&$ \lambda_{1}+\lambda_{2}						$	&  $\lambda_{4}$ \\ \hline
	& 2 &$	\lambda_{2}+\lambda_{3}+\lambda_{4}+\lambda_5		$&$ \lambda_{2}+\lambda_{4}+\lambda_{5}			$	& \\ \hline
	& 5 &$	\lambda_{2}+\lambda_{4}+\lambda_{5}+\lambda_6		$&$ \lambda_{2}+\lambda_{4} 					$	& $\lambda_{6}$\\ \hline
	& 4 &$	\lambda_{3}+\lambda_{4}+\lambda_{5}+\lambda_6		$&$ \lambda_{4}+\lambda_{5}+\lambda_{6} 		$		& \\ \hline
	& 6 &$	\lambda_{4}+\lambda_{5}+\lambda_{6}+\lambda_7		$&$ \lambda_{5}+\lambda_{6}+\lambda_{7} 		$		& \\ \hline
	& 7 &$	\lambda_{5}+\lambda_{6}+\lambda_{7}+\lambda_8		$&$ \lambda_{6}+\lambda_{7}+\lambda_{8} 		$		& \\ \hline \hline
5	& 3 &$	\lambda_{1}+\lambda_{2}+\lambda_{3}				$	&$ \lambda_{1}+\lambda_{2}+\lambda_{3}+\lambda_4$	& \\
	&   &$	+\lambda_{4}+\lambda_{5}						$	&$ 												$	& \\ \hline
	& 1 &$	\lambda_{2}+\lambda_{3}+2\lambda_{4}			$	&$ \lambda_{2}+\lambda_{3}+\lambda_{4}+\lambda_5$	& \\
	&   &$	+\lambda_{5}									$	&$ 												$	& \\ \hline
	& 4 &$	\lambda_{1}+\lambda_{2}+\lambda_{4}				$	&$ \lambda_{1}+\lambda_{2}		$	& $\lambda_5+\lambda_6$ \\
	&   &$	+\lambda_{5}+\lambda_{6}						$	&$ 												$	& \\ \hline
	& 2 &$	\lambda_{2}+\lambda_{3}+\lambda_{4}				$	&$\lambda_{2}+\lambda_{4}+\lambda_{5}+\lambda_6 $	& \\
	&   &$	+\lambda_{5}+\lambda_{6}						$	&$ 												$	& \\ \hline
	& 6 &$	\lambda_{2}+\lambda_{4}+\lambda_{5}				$	&$ \lambda_{2}+\lambda_{4}		$	& $\lambda_6+\lambda_7$ \\
	&   &$	+\lambda_{6}+\lambda_{7}						$	&$ 												$	& \\ \hline
	& 5 &$	\lambda_{3}+\lambda_{4}+\lambda_{5}				$	&$ \lambda_{4}+\lambda_{5}+\lambda_{6}+\lambda_7$	& \\
	&   &$	+\lambda_{6}+\lambda_{7}						$	&$ 												$	& \\ \hline
	& 7 &$	\lambda_{4}+\lambda_{5}+\lambda_{6}				$	&$ \lambda_{5}+\lambda_{6}+\lambda_{7}+\lambda_8$	& \\
	&   &$	+\lambda_{7}+\lambda_{8}						$	&$ 												$	& \\ \hline \hline
6	& 3 &$	\lambda_{1}+\lambda_{2}+\lambda_{3}				$	&$ \lambda_{1}+\lambda_{2}+\lambda_{3}			$	& \\
	&   &$	+2\lambda_{4}+\lambda_{5}						$	&$ 	+\lambda_{4}+\lambda_{5}					$	& \\ \hline
	& 5 &$	\lambda_{1}+\lambda_{2}+\lambda_{3}				$	&$ \lambda_{1}+\lambda_{2}+\lambda_{3}		$	& $\lambda_5+\lambda_6$  \\
	&   &$	+\lambda_{4}+\lambda_5+\lambda_6				$	&$ 												$	& \\ \hline
	& 2 &$	\lambda_{2}+\lambda_{3}+2\lambda_{4}			$	&$ \lambda_{2}+\lambda_{3}+2\lambda_{4}	+\lambda_{5}$	& \\
	&   &$	+\lambda_{5}+\lambda_{6}						$	&$ 												$	& \\ \hline
	& 4 &$	\lambda_{1}+\lambda_{2}+\lambda_{4}				$	&$ \lambda_{1}+\lambda_{2}+\lambda_{4}			$	& \\
	&   &$	+\lambda_{5}+\lambda_{6}+\lambda_7				$	&$ 	+\lambda_{5}+\lambda_{6}					$	& \\ \hline
	& 1 &$	\lambda_{2}+\lambda_{3}+\lambda_{4}				$	&$ \lambda_{2}+\lambda_{4}+\lambda_{5}			$	& \\
	&   &$	+\lambda_5+\lambda_{6}+\lambda_{7}				$	&$ 	+\lambda_{6}+\lambda_{7}					$	& \\ \hline
	& 7 &$	\lambda_{2}+\lambda_{4}+\lambda_{5}				$	&$ \lambda_{2}+\lambda_{4}+\lambda_{5}		$	& $\lambda_7+\lambda_8$\\
	&   &$	+\lambda_{6}+\lambda_{7}+\lambda_8				$	&$ 												$	& \\ \hline
	& 6 &$	\lambda_3+\lambda_{4}+\lambda_{5}				$	&$ \lambda_{4}+\lambda_{5}+\lambda_{6}			$	& \\
	&   &$	+\lambda_{6}+\lambda_{7}+\lambda_{8}			$	&$ 	+\lambda_{7}+\lambda_{8}					$	& \\ \hline \hline
7	& 1 &$	\lambda_{1}+2\lambda_{2}+\lambda_{3}			$	& $\lambda_{1}+\lambda_{2}+\lambda_{3}			$	& \\
	&   &$	+2\lambda_{4}+\lambda_{5}						$	&$+2\lambda_{4}+\lambda_{5}						$	& \\ \hline
	& 3 &$	\lambda_{1}+\lambda_{2}+\lambda_{3}				$	&$ \lambda_{1}+\lambda_{2}+\lambda_{3}			$	& \\
	&   &$	+2\lambda_{4}+\lambda_5+\lambda_6				$	&$ 	+\lambda_{4}+\lambda_5+\lambda_6			$	& \\ \hline
	& 2 &$	\lambda_{2}+\lambda_{3}+2\lambda_{4}			$	&$ \lambda_{2}+\lambda_{3}+2\lambda_{4}			$	& \\
	&   &$	+2\lambda_{5}+\lambda_{6}						$	&$ 	+\lambda_{5}+\lambda_{6}					$	& \\ \hline
	& 4 &$	\lambda_{1}+\lambda_{2}+\lambda_3+\lambda_{4}	$	&$ \lambda_{2}+\lambda_{3}+\lambda_{4}			$	& \\
	&   &$	+\lambda_{5}+\lambda_{6}+\lambda_7				$	&$ +\lambda_5+\lambda_{6}+\lambda_{7}			$	& \\ \hline
	& 6 &$	\lambda_{2}+\lambda_{3}+2\lambda_{4}			$	&$ \lambda_{2}+\lambda_{3}+2\lambda_{4}	+\lambda_5 $	& $\lambda_7$ \\
	&   &$	+\lambda_5+\lambda_{6}+\lambda_{7}				$	&$ 												$	& \\ \hline
	& 5 &$	\lambda_1+\lambda_{2}+\lambda_{4}+\lambda_{5}	$	&$ \lambda_{2}+\lambda_{4}+\lambda_{5}			$	& \\
	&   &$	+\lambda_{6}+\lambda_{7}+\lambda_8				$	&$ +\lambda_{6}+\lambda_{7}+\lambda_8			$	& \\ \hline
	& 7 &$	\lambda_2+\lambda_3+\lambda_{4}+\lambda_{5}		$	&$ \lambda_3+\lambda_{4}+\lambda_{5}			$	& \\
	&   &$	+\lambda_{6}+\lambda_{7}+\lambda_{8}			$	&$ 	+\lambda_{6}+\lambda_{7}+\lambda_{8}		$	& \\ \hline \hline
8	& 5 &$	\lambda_{1}+2\lambda_{2}+\lambda_{3}			$	&$ \lambda_{1}+2\lambda_{2}+\lambda_{3}			$	& \\
	&   &$	+2\lambda_{4}+\lambda_5+\lambda_6				$	&$ 	+2\lambda_{4}+\lambda_{5}					$	& \\ \hline
	& 1 &$	\lambda_{1}+\lambda_{2}+\lambda_{3}				$	&$ \lambda_{2}+\lambda_{3}+2\lambda_{4}			$	& \\
	&   &$	+2\lambda_4+2\lambda_{5}+\lambda_{6}			$	&$ +2\lambda_{5}+\lambda_{6}					$	& \\ \hline
	& 3 &$	\lambda_{1}+\lambda_{2}+\lambda_3+2\lambda_{4}	$	&$ \lambda_{1}+\lambda_{2}+\lambda_3+\lambda_{4}$	& \\
	&   &$	+\lambda_{5}+\lambda_{6}+\lambda_7				$	&$ 	+\lambda_{5}+\lambda_{6}+\lambda_7			$	& \\ \hline
	& 2 &$	\lambda_{2}+\lambda_{3}+2\lambda_{4}			$	&$ \lambda_{2}+\lambda_{3}+2\lambda_{4}			$	& \\
	&   &$	+2\lambda_5+\lambda_{6}+\lambda_{7}				$	&$ 	+\lambda_5+\lambda_{6}+\lambda_{7}			$	& \\ \hline
	& 6 &$	\lambda_1+\lambda_{2}+\lambda_{3}+\lambda_{4}	$	&$ \lambda_1+\lambda_{2}+\lambda_{4}+\lambda_{5}$	& \\
	&   &$	+\lambda_5+\lambda_{6}+\lambda_{7}+\lambda_8	$	&$ +\lambda_{6}+\lambda_{7}+\lambda_8			$	& \\ \hline
	& 4 &$	\lambda_2+\lambda_3+2\lambda_{4}+\lambda_{5}	$	&$ \lambda_2+\lambda_3+\lambda_{4}+\lambda_{5}	$	& \\
	&   &$	+\lambda_{6}+\lambda_{7}+\lambda_{8}			$	&$ 	+\lambda_{6}+\lambda_{7}+\lambda_{8}		$	& \\ \hline \hline
9	& 4 &$	\lambda_{1}+2\lambda_{2}+\lambda_{3}			$	&$ \lambda_{1}+2\lambda_{2}+\lambda_{3}			$	& \\
	&   &$	+2\lambda_{4}+2\lambda_5+\lambda_6				$	&$ 	+2\lambda_{4}+\lambda_5+\lambda_6			$	& \\ \hline
	& 5 &$	\lambda_{1}+2\lambda_{2}+\lambda_{3}+2\lambda_4	$	&$ \lambda_{1}+2\lambda_{2}+\lambda_{3}			$	&$\lambda_7$ \\
	&   &$	+\lambda_{5}+\lambda_{6}+\lambda_7				$	&$ 		+2\lambda_4+\lambda_5					$	& \\ \hline
	& 3 &$	\lambda_{1}+\lambda_{2}+\lambda_3+2\lambda_{4}	$	&$ \lambda_{1}+\lambda_{2}+\lambda_3+2\lambda_{4}$	& \\
	&   &$	+2\lambda_{5}+\lambda_{6}+\lambda_7				$	&$ +\lambda_{5}+\lambda_{6}+\lambda_7			$	& \\ \hline
	& 6 &$	\lambda_1+\lambda_{2}+\lambda_{3}+2\lambda_{4}	$	&$ \lambda_1+\lambda_{2}+\lambda_{3}+\lambda_{4}$	& \\
	&   &$	+\lambda_5+\lambda_{6}+\lambda_{7}+\lambda_8	$	&$ 	+\lambda_5+\lambda_{6}+\lambda_{7}+\lambda_8$	& \\ \hline
	& 1 &$	\lambda_2+\lambda_{3}+2\lambda_{4}+2\lambda_{5}	$	&$ \lambda_{2}+\lambda_{3}+2\lambda_{4}			$	& \\
	&   &$	+2\lambda_6+\lambda_{7}							$	&$ 	+2\lambda_5+\lambda_{6}+\lambda_{7}			$	& \\ \hline
	& 2 &$	\lambda_2+\lambda_3+2\lambda_{4}+2\lambda_{5}	$	&$ \lambda_2+\lambda_3+2\lambda_{4}+\lambda_{5} $	& \\
	&   &$	+\lambda_{6}+\lambda_{7}+\lambda_{8}			$	&$ 	+\lambda_{6}+\lambda_{7}+\lambda_{8}		$	& \\ \hline \hline
10	& 1 &$	\lambda_{1}+2\lambda_{2}+\lambda_{3}			$	&$ \lambda_{1}+2\lambda_{2}+\lambda_{3}			$	& \\
	&   &$	+3\lambda_{4}+2\lambda_5+\lambda_6				$	&$ 	+2\lambda_{4}+2\lambda_5+\lambda_6			$	& \\ \hline
	& 2 &$	\lambda_{1}+2\lambda_{2}+\lambda_{3}+2\lambda_4	$	&$ \lambda_{1}+2\lambda_{2}+\lambda_{3}+2\lambda_4$	& \\
	&   &$	+2\lambda_{5}+\lambda_{6}+\lambda_7				$	&$ 	+\lambda_{5}+\lambda_{6}+\lambda_7			$	& \\ \hline
	& 3 &$	\lambda_{1}+2\lambda_{2}+\lambda_3+2\lambda_{4}	$	&$ 	\lambda_1+\lambda_{2}+\lambda_{3}+2\lambda_{4}$	& \\
	&   &$	+\lambda_{5}+\lambda_{6}+\lambda_7+\lambda_8	$	&$ +\lambda_5+\lambda_{6}+\lambda_{7}+\lambda_8	$	& \\ \hline\newpage
	& 6 &$	\lambda_1+\lambda_{2}+\lambda_{3}+2\lambda_{4}	$	&$ \lambda_{1}+\lambda_{2}+\lambda_{3}+2\lambda_4$	& $\lambda_6$ \\
	&   &$	+2\lambda_5+2\lambda_{6}+\lambda_{7}			$	&$ 		+\lambda_5+\lambda_6+\lambda_7			$	& \\ \hline
	& 4 &$	\lambda_1+\lambda_{2}+\lambda_{3}+2\lambda_{4}	$	&$ \lambda_2+\lambda_3+2\lambda_{4}+2\lambda_{5}$	& \\
	&   &$	+2\lambda_5+\lambda_{6}+\lambda_7+\lambda_8		$	&$ 	+\lambda_{6}+\lambda_{7}+\lambda_{8}		$	& \\ \hline
	& 5 &$	\lambda_2+\lambda_3+2\lambda_{4}+2\lambda_{5}	$	&$ 	\lambda_2+\lambda_{3}+2\lambda_{4}+2\lambda_{5}$	& \\
	&   &$	+2\lambda_{6}+\lambda_{7}+\lambda_{8}			$	&$ 	+2\lambda_6+\lambda_{7}						$	& \\ \hline \hline
11	& 1 &$	\lambda_{1}+2\lambda_{2}+2\lambda_{3}			$	&$ \lambda_{1}+2\lambda_{2}+\lambda_{3}			$	& \\
	&   &$	+3\lambda_{4}+2\lambda_5+\lambda_6				$	&$ 	+3\lambda_{4}+2\lambda_5+\lambda_6			$	& \\ \hline
	& 2 &$	\lambda_{1}+2\lambda_{2}+\lambda_{3}+3\lambda_4	$	&$ \lambda_{1}+2\lambda_{2}+\lambda_{3}+2\lambda_4	$	& \\
	&   &$	+2\lambda_{5}+\lambda_{6}+\lambda_7				$	&$ 	+2\lambda_{5}+\lambda_{6}+\lambda_7			$	& \\ \hline
	& 3 &$	\lambda_{1}+2\lambda_{2}+\lambda_3+2\lambda_{4}	$	&$ \lambda_1+\lambda_{2}+\lambda_{3}+2\lambda_{4}$	& \\
	&   &$	+2\lambda_{5}+2\lambda_{6}+\lambda_7			$	&$ 	+2\lambda_5+2\lambda_{6}+\lambda_{7}		$	& \\ \hline
	& 6 &$	\lambda_1+2\lambda_{2}+\lambda_{3}+2\lambda_{4}	$	&$ \lambda_{1}+2\lambda_{2}+\lambda_3+2\lambda_{4}$	& \\
	&   &$	+2\lambda_5+\lambda_{6}+\lambda_{7}+\lambda_8	$	&$ 	+\lambda_{5}+\lambda_{6}+\lambda_7+\lambda_8$	& \\ \hline
	& 5 &$	\lambda_1+\lambda_{2}+\lambda_{3}+2\lambda_{4}	$	&$ 	\lambda_1+\lambda_{2}+\lambda_{3}+2\lambda_{4}$	& \\
	&   &$	+2\lambda_5+2\lambda_{6}+\lambda_7+\lambda_8	$	&$ 	+2\lambda_5+\lambda_{6}+\lambda_7+\lambda_8	$	& \\ \hline
	& 4 &$	\lambda_2+\lambda_3+2\lambda_{4}+2\lambda_{5}	$	&$ \lambda_2+\lambda_3+2\lambda_{4}+2\lambda_{5}$	& \\
	&   &$	+2\lambda_{6}+2\lambda_{7}+\lambda_{8}			$	&$ 	+2\lambda_{6}+\lambda_{7}+\lambda_{8}		$	& \\ \hline \hline
12	& 1 &$	\lambda_{1}+2\lambda_{2}+2\lambda_{3}			$	&$ \lambda_{1}+2\lambda_{2}+2\lambda_{3}		$	& \\
	&   &$	+3\lambda_{4}+2\lambda_5+\lambda_6+\lambda_7	$	&$ +3\lambda_{4}+2\lambda_5+\lambda_6			$	& \\ \hline
	& 2 &$	\lambda_{1}+2\lambda_{2}+\lambda_{3}+3\lambda_4	$	&$ \lambda_{1}+2\lambda_{2}+\lambda_3+2\lambda_{4}$	& \\
	&   &$	+2\lambda_{5}+2\lambda_{6}+\lambda_7			$	&$ 	+2\lambda_{5}+2\lambda_{6}+\lambda_7		$	& \\ \hline
	& 3 &$	\lambda_1+2\lambda_{2}+\lambda_{3}+3\lambda_{4}	$	&$ \lambda_1+2\lambda_{2}+\lambda_{3}+2\lambda_{4}$	& \\
	&   &$	+2\lambda_5+\lambda_{6}+\lambda_{7}+\lambda_8	$	&$ +2\lambda_5+\lambda_{6}+\lambda_{7}+\lambda_8$	& \\ \hline
	& 4 &$	\lambda_1+2\lambda_{2}+\lambda_{3}+2\lambda_{4}	$	&$ \lambda_1+\lambda_{2}+\lambda_{3}+2\lambda_{4}$	& \\
	&   &$	+2\lambda_5+2\lambda_{6}+\lambda_7+\lambda_8	$	&$ +2\lambda_5+2\lambda_{6}+\lambda_7+\lambda_8	$	& \\ \hline
	& 5 &$	\lambda_1+\lambda_2+\lambda_3+2\lambda_{4}		$	&$ \lambda_2+\lambda_3+2\lambda_{4}+2\lambda_{5}$	& \\
	&   &$	+2\lambda_{5}+2\lambda_{6}+2\lambda_{7}+\lambda_{8}$&$ 	+2\lambda_{6}+2\lambda_{7}+\lambda_{8}		$	& \\ \hline \hline
13	& 2 &$	\lambda_{1}+2\lambda_{2}+2\lambda_{3}			$	&$\lambda_{1}+2\lambda_{2}+2\lambda_{3}			$	& \\
	&   &$	+3\lambda_{4}+2\lambda_5+2\lambda_6+\lambda_7	$	&$ 	+3\lambda_{4}+2\lambda_5+\lambda_6+\lambda_7$	& \\ \hline
	& 3 &$	\lambda_1+2\lambda_{2}+2\lambda_{3}+3\lambda_{4}$	&$ \lambda_1+2\lambda_{2}+\lambda_{3}+3\lambda_{4}$	& \\
	&   &$	+2\lambda_5+\lambda_{6}+\lambda_{7}+\lambda_8	$	&$ +2\lambda_5+\lambda_{6}+\lambda_{7}+\lambda_8$	& \\ \hline
	& 1 &$	\lambda_{1}+2\lambda_{2}+\lambda_{3}+3\lambda_4	$	&$ \lambda_{1}+2\lambda_{2}+\lambda_{3}+3\lambda_4$	& \\
	&   &$	+3\lambda_{5}+2\lambda_{6}+\lambda_7			$	&$ 	+2\lambda_{5}+2\lambda_{6}+\lambda_7		$	& \\ \hline
	& 4 &$	\lambda_1+2\lambda_{2}+\lambda_{3}+3\lambda_{4}	$	&$ \lambda_1+2\lambda_{2}+\lambda_{3}+2\lambda_{4}$	& \\
	&   &$	+2\lambda_5+2\lambda_{6}+\lambda_{7}+\lambda_8	$	&$ 	+2\lambda_5+2\lambda_{6}+\lambda_7+\lambda_8$	& \\ \hline
	& 5 &$	\lambda_1+2\lambda_2+\lambda_3+2\lambda_{4}		$	&$	\lambda_1+\lambda_2+\lambda_3+2\lambda_{4}	$	& \\
	&   &$	+2\lambda_{5}+2\lambda_{6}+2\lambda_{7}+\lambda_{8}$&$+2\lambda_{5}+2\lambda_{6}+2\lambda_{7}+\lambda_{8}$	& \\ \hline \hline
14	& 1 &$	\lambda_{1}+2\lambda_{2}+2\lambda_{3}			$	&$ \lambda_{1}+2\lambda_{2}+2\lambda_{3}		$	& \\
	&   &$	+3\lambda_{4}+3\lambda_5+2\lambda_6+\lambda_7	$	&$ +3\lambda_{4}+2\lambda_5+2\lambda_6+\lambda_7$	& \\ \hline
	& 2 &$	\lambda_1+2\lambda_{2}+2\lambda_{3}+3\lambda_{4}$	&$ \lambda_1+2\lambda_{2}+2\lambda_{3}+3\lambda_{4}$	& \\
	&   &$	+2\lambda_5+2\lambda_{6}+\lambda_{7}+\lambda_8	$	&$ 	+2\lambda_5+\lambda_{6}+\lambda_{7}+\lambda_8$	& \\ \hline
	& 3 &$	\lambda_{1}+2\lambda_{2}+\lambda_{3}+3\lambda_4	$	&$ \lambda_1+2\lambda_{2}+\lambda_{3}+3\lambda_{4}$	& \\
	&   &$	+3\lambda_{5}+2\lambda_{6}+\lambda_7+\lambda_8	$	&$ 	+2\lambda_5+2\lambda_{6}+\lambda_{7}+\lambda_8$	& \\ \hline
	& 4 &$	\lambda_1+2\lambda_{2}+\lambda_{3}+3\lambda_{4}	$	&$ \lambda_1+2\lambda_2+\lambda_3+2\lambda_{4}	$	& \\
	&   &$	+2\lambda_5+2\lambda_{6}+2\lambda_{7}+\lambda_8	$	&$ 	+2\lambda_{5}+2\lambda_{6}+2\lambda_{7}+\lambda_{8}$	& \\ \hline \hline\newpage
15	& 1 &$	\lambda_{1}+2\lambda_{2}+2\lambda_{3}			$	&$ \lambda_{1}+2\lambda_{2}+2\lambda_{3}			$	& \\
	&   &$	+4\lambda_{4}+3\lambda_5+2\lambda_6+\lambda_7	$	&$ 	+3\lambda_{4}+3\lambda_5+2\lambda_6+\lambda_7$	& \\ \hline
	& 2 &$	\lambda_1+2\lambda_{2}+2\lambda_{3}+3\lambda_{4}$	&$\lambda_1+2\lambda_{2}+2\lambda_{3}+3\lambda_{4}$	& \\
	&   &$	+3\lambda_5+2\lambda_{6}+\lambda_{7}+\lambda_8	$	&$ +2\lambda_5+2\lambda_{6}+\lambda_{7}+\lambda_8$	& \\ \hline
	& 4 &$	\lambda_1+2\lambda_{2}+2\lambda_{3}+3\lambda_{4}$	&$ \lambda_1+2\lambda_{2}+2\lambda_{3}+3\lambda_{4}$	& $\lambda_7$ \\
	&   &$	+2\lambda_5+2\lambda_{6}+2\lambda_{7}+\lambda_8	$	&$ 	+2\lambda_5+\lambda_{6}+\lambda_{7}+\lambda_8$	& \\ \hline
	& 3 &$	\lambda_1+2\lambda_{2}+\lambda_{3}+3\lambda_{4}	$	&$ \lambda_1+2\lambda_{2}+\lambda_{3}+3\lambda_{4}$	& \\
	&   &$	+3\lambda_5+2\lambda_{6}+2\lambda_{7}+\lambda_8	$	&$ 	+2\lambda_5+2\lambda_{6}+2\lambda_{7}+\lambda_8$	& \\ \hline \hline
16	& 1 &$	\lambda_{1}+3\lambda_{2}+2\lambda_{3}			$	&$ 	\lambda_{1}+2\lambda_{2}+2\lambda_{3}		$	& \\
	&   &$	+4\lambda_{4}+3\lambda_5+2\lambda_6+\lambda_7	$	&$ +4\lambda_{4}+3\lambda_5+2\lambda_6+\lambda_7$	& \\ \hline
	& 2 &$	\lambda_1+2\lambda_{2}+2\lambda_{3}+4\lambda_{4}$	&$ \lambda_1+2\lambda_{2}+2\lambda_{3}+3\lambda_{4}$	& \\
	&   &$	+3\lambda_5+2\lambda_{6}+\lambda_{7}+\lambda_8	$	&$ 	+3\lambda_5+2\lambda_{6}+\lambda_{7}+\lambda_8$	& \\ \hline
	& 4 &$	\lambda_1+2\lambda_{2}+2\lambda_{3}+3\lambda_{4}$	&$ 	\lambda_1+2\lambda_{2}+2\lambda_{3}+3\lambda_{4}$	& \\
	&   &$	+3\lambda_5+2\lambda_{6}+2\lambda_{7}+\lambda_8	$	&$ 	+2\lambda_5+2\lambda_{6}+2\lambda_{7}+\lambda_8$	& \\ \hline
	& 3 &$	\lambda_1+2\lambda_{2}+\lambda_{3}+3\lambda_{4}	$	&$ \lambda_1+2\lambda_{2}+\lambda_{3}+3\lambda_{4}$	& \\
	&   &$	+3\lambda_5+3\lambda_{6}+2\lambda_{7}+\lambda_8	$	&$ 	+3\lambda_5+2\lambda_{6}+2\lambda_{7}+\lambda_8	$	& \\ \hline \hline
17	& 1 &$	2\lambda_{1}+3\lambda_{2}+2\lambda_{3}			$	&$ \lambda_{1}+3\lambda_{2}+2\lambda_{3}		$	& \\
	&   &$	+4\lambda_{4}+3\lambda_5+2\lambda_6+\lambda_7	$	&$ 	+4\lambda_{4}+3\lambda_5+2\lambda_6+\lambda_7$	& \\ \hline
	& 2 &$	\lambda_1+3\lambda_{2}+2\lambda_{3}+4\lambda_{4}$	&$ \lambda_1+2\lambda_{2}+2\lambda_{3}+4\lambda_{4}$	& \\
	&   &$	+3\lambda_5+2\lambda_{6}+\lambda_{7}+\lambda_8	$	&$ 	+3\lambda_5+2\lambda_{6}+\lambda_{7}+\lambda_8$	& \\ \hline
	& 3 &$	\lambda_1+2\lambda_{2}+2\lambda_{3}+4\lambda_{4}$	&$ \lambda_1+2\lambda_{2}+2\lambda_{3}+3\lambda_{4}$	& \\
	&   &$	+3\lambda_5+2\lambda_{6}+2\lambda_{7}+\lambda_8	$	&$ 	+3\lambda_5+2\lambda_{6}+2\lambda_{7}+\lambda_8$	& \\ \hline
	& 4 &$	\lambda_1+2\lambda_{2}+2\lambda_{3}+3\lambda_{4}$	&$ \lambda_1+2\lambda_{2}+\lambda_{3}+3\lambda_{4}$	& \\
	&   &$	+3\lambda_5+3\lambda_{6}+2\lambda_{7}+\lambda_8	$	&$ 	+3\lambda_5+3\lambda_{6}+2\lambda_{7}+\lambda_8$	& \\ \hline \hline
18	& 1 &$	2\lambda_{1}+3\lambda_{2}+2\lambda_{3}+4\lambda_{4}$&$ \lambda_1+3\lambda_{2}+2\lambda_{3}+4\lambda_{4}$	& \\
	&   &$	+3\lambda_5+2\lambda_6+\lambda_7+\lambda_8		$	&$ +3\lambda_5+2\lambda_{6}+\lambda_{7}+\lambda_8$	& \\ \hline
	& 2 &$	\lambda_1+3\lambda_{2}+2\lambda_{3}+4\lambda_{4}$	&$ \lambda_1+2\lambda_{2}+2\lambda_{3}+4\lambda_{4}$	& \\
	&   &$	+3\lambda_5+2\lambda_{6}+2\lambda_{7}+\lambda_8	$	&$ 	+3\lambda_5+2\lambda_{6}+2\lambda_{7}+\lambda_8$	& \\ \hline
	& 3 &$	\lambda_1+2\lambda_{2}+2\lambda_{3}+4\lambda_{4}$	&$ \lambda_1+2\lambda_{2}+2\lambda_{3}+3\lambda_{4}$	& \\
	&   &$	+3\lambda_5+3\lambda_{6}+2\lambda_{7}+\lambda_8	$	&$ 	+3\lambda_5+3\lambda_{6}+2\lambda_{7}+\lambda_8	$	& \\ \hline \hline
19	& 3 &$	2\lambda_{1}+3\lambda_{2}+2\lambda_{3}+4\lambda_{4}$&$2\lambda_{1}+3\lambda_{2}+2\lambda_{3}+4\lambda_{4}$	& \\
	&   &$	+3\lambda_5+2\lambda_6+2\lambda_7+\lambda_8		$	&$ +3\lambda_5+2\lambda_6+\lambda_7+\lambda_8	$	& \\ \hline
	& 2 &$	\lambda_1+3\lambda_{2}+2\lambda_{3}+4\lambda_{4}$	&$ 	\lambda_1+3\lambda_{2}+2\lambda_{3}+4\lambda_{4}$	& \\
	&   &$	+3\lambda_5+3\lambda_{6}+2\lambda_{7}+\lambda_8	$	&$ +3\lambda_5+2\lambda_{6}+2\lambda_{7}+\lambda_8	$	& \\ \hline
	& 1 &$	\lambda_1+2\lambda_{2}+2\lambda_{3}+4\lambda_{4}$	&$ \lambda_1+2\lambda_{2}+2\lambda_{3}+4\lambda_{4}$	& \\
	&   &$	+4\lambda_5+3\lambda_{6}+2\lambda_{7}+\lambda_8	$	&$ 	+3\lambda_5+3\lambda_{6}+2\lambda_{7}+\lambda_8$	& \\ \hline \hline
20	& 1 &$	2\lambda_{1}+3\lambda_{2}+2\lambda_{3}+4\lambda_{4}$&$ 	\lambda_1+3\lambda_{2}+2\lambda_{3}+4\lambda_{4}$	& \\
	&   &$	+3\lambda_5+3\lambda_6+2\lambda_7+\lambda_8		$	&$	+3\lambda_5+3\lambda_{6}+2\lambda_{7}+\lambda_8$	& \\ \hline
	& 2 &$	\lambda_1+3\lambda_{2}+2\lambda_{3}+4\lambda_{4}$	&$ \lambda_1+2\lambda_{2}+2\lambda_{3}+4\lambda_{4}$	& \\
	&   &$	+4\lambda_5+3\lambda_{6}+2\lambda_{7}+\lambda_8	$	&$ +4\lambda_5+3\lambda_{6}+2\lambda_{7}+\lambda_8$	& \\ \hline \hline
21	& 2 &$	2\lambda_{1}+3\lambda_{2}+2\lambda_{3}+4\lambda_{4}$&$ 2\lambda_{1}+3\lambda_{2}+2\lambda_{3}+4\lambda_{4}$	& \\
	&   &$	+4\lambda_5+3\lambda_6+2\lambda_7+\lambda_8		$	&$ +3\lambda_5+3\lambda_6+2\lambda_7+\lambda_8	$	& \\ \hline
	& 1 &$	\lambda_1+3\lambda_{2}+2\lambda_{3}+5\lambda_{4}$	&$ 	\lambda_1+3\lambda_{2}+2\lambda_{3}+4\lambda_{4}$	& \\
	&   &$	+4\lambda_5+3\lambda_{6}+2\lambda_{7}+\lambda_8	$	&$ 	+4\lambda_5+3\lambda_{6}+2\lambda_{7}+\lambda_8	$	& \\ \hline \hline
22	& 2 &$	2\lambda_{1}+3\lambda_{2}+2\lambda_{3}+5\lambda_{4}$&$ 2\lambda_{1}+3\lambda_{2}+2\lambda_{3}+4\lambda_{4}$	& \\
	&   &$	+4\lambda_5+3\lambda_6+2\lambda_7+\lambda_8		$	&$ 	+4\lambda_5+3\lambda_6+2\lambda_7+\lambda_8	$	& \\ \hline\newpage
	& 1 &$	\lambda_1+3\lambda_{2}+3\lambda_{3}+5\lambda_{4}$	&$ 	\lambda_1+3\lambda_{2}+2\lambda_{3}+5\lambda_{4}$	& \\
	&   &$	+4\lambda_5+3\lambda_{6}+2\lambda_{7}+\lambda_8	$	&$ 	+4\lambda_5+3\lambda_{6}+2\lambda_{7}+\lambda_8$	& \\ \hline \hline
23	& 1 &$	2\lambda_{1}+4\lambda_{2}+2\lambda_{3}+5\lambda_{4}$&$ 2\lambda_{1}+3\lambda_{2}+2\lambda_{3}+5\lambda_{4}$	& \\
	&   &$	+4\lambda_5+3\lambda_6+2\lambda_7+\lambda_8		$	&$ +4\lambda_5+3\lambda_6+2\lambda_7+\lambda_8		$	& \\ \hline
	& 2 &$	2\lambda_1+3\lambda_{2}+3\lambda_{3}+5\lambda_{4}$	&$ \lambda_1+3\lambda_{2}+3\lambda_{3}+5\lambda_{4}$	& \\
	&   &$	+4\lambda_5+3\lambda_{6}+2\lambda_{7}+\lambda_8	$	&$ 	+4\lambda_5+3\lambda_{6}+2\lambda_{7}+\lambda_8	$	& \\ \hline \hline
24	& 1 &$	2\lambda_{1}+4\lambda_{2}+3\lambda_{3}+5\lambda_{4}$&$ 2\lambda_{1}+4\lambda_{2}+2\lambda_{3}+5\lambda_{4}$	& \\
	&   &$	+4\lambda_5+3\lambda_6+2\lambda_7+\lambda_8		$	&$ +4\lambda_5+3\lambda_6+2\lambda_7+\lambda_8	$	& \\ \hline \hline
25	& 1 &$	2\lambda_{1}+4\lambda_{2}+3\lambda_{3}+6\lambda_{4}$&$ 2\lambda_{1}+4\lambda_{2}+3\lambda_{3}+5\lambda_{4}	$	& \\
	&   &$	+4\lambda_5+3\lambda_6+2\lambda_7+\lambda_8		$	&$ 	+4\lambda_5+3\lambda_6+2\lambda_7+\lambda_8		$	& \\ \hline \hline
26	& 1 &$	2\lambda_{1}+4\lambda_{2}+3\lambda_{3}+6\lambda_{4}$&$ 2\lambda_{1}+4\lambda_{2}+3\lambda_{3}+6\lambda_{4}$	& \\
	&   &$	+5\lambda_5+3\lambda_6+2\lambda_7+\lambda_8		$	&$ 	+4\lambda_5+3\lambda_6+2\lambda_7+\lambda_8	$	& \\ \hline \hline
27	& 1 &$	2\lambda_{1}+4\lambda_{2}+3\lambda_{3}+6\lambda_{4}$&$2\lambda_{1}+4\lambda_{2}+3\lambda_{3}+6\lambda_{4}$	& \\
	&   &$	+5\lambda_5+4\lambda_6+2\lambda_7+\lambda_8		$	&$ 	+5\lambda_5+3\lambda_6+2\lambda_7+\lambda_8		$	& \\ \hline \hline
28	& 1 &$	2\lambda_{1}+4\lambda_{2}+3\lambda_{3}+6\lambda_{4}$&$ 2\lambda_{1}+4\lambda_{2}+3\lambda_{3}+6\lambda_{4}$	& \\
	&   &$	+5\lambda_5+4\lambda_6+3\lambda_7+\lambda_8		$	&$ 	+5\lambda_5+4\lambda_6+2\lambda_7+\lambda_8		$	& \\ \hline \hline
29	& 1 &$	2\lambda_{1}+4\lambda_{2}+3\lambda_{3}+6\lambda_{4}$&$ 2\lambda_{1}+4\lambda_{2}+3\lambda_{3}+6\lambda_{4}$	& \\
	&   &$	+5\lambda_5+4\lambda_6+3\lambda_7+2\lambda_8	$	&$ 	+5\lambda_5+4\lambda_6+3\lambda_7+\lambda_8	$	& \\ \hline 
\end{longtable}

\begin{table}[h!]
\caption{The simple case for root systems of type $F_{4}$}\label{table:F_4}
$$\begin{array}{|l|l|c|c|c|}
\hline

\mbox{Height}& \mbox{Order} & \lambda & \mu\mbox{ or }\mu_{1} & \mu_{2} \mbox{ (if needed)}\\ \hline \hline
2	& 1 &	\lambda_{1} + \lambda_{2}							& \lambda_{2}										& \\ \hline
	& 2 &	\lambda_{2} + \lambda_{3} 							& \lambda_{3} 										& \\ \hline
	& 3 &	\lambda_{3} + \lambda_{4}							& \lambda_{4} 										& \\ \hline \hline
3	& 2 &	\lambda_{1}+\lambda_{2}+\lambda_{3} 				& \lambda_{1} + \lambda_{2}							& \\ \hline
	& 1 &	\lambda_{2}+2\lambda_{3}							& \lambda_{2} + \lambda_{3}							& \\ \hline
	& 3 &	\lambda_{2}+\lambda_{3}+\lambda_{4}					& \lambda_{3} + \lambda_{4}							& \\ \hline \hline
4	& 1 &	\lambda_{1}+\lambda_{2}+2\lambda_{3}				& \lambda_{1}+\lambda_{2}+\lambda_{3} 				& \\ \hline
	& 3 &	\lambda_{1}+\lambda_{2}+\lambda_{3}+\lambda_{4}		& \lambda_{1}+\lambda_{2}							&  \lambda_{4} \\ \hline
	& 2 &	\lambda_{2}+2\lambda_{3}+\lambda_{4}				& \lambda_{2}+\lambda_{3}+\lambda_{4}				& \\ \hline \hline
5	& 1 &	\lambda_{1}+2\lambda_{2}+2\lambda_{3}				& \lambda_{1}+\lambda_{2}+2\lambda_{3}				& \\ \hline
	& 3 &	\lambda_{1}+\lambda_{2}+2\lambda_{3}+\lambda_{4}	& \lambda_{1}+\lambda_{2}+\lambda_{3}+\lambda_{4}	& \\ \hline
	& 2 &	\lambda_{2}+2\lambda_{3} +2\lambda_{4}				& \lambda_{2}+2\lambda_{3}+\lambda_{4}				& \\ \hline \hline
6	& 1 &	\lambda_{1}+2\lambda_{2}+2\lambda_{3}+\lambda_{4}	& \lambda_{1}+2\lambda_{2}+2\lambda_{3}				& \\ \hline	
	& 2 &	\lambda_{1}+\lambda_{2}+2\lambda_{3}+2\lambda_{4}	& \lambda_{2}+2\lambda_{3} +2\lambda_{4}			& \\ \hline \hline
7	& 1 &	\lambda_{1}+2\lambda_{2}+3\lambda_{3}+\lambda_{4}	& \lambda_{1}+2\lambda_{2}+2\lambda_{3}+\lambda_{4}	& \\ \hline
	& 2 &	\lambda_{1}+2\lambda_{2}+2\lambda_{3}+2\lambda_{4}	& \lambda_{1}+\lambda_{2}+2\lambda_{3}+2\lambda_{4}	& \\ \hline \hline
8	& 1 &	\lambda_{1}+2\lambda_{2}+3\lambda_{3}+2\lambda_{4}	& \lambda_{1}+2\lambda_{2}+3\lambda_{3}+\lambda_{4}	& \\ \hline \hline
9	& 1 &	\lambda_{1}+3\lambda_{2}+3\lambda_{3}+2\lambda_{4}	& \lambda_{1}+2\lambda_{2}+3\lambda_{3}+2\lambda_{4}& \\ \hline \hline
10	& 1 &	\lambda_{1}+3\lambda_{2}+4\lambda_{3}+2\lambda_{4}	& \lambda_{1}+2\lambda_{2}+3\lambda_{3}+2\lambda_{4}& \\ \hline \hline
11	& 1 &	2\lambda_{1}+3\lambda_{2}+4\lambda_{3}+2\lambda_{4}	& \lambda_{1}+3\lambda_{2}+4\lambda_{3}+2\lambda_{4}& \\ \hline

\end{array}
$$
\end{table}

\begin{figure}[h!]
   \labellist
    \small\hair 5pt
       \pinlabel $\lambda_{1}$ [r] at 1 642
       \pinlabel $\lambda_{2}$ [r] at 82 642
       \pinlabel $\lambda_{3}$ [l] at 162 642
       \pinlabel $\lambda_{4}$ [l] at 242 642
       \tiny
       \pinlabel $\lambda_1+\lambda_2$ [r] at 42 577
       \pinlabel $\lambda_1+\lambda_2+\lambda_3$ [r] at 41 513
       \pinlabel $\lambda_1+\lambda_2+2\lambda_3$ [r] at 41 449
       \pinlabel $\lambda_1+2\lambda_2+2\lambda_3$ [r] at 41 386
   \endlabellist

   \centering
   \includegraphics[width=3.4cm]{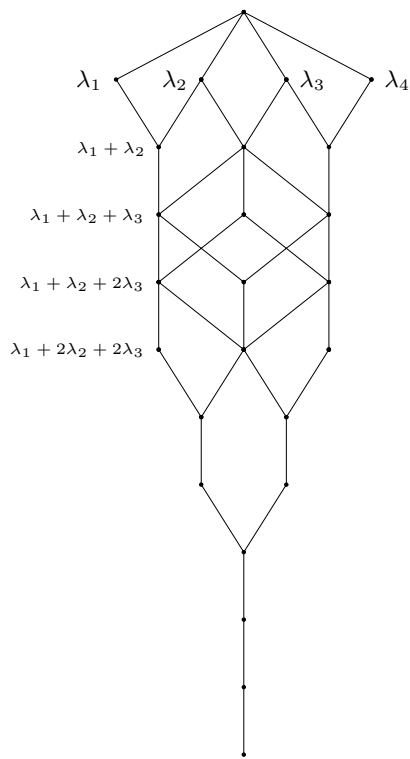}
   \caption[The root system $F_4$]{A graphical depiction of the positive roots in $F_{4}$. The horizontal levels correspond to heights. The elements of height one are labelled, and one more in each height up to $5$, but the rest are not. When you move down a height, following an edge corresponds to adding a simple root.}
   \label{fig:F_4}
\end{figure}

\begin{figure}[h!]
   \labellist
    \small\hair 5pt
     \pinlabel $\lambda_{1}$ [b] at 3 174
     \pinlabel $\lambda_{2}$ [b] at 98 174
     \pinlabel $\lambda_{1}+\lambda_{2}$ [l] at 50 131
     \pinlabel $2\lambda_{1}+\lambda_{2}$ [l] at 50 88
     \pinlabel $3\lambda_{1}+\lambda_{2}$ [l] at 50 45
     \pinlabel $3\lambda_{1}+2\lambda_{2}$ [l] at 50 1
   \endlabellist

   \centering
   \includegraphics[width=3cm]{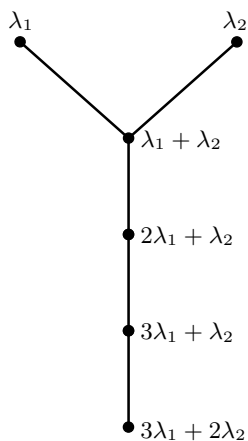}
   \caption[The root system $G_2$]{A graphical depiction of the positive roots in $G_{2}$.}
   \label{fig:G_2}
\end{figure}

\bibliography{bibliography}{}
\bibliographystyle{alpha}

\end{document}